\PassOptionsToPackage{unicode}{hyperref}
\PassOptionsToPackage{hyphens}{url}
\documentclass[
  11pt,
  letterpaper]{siamart251216}
\usepackage{xcolor}
\usepackage[margin=1in]{geometry}
\usepackage{amsmath,amssymb}
\usepackage{iftex}
\ifPDFTeX
  \usepackage[T1]{fontenc}
  \usepackage[utf8]{inputenc}
  \usepackage{textcomp} % provide euro and other symbols
\else % if luatex or xetex
  \usepackage{unicode-math} % this also loads fontspec
  \defaultfontfeatures{Scale=MatchLowercase}
  \defaultfontfeatures[\rmfamily]{Ligatures=TeX,Scale=1}
\fi
\usepackage{lmodern}
\ifPDFTeX\else
\fi
\IfFileExists{upquote.sty}{\usepackage{upquote}}{}
\IfFileExists{microtype.sty}{% use microtype if available
  \usepackage[]{microtype}
  \UseMicrotypeSet[protrusion]{basicmath} % disable protrusion for tt fonts
}{}
\makeatletter
\@ifundefined{KOMAClassName}{% if non-KOMA class
  \IfFileExists{parskip.sty}{%
    \usepackage{parskip}
  }{% else
    \setlength{\parindent}{0pt}
    \setlength{\parskip}{6pt plus 2pt minus 1pt}}
}{% if KOMA class
  \KOMAoptions{parskip=half}}
\makeatother
\makeatletter
\ifx\paragraph\undefined\else
  \let\oldparagraph\paragraph
  \renewcommand{\paragraph}{
    \@ifstar
      \xxxParagraphStar
      \xxxParagraphNoStar
  }
  \newcommand{\xxxParagraphStar}[1]{\oldparagraph*{#1}\mbox{}}
  \newcommand{\xxxParagraphNoStar}[1]{\oldparagraph{#1}\mbox{}}
\fi
\ifx\subparagraph\undefined\else
  \let\oldsubparagraph\subparagraph
  \renewcommand{\subparagraph}{
    \@ifstar
      \xxxSubParagraphStar
      \xxxSubParagraphNoStar
  }
  \newcommand{\xxxSubParagraphStar}[1]{\oldsubparagraph*{#1}\mbox{}}
  \newcommand{\xxxSubParagraphNoStar}[1]{\oldsubparagraph{#1}\mbox{}}
\fi
\makeatother

\usepackage{longtable,booktabs,array}
\usepackage{calc} % for calculating minipage widths
\usepackage{etoolbox}
\makeatletter
\patchcmd\longtable{\par}{\if@noskipsec\mbox{}\fi\par}{}{}
\makeatother
\IfFileExists{footnotehyper.sty}{\usepackage{footnotehyper}}{\usepackage{footnote}}
\makesavenoteenv{longtable}
\usepackage{graphicx}
\makeatletter
\newsavebox\pandoc@box
\newcommand*\pandocbounded[1]{% scales image to fit in text height/width
  \sbox\pandoc@box{#1}%
  \Gscale@div\@tempa{\textheight}{\dimexpr\ht\pandoc@box+\dp\pandoc@box\relax}%
  \Gscale@div\@tempb{\linewidth}{\wd\pandoc@box}%
  \ifdim\@tempb\p@<\@tempa\p@\let\@tempa\@tempb\fi% select the smaller of both
  \ifdim\@tempa\p@<\p@\scalebox{\@tempa}{\usebox\pandoc@box}%
  \else\usebox{\pandoc@box}%
  \fi%
}
\def\fps@figure{htbp}
\makeatother

\providecommand{\tightlist}{%
  \setlength{\itemsep}{0pt}\setlength{\parskip}{0pt}}

\usepackage[numbers]{natbib}
\usepackage{amsmath}
\usepackage{amssymb}
\usepackage{mathtools}
\usepackage{enumitem}

\usepackage{booktabs}
\usepackage{longtable}
\usepackage{array}
\usepackage{makecell}
\usepackage{caption}
\usepackage{tikz}
\usetikzlibrary{shapes.geometric, arrows}
\usepackage[ruled,vlined,linesnumbered,algo2e]{algorithm2e}

\newsiamremark{remark}{Remark}

\usepackage{acro}
\acsetup{
  make-links = false,
  pages/display = first
}
\DeclareAcronym{SC}{
    short = {SC},
    long = {stochastic co-spectrality},
    tag = {abbrev},
}
\DeclareAcronym{SSI}{
    short = {SSI},
    long = {stochastic spectral identifiability},
    tag = {abbrev},
}
\DeclareAcronym{ROSA}{
    short = {ROSA},
    long = {Robust Order-aware Spectral Amplification},
    tag = {abbrev},
}
\DeclareAcronym{ROSA-S}{
    short = {ROSA-S},
    long = {\acl{ROSA}-single},
    tag = {abbrev},
}
\DeclareAcronym{ROSA-C}{
    short = {ROSA-C},
    long = {\acl{ROSA}-contrastive},
    tag = {abbrev},
}
\DeclareAcronym{PH}{
    short = {PH},
    long = {persistent homology},
    tag = {abbrev},
}
\DeclareAcronym{SNR}{
    short = {SNR},
    long = {signal to noise ratio},
    tag = {abbrev},
}
\DeclareAcronym{IS2}{
    short = {IS2},
    long = {inverse coefficient of variation stability score},
    tag = {abbrev},
}
\DeclareAcronym{CTPL}{
    short = {CTPL},
    long = {calibrated tempered power-law},
    tag = {abbrev},
}
\DeclareAcronym{JCP}{
    short = {JCP},
    long = {Jacobian-coherence priority},
    tag = {abbrev},
}
\DeclareAcronym{MST}{
    short = {MST},
    long = {minimum spanning tree},
    tag = {abbrev},
}
\DeclareAcronym{MSTA}{
    short = {MSTA},
    long = {augmented minimum spanning tree},
    tag = {abbrev},
}
\DeclareAcronym{BLAS}{
    short = {BLAS},
    long = {Basic Linear Algebra Subprograms},
    tag = {abbrev},
}
\DeclareAcronym{SIMD}{
    short = {SIMD},
    long = {single instruction, multiple data},
    tag = {abbrev},
}
\DeclareAcronym{ROI}{
    short = {ROI},
    long = {region of interest},
    tag = {abbrev},
}
\DeclareAcronym{PDB}{
    short = {PDB},
    long = {Protein Data Bank},
    tag = {abbrev},
}

\makeatletter
\@ifpackageloaded{caption}{}{\usepackage{caption}}
\AtBeginDocument{%
\ifdefined\contentsname
  \renewcommand*\contentsname{Table of contents}
\else
  \newcommand\contentsname{Table of contents}
\fi
\ifdefined\listfigurename
  \renewcommand*\listfigurename{List of Figures}
\else
  \newcommand\listfigurename{List of Figures}
\fi
\ifdefined\listtablename
  \renewcommand*\listtablename{List of Tables}
\else
  \newcommand\listtablename{List of Tables}
\fi
\ifdefined\figurename
  \renewcommand*\figurename{Figure}
\else
  \newcommand\figurename{Figure}
\fi
\ifdefined\tablename
  \renewcommand*\tablename{Table}
\else
  \newcommand\tablename{Table}
\fi
}
\@ifpackageloaded{float}{}{\usepackage{float}}
\floatstyle{ruled}
\@ifundefined{c@chapter}{\newfloat{codelisting}{h}{lop}}{\newfloat{codelisting}{h}{lop}[chapter]}
\floatname{codelisting}{Listing}

\makeatother
\makeatletter
\@ifpackageloaded{caption}{}{\usepackage{caption}}
\@ifpackageloaded{subcaption}{}{\usepackage{subcaption}}
\makeatother
\usepackage{bookmark}
\IfFileExists{xurl.sty}{\usepackage{xurl}}{} % add URL line breaks if available
\hypersetup{
  pdftitle={ROSA: Metric Amplification on Noisy Graphs with Theoretical Guarantees for Amplified Spectral Distances},
  pdfauthor={Ben Cardoen; Fabian Spill},
  hidelinks,
  pdfcreator={LaTeX via pandoc}}

\title{ROSA: Metric Amplification on Noisy Graphs with Theoretical
Guarantees for Amplified Spectral Distances}
\author{Ben Cardoen \and Fabian Spill}
\date{2026-07-30}
\begin{document}
\maketitle
\begin{abstract}
In many applications, graphs are observed repeatedly, where the task is
to observe and track weak localized structural perturbations under
noise, such as vertex-coordinate displacements or edge-weight/attribute
changes on a fixed vertex set. One-shot graph distances can miss these
localized perturbations or fluctuate too strongly to support reliable
monitoring, especially in large graphs where signal dilutes with scale.
We propose \ac{ROSA}, a distance-amplification operator that integrates
metric evaluations along an order-aware edge-removal filtration. We
prove that \ac{ROSA} cannot reduce the included base distance and
conditionally increases it whenever a filtered step exposes additional
signal; separately, we prove that it can strictly improve an \ac{IS2},
defined as the mean of a noisy graph distance divided by its standard
deviation, under an explicit sufficient condition that can be checked
empirically. Experiments on synthetic graphs with localized edits under
three noise models, varying graph densities, and three algorithmically
selected edit sites show that \ac{ROSA} can often approximately double
the \ac{IS2} of the base spectral distance. We show where \ac{ROSA}
works and stops working on real-world use cases including mitochondrial
networks, fMRI correlation matrices, tissue networks, protein conformer
graphs, and retinal vasculature multiplex graphs. The empirical results
across graph, operator, and noise settings are consistent with the
theoretical amplification conditions and boundary constructions.
Evaluated with empirically estimated operative quantities, the bound
agrees with the observed gain in direction and approximate scale.
\end{abstract}



\section{Introduction}\label{sec-intro-rosa}

Global graph metrics are widely used to compare networks and quantify
differences \cite{Chung1997,Spielman2007,SpielmanTeng2007}, supporting
tasks from community detection \cite{VonLuxburg2007} to graph similarity
and anomaly detection \cite{Koutra2016DeltaCon,Tantardini2020}. In many
practical workflows, however, the aim is not just to distinguish two
networks once, but to observe and track weak localized structural
perturbations under repeated noisy observation. In this paper, such
perturbations are typically vertex-coordinate displacements or
edge-weight/attribute changes on a fixed vertex set, rather than
unrestricted graph-edit-distance operations. This setting arises in
dynamic monitoring, repeated sampling, and stochastic perturbation,
where the relevant signal may be small relative to both graph size and
measurement noise. Direct one-shot graph distances can be sensitive to
two related fragilities in this regime. Classical perturbation bounds
show that eigenvalue displacement under structured perturbations can be
tight \cite{StewartSun1990,Bhatia1997}, while even modest localized
perturbations can yield negligible global spectral change when the
perturbation is diluted across many eigenvalues. Under
moderate-to-heavy-tailed noise, this signal dilution can become severe
enough in practice that distances between genuinely distinct graphs are
hard to separate from noise. The practical need is therefore not only a
graph distance, but a way to amplify weak localized signal without
changing the underlying comparison target.

Several existing approaches address adjacent parts of this problem.
Graph kernels and spectral descriptors compare graphs through
eigenvalues, heat kernels, or diffusion summaries
\cite{Vishwanathan2010,Kondor2016,Tsitsulin2018NetLSD}. These methods
provide useful comparison summaries, but they are not designed
specifically to amplify weak localized signal under noise. Metric-based
methods such as DeltaCon \cite{Koutra2016DeltaCon} provide principled
baselines but remain sensitive to noise. Persistent-homology methods
track topological events across filtrations and enjoy strong stability
guarantees \cite{CohenSteiner2007}, while persistent Laplacians enrich
filtered complexes with spectral information
\cite{MemoliWanWang2022PersistentLaplacians,GulenMemoliWanWang2023PersistentLaplacianMaps}.
Spectral sparsification preserves spectra under graph reduction
\cite{spielman2011spectral}, while dynamic spectral monitoring methods
track evolving graphs over time \cite{Aleardi2018SpectralDistortion}.
Recent work also studies online network change-point localization
directly from sequential graph observations
\cite{YuPadillaWangRinaldo2024}. This downstream monitoring task is
complementary to our setting rather than a detector-level benchmark
considered here. Geometric graph distances such as fGOT provide
alternative transport-based comparison frameworks
\cite{Maretic2021fGOT}. These methods are valuable, but they do not
provide an operator whose explicit purpose is to amplify an existing
base distance under noise.

Here we propose an operator that acts on a chosen base graph metric by
integrating it along an order-aware edge-removal filtration. We make
five contributions:

\begin{itemize}
\tightlist
\item
  introduce \ac{ROSA}, an operator on graph distances that amplifies a
  base metric by integrating it along an order-aware edge-removal
  filtration.
\item
  prove a conditional amplification bound for spectral distances: when
  cumulative mean growth dominates variance growth along the filtration,
  \ac{ROSA} improves the inverse coefficient-of-variation stability
  score of the base metric.
\item
  identify structural boundary regimes for this amplification mechanism,
  including an explicit weighted-star instantiation, a sharp
  obstruction-by-symmetry case, and dense graph settings in which local
  filtration structure collapses.
\item
  validate \ac{ROSA} empirically across synthetic graph families,
  multiple noise models, graph densities, and alternative operators,
  showing broad but non-universal amplification.
\item
  show that the same mechanism transfers to real graph geometries,
  including controlled biological trajectories and a fixed-topology
  protein conformer benchmark, while also clarifying where the gain
  depends on noise scale, graph structure, and filtration design.
\end{itemize}

Because \ac{ROSA} acts on a base distance, it can in principle be
inserted into pipelines that rely on graph dissimilarity measures,
including graph kernel classification \cite{Vishwanathan2010}, graph
generative model evaluation via spectral MMD
\cite{You2018GraphRNN,Liao2019GRAN}, and dynamic network monitoring.
This paper studies amplification itself: when and why an order-aware
filtration can strengthen a base graph distance under noise, when it
fails, and how these regimes appear in synthetic and real graph
geometries. We combine conditional theory, explicit boundary
constructions, and controlled experiments on synthetic, biological, and
connectome-derived graphs. Notation used throughout is summarized in
Appendix Table \ref{tbl:notation-summary}. All code, data, and
reproduction scripts are publicly available in the paper
repository.\footnote{\url{https://github.com/systems-mechanobiology/ROSA}.}

\section{Methods}\label{sec-methods}

\subsection{ROSA-C Amplification}\label{sec-filtration}

Given a base graph distance
\(d: \mathcal{G} \times \mathcal{G} \to \mathbb{R}_{\geq 0}\) and an
edge-ranking heuristic that assigns a score \(f(e,X)\) to each edge
\(e \in E_X\) of a graph \(X\), \ac{ROSA} acts as an operator on
distances, removing edges according to a precomputed graph-wise ordering
and accumulating base metric evaluations along a filtration trajectory.
For two graphs \(G=(V,E_G)\) and \(H=(V,E_H)\) on a common vertex set,
let \(\pi^G=(e_G^{(1)},\dots,e_G^{(|E_G|)})\) and
\(\pi^H=(e_H^{(1)},\dots,e_H^{(|E_H|)})\) be the edge orderings obtained
by sorting the scores \(f(e,G)\) and \(f(e,H)\) in decreasing order,
with any ties resolved by the implementation's graph-wise edge
enumeration and sort behavior. Fix an integer filtration depth \(m\)
with \(0 \le m \le \min(|E_G|,|E_H|)\). For \(k=0,\dots,m\), define
\(G_k=(V,E_G \setminus \{e_G^{(1)},\dots,e_G^{(k)}\})\) and
\(H_k=(V,E_H \setminus \{e_H^{(1)},\dots,e_H^{(k)}\})\), with \(G_0=G\)
and \(H_0=H\). The induced trajectory is the ordered vector
\(\mathcal{T}_C(G,H) = (d(G_k,H_k))_{k=0}^{m}\), and the amplified
distance is \(d^{\mathrm{ROSA}}_f(G,H) := \sum_{k=0}^{m} d(G_k, H_k)\).
In this work, we instantiate \(d\) with the Wasserstein-2 spectral
distance on combinatorial Laplacian eigenvalues
\cite{santambrogio2015optimal}. Given eigenvalue multisets
\(\boldsymbol{\lambda} = \{\lambda_1,\dots,\lambda_n\}\) and
\(\boldsymbol{\lambda}' = \{\lambda'_1,\dots,\lambda'_n\}\) of graphs
\(G\) and \(G'\) on the same vertex set, define empirical spectral
measures \(\mu_G = \frac{1}{n}\sum_i \delta_{\lambda_i}\) and
\(\mu_{G'} = \frac{1}{n}\sum_j \delta_{\lambda'_j}\). The spectral
distance is then \(d_{\mathrm{SD}}(G,G') = W_2(\mu_G,\mu_{G'})\), where
\(W_2\) is the one-dimensional Wasserstein-2 distance; for empirical
measures on \(\mathbb{R}\) with matched mass,
\(W_2^2(\mu_G,\mu_{G'}) = \frac{1}{n}\sum_{i=1}^n (\lambda_{(i)}-\lambda'_{(i)})^2\),
where \(\lambda_{(i)}\) and \(\lambda'_{(i)}\) are the sorted
eigenvalues \cite{santambrogio2015optimal}. The framework is compatible,
by construction, with any base distance or dissimilarity that is
sensitive to edge perturbations (e.g., NetLSD
\cite{Tsitsulin2018NetLSD}, DeltaCon \cite{Koutra2016DeltaCon}, or
distances derived from graph kernels \cite{Vishwanathan2010}). To
compare two graphs \(G\) and \(H\), \ac{ROSA} filters both in lockstep,
computing the base distance at each step. Each graph receives its own
graph-wise edge ordering: \(\pi^G\) is induced by \(f(\cdot,G)\) and
\(\pi^H\) by \(f(\cdot,H)\). The orderings are not shared, coupled, or
conditioned on each other; the heuristic \(f\) is applied to each graph
separately (Algorithm \ref{alg:rosa-c}). The only alignment is in the
step index \(k\), not in the ordering itself. In implementations, the
depth may be specified by an integer retained-edge delimiter \(q\) with
\(0 \le q \le \min(|E_G|,|E_H|)\), giving \(m = \min(|E_G|, |E_H|) - q\)
(Algorithm \ref{alg:rosa-c}); for equal-edge comparisons this leaves
\(q\) edges in each graph after the final removal step. For the metric
theorem below, \(m\) rather than a global \(q\) is the fixed parameter
across the graph class being compared. For the spectral \(W_2\) distance
used here, we choose the combinatorial Laplacian. In contrast, the
normalized Laplacian adds graph-specific, step-dependent degree
rescaling, so it would require separate justification in a cumulative
filtration comparison. The metric theorem applies when the chosen base
distance is a genuine metric; the spectral instantiation used in the
experiments is a spectral dissimilarity on graphs and can be a
pseudometric when non-identical graphs are Laplacian-cospectral.

\subsection{\texorpdfstring{\ac{JCP}
ordering}{ ordering}}\label{sec-jacobian}

For each edge \(e=(i,j)\), the local sensitivity of the combinatorial
Laplacian \(L=D-A\) to an infinitesimal change in edge weight is
captured by the Jacobian
\(B_{ij}=(\mathbf{e}_i-\mathbf{e}_j)(\mathbf{e}_i-\mathbf{e}_j)^{\top}\),
where \(\mathbf{e}_i\) is the \(i\)th standard basis vector. This is a
rank-one update describing how \(L\) changes with \(w_{ij}\). Let
\(N_k(i)\) denote the vertices within at most \(k\) hops of \(i\) and
set \(U_e^{(k)}=N_k(i)\cup N_k(j)\). To quantify how strongly this
virtual perturbation aligns with the existing local spectral structure,
define \[
\rho_e^{(k)}=\frac{\langle B_{ij}[U_e^{(k)}],\,L[U_e^{(k)}]\rangle_F}
{\|B_{ij}[U_e^{(k)}]\|_F\,\|L[U_e^{(k)}]\|_F}\in[-1,1],
\label{eq:jcp-score}
\] where \(\langle\cdot,\cdot\rangle_F\) and \(\|\cdot\|_F\) denote the
Frobenius inner product and norm, respectively, and brackets denote the
principal submatrix restricted to the indicated neighbourhood. The
implementation exposes this neighbourhood size as the \texttt{radius}
argument; the default \ac{JCP} ordering uses \(U_e=U_e^{(1)}\) and
\(\rho_e=\rho_e^{(1)}\), and we write \(\mathrm{JCP}_k\) for radius
\(k\). Edges are sorted by decreasing \(\rho_e\), matching the
implementation's removal order, so high-coherence edges are removed
earlier in the filtration. \ac{JCP} is a spectrum-aware and
computationally efficient heuristic for ordering edges: it is directly
coupled to local spectral structure, uses only local neighbourhood
computations, and is designed to respond to localized changes. Unless
stated otherwise, the experiments use the radius-one definition. The
effect of \(k>1\) using \(\mathrm{JCP}_k\) is quantified in Figure
\ref{fig:retina-multiplex-jcp}. For a fixed edge with nonnegative
weights, increasing \(k\) leaves the numerator unchanged and cannot
decrease the denominator in Equation \ref{eq:jcp-score}; the raw score
is therefore non-increasing. This does not imply that the edge ranking
is fixed, because different edges can dilute at different rates, as
highlighted by the multiplex-radius experiment.

\subsection{IS2 as the Stability Criterion}\label{sec-is2}

\begin{definition}[Metric IS2]\label{def:metric-snr}
Let $d(G, H)$ be a distance metric between two graphs subject to a joint noise process $\eta=(\eta_G,\eta_H)$.
We define the inverse coefficient of variation stability score (\ac{IS2}) of the metric as
$$
\mathrm{IS2}_d = \frac{\mathbb{E}_\eta[d(G^{\eta_G}, H^{\eta_H})]}{\sqrt{\mathrm{Var}_\eta(d(G^{\eta_G}, H^{\eta_H}))}},
\label{eq:is2-definition}
$$
where $G^{\eta_G}, H^{\eta_H}$ denote the noisy graphs.
\end{definition}

This inverse-CV-style score quantifies the stability of the observed
graph-comparison signal relative to measurement uncertainty. We use it
as a stability or observability index rather than as a classical
engineering \ac{SNR}, in a spirit closer to normalized uncertainty
quantification than to power-based signal-processing conventions
\cite{WithersNadarajah2012,BrownGoodwinSorger2001}.

For a base metric \(d\) and its amplified counterpart
\(d_f^{\mathrm{ROSA}} = \mathcal{A}_f(d)\), we define the amplification
ratio \(R := \mathrm{IS2}_{\mathrm{ROSA}} / \mathrm{IS2}_d\). Thus
\(R>1\) means that amplification improves the stability score of the
base metric.

\subsection{Stochastic ROSA}\label{sec-stochastic-rosa}

In real world conditions it is likely ordering-stability issues emerge
when ranking edges: if two edges receive sufficiently similar heuristic
scores, noise can make their relative order unstable. This matters
because \ac{ROSA} compares graphs along their independently induced
filtrations, so local ordering uncertainty can propagate into the
accumulated trajectory; see Section \ref{sec-theory-rosa} for the role
of ordering stability in the theoretical assumptions and Section
\ref{sec-results-stochastic} for the empirical stochastic analysis.
Under a noisy edge-score model \(S'(e) = S(e) + \eta(e)\) with
i.i.d.~symmetric score noise, the probability of a pairwise order flip
between edges with score gap \(g > 0\) is
\(\mathbb{P}_{\text{flip}} = 1 - F_Z(g)\), where \(Z = \eta_x - \eta_y\)
is the induced symmetric difference variable. This probability is
strictly decreasing in \(g\): edges with similar scores are most
vulnerable to reordering. To make the filtration robust against such
noise-induced local reorderings, we employ a stochastic filtration that
randomizes adjacent ranks in the sorted edge list. In practice, the
adjacent-swap probability \(p_{\mathrm{swap}}\) is used as a tunable
surrogate for local ordering instability rather than being estimated
directly from \(\mathbb{P}_{\text{flip}}\). Given \(B\) randomized
orderings \(\widetilde{\Pi}^{(1)},\dots,\widetilde{\Pi}^{(B)}\), the
implementation applies independent fixed-parity adjacent swaps to
positions \((1,2),(3,4),\dots\) in each graph's precomputed ordering. It
then reports Monte Carlo means, standard deviations, and empirical
\([0.025,0.975]\) quantiles of the scalar \ac{ROSA} sums and violation
rates; stored trajectories are diagnostic outputs rather than the
primary aggregate. This nonparametrically probes a restricted
distribution of local rank permutations without requiring explicit
knowledge of the noise parameters. Stochastic \ac{ROSA} marginalizes the
\ac{ROSA} functional over a restricted class of local rank
perturbations: adjacent-swap perturbations of the deterministic edge
ordering. It approximates
\(\Phi_{\mathrm{rob}}(G,H) := \mathbb{E}_{\Pi \sim \mathcal{D}_{\text{local}}}[\Phi(G,H;\Pi)]\),
where \(\Phi\) is the \ac{ROSA} trajectory integral and
\(\mathcal{D}_{\text{local}}\) is the distribution induced by pairwise
adjacent swaps. This should be interpreted as: under benign
ordering-noise assumptions, how stable the amplified graph-comparison
signal remains once ordering uncertainty is accounted for. Stochastic
\ac{ROSA} does not estimate the true permutation distribution over all
possible edge orderings, and pathological global reorderings, where the
entire ranking is scrambled, are explicitly out of scope. The algorithm
only considers non-overlapping adjacent swaps, not arbitrary reorderings
with large index jumps, because the latter would rapidly induce
prohibitive computational costs.

\subsection{Complexity and Scalability}\label{sec-complexity}

The cost of \ac{ROSA} is dominated by repeated base-metric evaluations
along the filtration. For the \(W_2\) instantiation, the expensive step
is the eigendecomposition of the combinatorial Laplacian at successive
filtration states. This makes scalability depend on three complementary
levers: reducing the number of evaluated filtration steps, reducing the
cost per spectral evaluation, and reducing wall-clock time through
parallel execution. Table \ref{tbl:complexity-options} summarizes the
main options used in this paper; Section \ref{sec-results-scaling}
reports their empirical scaling behavior.

\begin{table}[t]
\centering
\caption{Complexity and scalability options for ROSA with spectral $W_2$. Here $n$ is the number of vertices, $m$ the filtration depth, $s$ the stride, $L$ the retained filtration length, $S$ the checkpoint stride for the incremental eigensolver, $r \in (0,1]$ the retained spectral fraction, and $P$ the number of workers.}
\label{tbl:complexity-options}
\small
\renewcommand{\arraystretch}{1.12}
\begin{tabular}{@{}>{\raggedright\arraybackslash}p{0.18\textwidth} >{\raggedright\arraybackslash}p{0.34\textwidth} >{\raggedright\arraybackslash}p{0.42\textwidth}@{}}
\toprule
Variant & Cost / effect & Comment \\
\midrule
Full ROSA & $O(m \cdot n^3)$ & Full eigendecomposition at every filtration step \\
Strided ROSA & $O(\lceil m/s \rceil \cdot n^3)$ & Evaluate every $s$-th step only \\
Truncated filtration & $O(L \cdot n^3)$ & Keep early, high-information part of the trajectory \\
Incremental eigensolver & \makecell[l]{$O\!\left(n^3 + \lfloor L/S \rfloor n^3\right.$ \\ $\left.\quad + (L-\lfloor L/S \rfloor)n^2\right)$} & Approximate rank-1 secular updates between checkpoints; validate against full spectra \\
Top-$r$ fraction / sparse eigensolver & solver-dependent, below full spectrum when converged & Exploits spectral localization; runtime depends on tolerance, gaps, and backend \\
Parallel execution & ideal $\approx 1/P$ wall-clock reduction & Independent after masks/orderings are fixed; speedup is limited by BLAS/cache/memory contention \\
\bottomrule
\end{tabular}
\end{table}

A first strategy is striding, because a full \(m\)-step trajectory is
not always necessary. With stride \(s \geq 1\), the strided \ac{ROSA}
distance is
\(d^{\text{ROSA}}_{f,s}(G,H) := \sum_{j=0}^{\lfloor m/s \rfloor} d(G_{js}, H_{js})\),
where \(G_{js}, H_{js}\) are the filtration states after \(js\) edge
removals and \(s=1\) recovers ROSA-C. Empirically (Figure
\ref{fig:scaling}, panels C--D), much of the amplification signal is
concentrated in the early filtration steps, so moderate stride and
truncation can reduce cost while preserving \(R > 1\). Next, we can
exploit that for spectral-based distances such as \(W_2\), each edge
removal is a rank-1 perturbation of the Laplacian. We therefore
implement an incremental strategy: full eigendecompositions are
refreshed only every \(S\) steps, while intermediate steps use a secular
equation update \cite{gu1995divide}. For \(S>1\), intermediate spectra
are approximate because eigenvectors are refreshed only at checkpoints;
the benchmark records agreement of the integrated \ac{ROSA} value with
sequential eigendecomposition in the tested regime. For a retained
filtration of length \(L\), a coarse runtime model is
\(T(n, L, S) = T_{\mathrm{setup}}(n) + \lfloor L/S \rfloor \cdot T_{\mathrm{eigen}}(n) + (L - \lfloor L/S \rfloor) \cdot T_{\mathrm{secular}}(n)\),
with \(T_{\mathrm{setup}} = O(n^3)\), \(T_{\mathrm{eigen}} = O(n^3)\),
and \(T_{\mathrm{secular}} = O(n^2)\).

The third lever is wall-clock reduction through parallelism: once the
edge orderings and edge masks are fixed, selected filtration states or
contiguous step blocks can be reconstructed independently and their
distance contributions integrated. In practice, dense
\acs{BLAS}-accelerated eigensolvers quickly saturate because each worker
may also spawn internal \acs{BLAS} threads. Even when those threads are
controlled, cache, \acs{SIMD}, and memory-bandwidth contention can limit
speedup. Process-level parallelism with separate worker processes and
controlled \acs{BLAS} pools is therefore the more reliable route, but it
still remains hardware- and workload-dependent. Finally, if spectral
localization is strong and known, sparse iterative top-fraction
eigensolvers provide an additional approximation route by reducing the
number of eigenvalues computed per step; convergence and runtime depend
on the tolerance, spectral gaps, retained fraction, and backend. All
these tradeoffs are evaluated empirically in Figure \ref{fig:scaling}.

\subsection{Theoretical Properties of ROSA}\label{sec-theory-rosa}

\subsubsection{Perturbation Bounds and Filtration
Intuition}\label{sec-regimes}

Classical perturbation bounds provide complementary intuition for when a
filtration may preserve local spectral structure, although neither bound
alone guarantees amplification. For a Hermitian perturbation
\(L' = L + E\), the Weyl inequalities \cite{horn2012matrix} bound each
ordered eigenvalue displacement by the operator norm:
\(|\lambda_i(L') - \lambda_i(L)| \le \|E\|_2\). This worst-case bound is
structure-blind: it depends only on the largest singular value of the
perturbation, regardless of how \(E\) is distributed across the graph.
When \(\|E\|_2\) is comparable to local spectral gaps or to the score
margins that determine the edge ordering, eigenspaces and induced local
edge rankings can become unstable \cite{DavisKahan1970}, which is
unfavorable for filtration-based amplification.

By contrast, the Hoffman--Wielandt inequality controls aggregate squared
eigenvalue displacement through the Frobenius norm,
\(\sum_i |\lambda_i(L')-\lambda_i(L)|^2 \le \|E\|_F^2\)
\cite{HoffmanWielandt1953}. This perspective can reflect perturbation
energy distributed across many localized components rather than only one
extremal direction, but it still does not imply amplification.
Filtration can benefit from those components only when they remain
locally resolvable and the edge-ranking heuristic continues to induce a
meaningful ordering. The \ac{JCP} heuristic (Section \ref{sec-jacobian})
is intended to bias the filtration toward removals whose perturbation
directions align with local spectral structure; a large
Jacobian-coherence score \(\rho_e\) indicates that the edge update
\(B_{ij}\) is well matched to the local Laplacian block \(L[U_e]\). This
interpretation remains heuristic outside special cases, but it is
consistent with the empirical boundary in Section
\ref{sec-results-fmri-boundary}: amplification is strongest when local
structure is preserved and can collapse on dense near-uniform graphs
where edge neighbourhoods cease to be local. A detailed spectral
analysis of structured perturbations under geometric vertex noise is
given in \cite{cardoen2026spectral}.

\subsubsection{IS2 Amplification Condition}\label{sec-snr-amplification}

For permutations \(\pi, \rho\) of an edge index set \(E\), let
\(d_K(\pi, \rho)\) denote the Kendall--tau distance (number of
discordant pairs). Given an edge-ranking heuristic \(f\) that induces a
permutation \(\pi^X\) for each graph \(X\), we define the signal-induced
filtration-path divergence
\(\Delta_f^{\mathrm{signal}} := d_K(\pi^G, \pi^H)\) and the
noise-induced filtration-path divergence
\(\Delta_f^{\mathrm{noise}} := \max\bigl\{\mathbb{E}_\eta[d_K(\pi^{G^{\eta_G}}, \pi^G)],\; \mathbb{E}_\eta[d_K(\pi^{H^{\eta_H}}, \pi^H)]\bigr\}\).

These definitions make the sensitivity condition precise: the heuristic
is effective when
\(\Delta_f^{\mathrm{signal}} > \Delta_f^{\mathrm{noise}}\), i.e., the
structural edit displaces the edge ordering more than noise does. Using
the \ac{IS2} score from \eqref{eq:is2-definition} and the amplification
ratio \(R\), we now state the operative amplification proposition.

\begin{proposition}[Operative Conditions for ROSA IS2 Amplification]\label{thm:snr-guarantee}
Let $G$ and $H$ be weighted graphs on the same vertex set and fixed potential edge set $E$, and let $d$ be a measurable nonnegative graph dissimilarity.
Let $m$ be a fixed filtration depth with $0 \le m \le |E|$, and let $f$ be an edge-ranking heuristic that,
for any graph $X$, induces a permutation $\pi^X$ of $E$.
For a graph $X$ and $k\in\{0,\dots,m\}$, let $X_k$ denote the subgraph obtained from $X$ by removing
the first $k$ edges in the ordering $\pi^X$.
Define the ROSA-C dissimilarity $d^{\mathrm{ROSA}}_f(G,H) := \sum_{k=0}^{m} d(G_k,H_k)$.
Let $\eta=(\eta_G,\eta_H)$ be a joint noise process on graph pairs, and write $G^{\eta_G}$ and $H^{\eta_H}$
for the noisy graphs; in the experiments the two components are sampled independently unless stated otherwise.
The filtration is applied after noise is sampled: for each realization, $(G^{\eta_G})_k$ and $(H^{\eta_H})_k$ are obtained by ranking and removing edges from $G^{\eta_G}$ and $H^{\eta_H}$ respectively.
For any graph dissimilarity $\delta$, write $\mu_\delta := \mathbb{E}_\eta[\delta(G^{\eta_G},H^{\eta_H})]$,
$\sigma_\delta := \sqrt{\mathrm{Var}_\eta(\delta(G^{\eta_G},H^{\eta_H}))}$, and assume these quantities are finite with positive variance for $\delta \in \{d, d^{\mathrm{ROSA}}_f\}$.
Instantiating $\delta \in \{d,\, d^{\mathrm{ROSA}}_f\}$ gives the base-metric quantities
$(\mu_{SD}, \sigma_{SD}, \mathrm{IS2}_{SD})$ and ROSA quantities
$(\mu_{\mathrm{ROSA}}, \sigma_{\mathrm{ROSA}}, \mathrm{IS2}_{\mathrm{ROSA}})$ respectively.
Assume the following conditions hold for constants $\mu_0>0$, $\beta>0$, and $\gamma>0$:

\begin{enumerate}[label=\textup{(A\arabic*)},ref=A\arabic*,leftmargin=*,itemsep=6pt]
\item\label{a:signal} \textbf{Non-degenerate base signal.}
$\mu_{SD} \ge \mu_0 > 0$ and $\sigma_{SD}>0$.

\item\label{a:variance} \textbf{Noise variance control across the filtration.}
$\sigma_{\mathrm{ROSA}} \;\le\; \beta\,\sqrt{m+1}\,\sigma_{SD}.$

If the $m+1$ step distances were independent and had common variance $\sigma_{SD}^2$, then $\sigma_{\mathrm{ROSA}}^2 = (m+1)\sigma_{SD}^2$, giving $\beta=1$.
The nested filtration structure induces positive correlation across steps, which increases variance: $\operatorname{Var}(\sum_k d_k) = \sum_k \operatorname{Var}(d_k) + 2\sum_{i<j}\operatorname{Cov}(d_i,d_j)$.
The constant $\beta$ absorbs this correlation; empirically $\beta \approx 4$ (Section \ref{sec-results-sensitivity}), reflecting moderate positive dependence from shared edge structure across consecutive filtration steps.

\item\label{a:filtration-amplifies} \textbf{Filtration amplifies expected cross-distance.}
\[
\sum_{k=1}^{m} \mathbb{E}_\eta\bigl[d((G^{\eta_G})_k,(H^{\eta_H})_k) - d(G^{\eta_G},H^{\eta_H})\bigr]
\;\ge\; (m+1)\,\gamma\,\mu_{SD}
\quad\text{for some }\gamma>0.
\]

(\ref{a:filtration-amplifies}) states that removing edges increases the expected distance between the noisy graphs on average.
This is a signed mean-shift condition: it concerns the cumulative expected distance trajectory directly rather than absolute ordering-displacement bounds.
It is empirically testable: in Section \ref{sec-results-sensitivity}, we test when $\mu_{\mathrm{ROSA}} > (m+1)\mu_{SD}$ holds and explicitly report the regimes where this sufficient condition is not necessary for $R>1$.
\end{enumerate}

The proposition is intentionally stated at the level of the operative mean and variance conditions.
In concrete instantiations, one may seek stronger structural conditions on the heuristic and noise process that imply (A3), but such conditions are not required by the proof itself.

Then the amplification ratio satisfies
\[
R=\frac{\mathrm{IS2}_{\mathrm{ROSA}}}{\mathrm{IS2}_{SD}}
\;\ge\;
\frac{\sqrt{m+1}(1+\gamma)}{\beta}.
\]
In particular, whenever $\sqrt{m+1}(1+\gamma) > \beta$,
we have $R>1$.
\end{proposition}

\begin{proof}
We relate the ROSA mean $\mu_{\mathrm{ROSA}}$ to the base mean $\mu_{SD}$, bound
$\sigma_{\mathrm{ROSA}}$ in terms of $\sigma_{SD}$, and combine these to obtain a lower
bound on the IS2 ratio.

\textbf{Step 1: Decomposition of the ROSA mean.}
By definition,
$\mu_{\mathrm{ROSA}}
= \mathbb{E}_\eta\left[\sum_{k=0}^{m} d((G^{\eta_G})_k,(H^{\eta_H})_k)\right]
= \sum_{k=0}^{m} \mathbb{E}_\eta\bigl[d((G^{\eta_G})_k,(H^{\eta_H})_k)\bigr]$.
The $k=0$ term equals $\mu_{SD}$. For each $k \ge 1$, add and subtract $d(G^{\eta_G},H^{\eta_H})$:
$\mathbb{E}_\eta\bigl[d((G^{\eta_G})_k,(H^{\eta_H})_k)\bigr]
= \mu_{SD}
+ \mathbb{E}_\eta\bigl[d((G^{\eta_G})_k,(H^{\eta_H})_k) - d(G^{\eta_G},H^{\eta_H})\bigr]$.
Summing over $k=1,\dots,m$ yields
\begin{equation}\label{eq:mu-rosa-decomp}
\mu_{\mathrm{ROSA}}
= (m+1)\,\mu_{SD}
+ \sum_{k=1}^{m} \mathbb{E}_\eta\bigl[d((G^{\eta_G})_k,(H^{\eta_H})_k) - d(G^{\eta_G},H^{\eta_H})\bigr].
\end{equation}

\textbf{Step 2: Bounding the mean shift.}
By assumption (A3), the remainder in \eqref{eq:mu-rosa-decomp} satisfies
$\sum_{k=1}^{m} \mathbb{E}_\eta\bigl[d((G^{\eta_G})_k,(H^{\eta_H})_k) - d(G^{\eta_G},H^{\eta_H})\bigr]
\;\ge\; (m+1)\,\gamma\,\mu_{SD}$.
Substituting into \eqref{eq:mu-rosa-decomp}:
$\mu_{\mathrm{ROSA}}
\;\ge\; (m+1)\,\mu_{SD} + (m+1)\,\gamma\,\mu_{SD}
= (m+1)(1+\gamma)\,\mu_{SD}$.

\textbf{Step 3: Variance bound and amplification ratio.}
By assumption (A2), $\sigma_{\mathrm{ROSA}} \le \beta\,\sqrt{m+1}\,\sigma_{SD}$.
Combining with the mean bound:
$R=\frac{\mathrm{IS2}_{\mathrm{ROSA}}}{\mathrm{IS2}_{SD}}
=
\frac{\mu_{\mathrm{ROSA}}/\sigma_{\mathrm{ROSA}}}{\mu_{SD}/\sigma_{SD}}
\;\ge\;
\frac{\sqrt{m+1}(1+\gamma)}{\beta}$.
\end{proof}

Proposition \ref{thm:snr-guarantee} isolates the two quantities that
determine amplification, namely the cumulative mean shift across the
filtration (A3) and the variance-growth factor (A2). For the
\(W_2\)+\acs{JCP}+\ac{CTPL} instantiation, these operative conditions
are verified empirically in Section \ref{sec-results-sensitivity}, where
they correctly predict the observed amplification regime. In general,
verifying these assumptions from graph structure alone is combinatorial:
it depends jointly on the specific weighted graph, the ordering induced
by the heuristic, and the noise process, so a uniform graph-family
theorem is not realistic. For that reason we give one explicit main-text
instantiation (Section Section~\ref{sec-results-sensitivity}) and treat
the remaining variance-growth term empirically. This also motivates
stochastic \ac{ROSA} (Section \ref{sec-stochastic-rosa}): by randomizing
over orderings, stochastic \ac{ROSA} hedges against ordering uncertainty
without requiring any single deterministic ordering to satisfy a strong
sufficient-condition argument.

For an explicit weighted-star instantiation, consider the five-vertex
star with one hub and four leaves at the corners of the unit square,
with spoke weights \(w_i = 1/\|v_0-v_i\|\). For this motif, the \ac{JCP}
score of spoke \(i\) is
\(\rho_i = \frac{4w_i + \sum_{j\neq i} w_j}{2\|L\|_F}\), so
\(\rho_i > \rho_j\) if and only if \(w_i > w_j\): the filtration order
is exactly the spoke-weight order, with larger scores removed earlier
under the implementation convention. This gives an explicit symbolic
check of the ordering mechanism, not a proof of the full \ac{IS2}
conditions. For a generic asymmetric hub perturbation, the deterministic
filtration trajectories diverge and the empirical hub-position scan
yields \(\gamma>0\) on sampled positive regions outside the spoke-length
equality loci; for the mirror-symmetric hub displacement, the weighted
graphs are permutation-similar and \(\gamma=0\); and for a first-removed
single-edge tree edit the post-removal trajectories coincide, giving
\(\gamma<0\). Under iid Gaussian edge noise below the deterministic
score-gap scale, the spoke ordering is stable with high probability,
which supports the local-ordering part of the construction. Stable
ordering alone does not prove nondegeneracy or the signed mean-shift
condition \textup{(A3)}; those remain operative checks, with the
positive regions evidenced numerically here. The remaining
variance-growth condition \textup{(A2)} is nontrivial even for the
weighted star because consecutive filtration steps are correlated; in
this paper it is therefore estimated empirically rather than closed
symbolically.

\begin{theorem}[Instantiation on weighted-star boundary]
For the five-vertex weighted star with \ac{JCP}-ordered spoke filtration:
\begin{enumerate}[label=\textup{(\roman*)},leftmargin=*,itemsep=4pt]
\item a generic asymmetric hub perturbation changes multiple spoke weights and yields divergent deterministic filtration trajectories;
\item a mirror-symmetric hub displacement induces permutation-similar weighted graphs and yields $\gamma=0$;
\item a first-removed single-edge tree edit confines the signal to one filtration step and yields $\gamma<0$.
\end{enumerate}
Thus the same motif exhibits observable divergence, obstruction, and dilution depending only on how the edit interacts with the graph symmetry and filtration order; the corresponding amplification region $\gamma>0$ is then mapped empirically in the appendix.
\end{theorem}

The appendix gives the full symbolic derivation and supporting evidence
for this motif: Theorem \ref{thm:sharp-boundary}, the setup and
\(\gamma\)-landscape figures (Figures \ref{fig:sharp-boundary-setup} and
\ref{fig:sharp-boundary-heatmap}), and the worked-example table
\ref{tab:jcp-worked-example}.

These operative conditions and boundary cases, combined with our
extensive results (Section \ref{sec-results}) suggest that the most
plausible regime where \ac{ROSA} works with \ac{JCP} and spectral
distance is one with meaningful local asymmetry, where the edge-ordering
heuristic induces a nontrivial filtration and localized perturbations
accumulate coherently. At the opposite extreme, highly symmetric or
near-uniform dense graphs provide little edge-order signal: for an
equal-weight complete graph, \ac{JCP} is exactly flat by symmetry, and
in dense near-uniform weighted graphs even small perturbations can leave
the filtration nearly indistinguishable from a rescaled one-shot
comparison. Our negative results on some graphs (Section
\ref{sec-results-fmri-boundary}) confirm that there exist real-world
regimes in which its operative conditions are not expected to hold. The
synthetic experiments locate the positive regime empirically, and
Appendix Table \ref{tab:jcp-worked-example} gives a small worked example
showing both symmetry-induced \ac{JCP} collapse and a sparse asymmetric
case where the ordering becomes informative.

A first failure mode occurs when \ac{JCP} ceases to be local. If an edge
\(e=(i,j)\) satisfies \(U_e = \{i,j\} \cup N(i) \cup N(j) = V\), then
the local Laplacian block in the \ac{JCP} denominator is the full graph
Laplacian, so the denominator is identical for all such edges.
Amplification may collapse on such substrates (see Figure
\ref{fig:fmri-boundary}). In that regime the \ac{JCP} score reduces to
\(\rho_e = \frac{L_{ii} + L_{jj} - 2L_{ij}}{2\|L\|_F} = \frac{d_i + d_j + 2w_{ij}}{2\|L\|_F}\),
and the heuristic no longer reflects edgewise local spectral variation.
Complete graphs satisfy this condition for every edge, and sufficiently
dense near-uniform graphs can approximate it closely, explaining why
\ac{JCP}-based amplification may collapse on such substrates.

A second limitation is that signal and noise orderings do not compose.
In that case there is no structure of the form
\(\pi_{\mathrm{noise}} \circ \pi_{\mathrm{signal}}\) for edge orderings:
the observed ordering
\(\pi^{\mathrm{obs}} = \operatorname{argsort}(x + s(x) + \eta(x))\) is a
nonlinear function of the joint signal-plus-noise process. The
sensitivity condition
\(\Delta_f^{\mathrm{signal}} > \Delta_f^{\mathrm{noise}}\) is therefore
testable empirically, by comparing the ordering displacement under
signal-only versus noise-only perturbations, but does not decompose the
observed ordering analytically.

\subsubsection{Non-Contraction and Conditional
Anti-Dilution}\label{sec-anti-dilution}

As graph size increases, the direct spectral effect of a bounded
localized edit may dilute, causing the one-shot base distance to
approach zero. No graph-independent claim can ensure that a particular
filtration avoids this behavior: the outcome depends on the graph, edit,
ordering, and, when present, noise realization. \ac{ROSA}-C nevertheless
has the following pointwise safeguard because its accumulated trajectory
includes the unfiltered comparison at \(k=0\).

\begin{proposition}[ROSA-C non-contraction and conditional anti-dilution]\label{thm:anti-dilution}
Let $G$ and $H$ be graphs on a common vertex set, let $d$ be any nonnegative graph dissimilarity, and let $f$ induce graph-wise filtrations $(G_k)_{k=0}^{m}$ and $(H_k)_{k=0}^{m}$ through a common depth $m$.
Then
\[
d_f^{\mathrm{ROSA}}(G,H)
= \sum_{k=0}^{m} d(G_k,H_k)
\ge d(G,H).
\]
Equality holds if and only if $d(G_k,H_k)=0$ for every $k=1,\dots,m$.
If $d(G,H)=0$ but $d(G_k,H_k)>0$ for at least one filtered step, then $d_f^{\mathrm{ROSA}}(G,H)>0$.
If $d(G,H)>0$ and some filtered step satisfies $d(G_k,H_k)\ge c\,d(G,H)$ for $c>0$, then
\[
\frac{d_f^{\mathrm{ROSA}}(G,H)}{d(G,H)} \ge 1+c.
\]
\end{proposition}

\begin{proof}
The $k=0$ term is $d(G_0,H_0)=d(G,H)$, and every remaining term is nonnegative.
Equality therefore holds exactly when every filtered-step term is zero.
The two conditional conclusions follow by retaining the stated witness term in the sum.
\end{proof}

The result is pointwise and therefore holds for every realization under
any noise process for which the filtration is defined; taking
expectations preserves the weak inequality whenever the expectations are
finite. It does not claim that every ordering breaks cospectrality, that
strict amplification occurs for every graph or edit, or that
\(d_f^{\mathrm{ROSA}}(G_n,H_n)\) remains bounded away from zero as graph
size grows. Both the base and accumulated distances may tend to zero;
the proposition states only that accumulation cannot reduce the included
base distance and gives exact conditions for strict or relative gain.
For the empirical \(W_2\) dissimilarity on combinatorial-Laplacian
spectra used here, non-identical cospectral graphs can have zero base
distance, and an adversarial or uninformative ordering may preserve zero
distance at every filtered step. Whether useful witness steps occur, and
whether their gain remains stable under noise, are graph- and
problem-dependent questions addressed empirically and through
Proposition \ref{thm:snr-guarantee}.

\subsubsection{Stability Under Perturbation}\label{sec-lipschitz}

The following proposition establishes a local Lipschitz bound for
\ac{ROSA}-C when each graph's filtration ordering is stable under its
perturbation.

\begin{proposition}[Local stability of ROSA-C under fixed filtration ordering]\label{thm:lipschitz}
Let $G,H,\tilde{G},\tilde{H}$ be weighted graphs on a common vertex set, where
$\tilde{G}$ and $\tilde{H}$ are perturbed versions of $G$ and $H$, respectively.
Fix a filtration depth $m$ such that the first $m$ edge removals are defined for all four graphs.
Assume edge identities are comparable by their vertex endpoints.
For each pair $(G,\tilde{G})$ and $(H,\tilde{H})$, assume that the first $m$ endpoint-labelled removals lie in that pair's common removable edge set.

Assume:
\begin{enumerate}
\item The base dissimilarity $d$ is jointly Lipschitz as a function of Laplacians:
there exists $C>0$ such that, for the chosen matrix norm,
\[
| d(L_1,L_2) - d(L_1',L_2') |
\le C( \|L_1 - L_1'\| + \|L_2 - L_2'\| ).
\]
\item The heuristic $f$ induces the same first $m$ endpoint-labelled edge-removal sequence on
$G$ and $\tilde{G}$, and the same first $m$ endpoint-labelled edge-removal sequence on
$H$ and $\tilde{H}$.
The $G$-side and $H$-side sequences need not be identical to each other.
The fixed-order assumption is only between each graph and its perturbation.
\item Along these matched filtrations, the stepwise perturbations remain
uniformly controlled: there exist bounds $\bar\varepsilon_G,\bar\varepsilon_H$
such that
\[
\|L_{G_k} - L_{\tilde{G}_k}\| \le \bar\varepsilon_G,
\qquad
\|L_{H_k} - L_{\tilde{H}_k}\| \le \bar\varepsilon_H
\]
for all $k=0,\dots,m$.
\end{enumerate}

Then $\bigl| d_f^{\mathrm{ROSA}}(G,H) - d_f^{\mathrm{ROSA}}(\tilde{G},\tilde{H}) \bigr|
\;\le\;
C\,(m+1)\,(\bar\varepsilon_G + \bar\varepsilon_H)$.
\end{proposition}

\begin{proof}
At each filtration step $k$, writing $d(G_k,H_k)$ for $d(L_{G_k},L_{H_k})$,
the joint Lipschitz property gives
$|d(G_k,H_k) - d(\tilde{G}_k,\tilde{H}_k)|
\le
C\bigl(\|L_{G_k}-L_{\tilde{G}_k}\| + \|L_{H_k}-L_{\tilde{H}_k}\|\bigr)$.
By assumption (3),
$|d(G_k,H_k) - d(\tilde{G}_k,\tilde{H}_k)|
\le
C(\bar\varepsilon_G + \bar\varepsilon_H)
\qquad\text{for all } k=0,\dots,m$.
Applying the scalar inequality $|\sum_{k=0}^{m} a_k| \le \sum_{k=0}^{m} |a_k|$ and summing over the filtration steps yields
$\bigl| d_f^{\mathrm{ROSA}}(G,H) - d_f^{\mathrm{ROSA}}(\tilde{G},\tilde{H}) \bigr|
\le
C\,(m+1)\,(\bar\varepsilon_G + \bar\varepsilon_H)$.
\end{proof}

The proposition is a local fixed-order stability statement rather than a
global robustness guarantee. Its force depends on maintaining both
ordering stability and bounded stepwise Laplacian perturbation along the
matched filtration. For weighted graphs, removing the same edge from two
nearby filtrations need not preserve the initial Laplacian difference
exactly, so the relevant quantity is the stepwise bound on
\(\|L_{G_k}-L_{\tilde{G}_k}\|\), not invariance of that difference. This
again connects to the structured regime picture: open local score
margins can help preserve the fixed-order hypothesis, and when such
margins collapse, stochastic \ac{ROSA} provides an empirical
ordering-sensitivity mechanism rather than a deterministic Lipschitz
guarantee.

\subsubsection{Metric Property of ROSA}\label{metric-property-of-rosa}

\begin{theorem}
Let $\mathcal{G}$ be a class of weighted undirected graphs, and assume that for every $X\in\mathcal{G}$ the first $m$ graph-wise filtered states $X_0,\dots,X_m$ lie in a common ambient class $\mathcal{H}$ on which $d$ is a metric.
Let $f$ assign to each graph $X$ a deterministic edge ordering $\pi^X$, with deterministic graph-wise tie-breaking independent of the comparison partner, and fix the same depth $m$ for all pairs.
Then, for such metric $d$, the corresponding \ac{ROSA}-C distance $d_f^{\mathrm{ROSA}}$ is a metric on $\mathcal{G}$.
Equivalently, a fixed retained-edge delimiter $q$ is covered only on graph classes with constant edge count.
\end{theorem}

\begin{proof}
Non-negativity follows directly from $d$.
Symmetry follows because each filtered state is graph-wise deterministic and independent of the comparison partner, while $d$ is symmetric.
If $d_f^{\mathrm{ROSA}}(G,H)=0$, then in particular $d(G_0,H_0)=0$, so $G=H$.
For any third graph $K$, the three states $G_k,K_k,H_k$ are therefore the same in the three pairwise \ac{ROSA} sums, and
$d(G_k,H_k)\le d(G_k,K_k)+d(K_k,H_k)$
by the triangle inequality for $d$.
Summing over $k=0,\dots,m$ gives
$d_f^{\mathrm{ROSA}}(G,H)\le d_f^{\mathrm{ROSA}}(G,K)+d_f^{\mathrm{ROSA}}(K,H)$.
\end{proof}

When \(d\) is the spectral \(W_2\) dissimilarity used in the
experiments, the same argument gives only a pseudometric on graphs,
because Laplacian-cospectral non-identical graphs can have zero base
distance.

\subsection{Experimental Design for Synthetic Graphs}\label{sec-setup}

\subsubsection{Synthetic Graph and Edit-Site
Selection}\label{synthetic-graph-and-edit-site-selection}

Synthetic experiments use an augmented minimum spanning tree (\ac{MSTA})
baseline, obtained by adding the five shortest non-tree edges to an
\ac{MST} in order to retain a sparse backbone while introducing local
cycles. This sparse-plus-cycles topology is inspired by biophysical
networks such as mitochondrial skeletons and tissue contact graphs,
which combine tree-like backbone connectivity with local clusters of
short-range edges \cite{wang2023mitotnt,etournay2016tissueminer}. For
each synthetic edit site, structural edits are applied by displacing one
selected vertex by \(\Delta x\) coordinate units along the \(x\)-axis,
leaving all other vertices and the edge set fixed and producing an
edited graph \(G'\) on the same vertex set. For synthetic experiments,
vertices are selected algorithmically at three edit sites (Figure
\ref{fig:mst-sites}): bridge-adjacent, selected by the largest
Fiedler-coordinate jump across incident edges and used as a proxy for
bridge or cut-edge proximity; leaf/endpoint, a degree-1 vertex; and
small-Fiedler-magnitude hub, selected among degree-\(\ge 3\) vertices by
the smallest absolute Fiedler coordinate. Within each synthetic block,
one representative vertex per selection rule is chosen once on that
block's reference graph and then held fixed over the corresponding
parameter sweep. Across the synthetic study, we vary one main
experimental axis at a time, such as edit site, noise model, density, or
operator, while keeping the remaining design choices fixed; exact
parameter settings are listed in Table \ref{tbl:experiment-config}. For
real biological graphs (Section \ref{sec-results-real-geometry}), edit
vertices are chosen manually to reflect biologically plausible local
changes on each substrate.

\subsubsection{Noise Models}\label{noise-models}

To probe sensitivity to the noise process, we test three qualitatively
different noise models, covering i.i.d.~edge noise, i.i.d.~vertex noise,
and heavy-tailed vertex noise. For the Gaussian models, the baseline
noise scale is set from the shortest edge in \(G\): if \(w_{\min}\) is
the minimum edge weight, we take
\(\sigma_{\mathrm{gauss}} = 0.45\,w_{\min}/1.96\). For scalar Gaussian
edge noise, and for each coordinate of Gaussian vertex noise, this
matches a scalar 95\% reference scale; it is not a radial 2D
displacement bound. A multiplier \(m\) then scales this baseline up or
down. Gaussian noise (N1, N2) is standard in graph perturbation studies;
the tempered power-law model (N3) follows the well-established tempered
stable framework \cite{Rosinski2007} and provides a heavier-tailed
stress test.

\paragraph[N1: i.i.d. rectified Gaussian edge
noise.]{\texorpdfstring{N1: i.i.d. rectified\footnote{we consider
  symmetric positive weighted graphs, so the noise is positive additive
  to avoid negative distances} Gaussian edge
noise.}{N1: i.i.d. rectified Gaussian edge noise.}}\label{n1-i.i.d.-rectified-gaussian-edge-noise.}

Each edge weight is perturbed independently
\(w_{ij} \;\mapsto\; w_{ij} + |\xi_{ij}|, \qquad \xi_{ij} \sim \mathcal{N}(0,\, m^2 \sigma_{\mathrm{gauss}}^2)\).
The absolute value ensures positivity, so this model introduces a
positive bias rather than a zero-mean additive perturbation. It is the
simplest edge-local stress test: i.i.d., non-geometric, and uncorrelated
across vertices.

\paragraph{N2: Gaussian vertex noise.}\label{n2-gaussian-vertex-noise.}

Each vertex coordinate is perturbed independently:
\(v_i \;\mapsto\; v_i + \boldsymbol{\xi}_i, \qquad \boldsymbol{\xi}_i \sim \mathcal{N}(\mathbf{0},\, m^2 \sigma_{\mathrm{gauss}}^2 I_2)\).
Edge weights \(w_{ij} = \|v_i - v_j\|\) inherit correlated perturbations
from independent vertex displacements: two edges sharing a vertex
receive correlated noise.

\paragraph{\texorpdfstring{N3: \ac{CTPL} vertex
noise.}{N3:  vertex noise.}}\label{n3-vertex-noise.}

Each vertex is perturbed as \(v_i \mapsto v_i + \boldsymbol{\xi}_i\),
where \(\boldsymbol{\xi}_i = R_i(\cos\Theta_i, \sin\Theta_i)\) with
\(\Theta_i \sim \mathrm{Uniform}(0,2\pi)\) and the radial displacement
\(R_i\) follows a tempered Lévy-stable density \cite{Rosinski2007}:
\(f_{\mathrm{CTPL}}(r; c, \lambda) = Z(c,\lambda)^{-1}\, f_{\mathrm{Levy}}(r; c)\, e^{-\lambda r}, \qquad r > 0,\quad Z(c,\lambda) = \int_0^\infty f_{\mathrm{Levy}}(r;c)\,e^{-\lambda r}\,dr\),
where \(c > 0\) controls the power-law regime and \(\lambda \ge 0\) is
the exponential tempering parameter. The Lévy kernel retains a
heavy-tail core (\(\sim r^{-3/2}\)) while tempering ensures finite
moments \cite{Rosinski2007}. Given \(r_{\max}\) and tail tolerance
\(\delta\), \(\lambda\) is calibrated on the uncapped law so that
\(\Pr[R > r_{\max}] \le \delta\) before realized radii are capped at
\(r_{\max}\); the calibration procedure and implementation are described
in \cite{cardoen2026spectral}. This is the most challenging model:
vertex displacements are independent and heavy-tailed, while edge-weight
perturbations become dependent for incident edges and remain calibrated
to stay within graph constraints. In the main synthetic sweeps, the
maximum vertex displacement is set to \(r_{\max}=0.45\,w_{\min}\) with
\(\delta=0.05\); the real-geometry trajectory experiments below use the
stricter \(r_{\max}=w_{\min}/3\) convention. For each \ac{CTPL} scale
parameter \(c\), the tempering parameter \(\lambda\) is auto-calibrated
to enforce the stated tail bound. Sections
\ref{sec-results-heatmaps}--\ref{sec-results-operators} use N3 as the
primary noise model; the second row of Section
\ref{sec-results-heatmaps} compares all three.

\subsubsection{Metrics, Heuristics, and
Inference}\label{metrics-heuristics-and-inference}

The base metric is the spectral distance \(d_{\mathrm{SD}}\) with the
\ac{JCP} heuristic (Section \ref{sec-jacobian}). For each condition,
\(K = 100\) independent noise realizations estimate \(\mathrm{IS2}_d\)
and \(\mathrm{IS2}_{\mathcal{A}_f(d)}\). Within each realization, the
base and edited graphs receive independently seeded noise draws, and the
resulting direct and \ac{ROSA}-amplified distances are paired by
realization for inference. The amplification ratio
\(R = \mathrm{IS2}_{\mathrm{ROSA}} / \mathrm{IS2}_{\mathrm{SD}}\) is
reported with paired log-scale 95\% bootstrap CIs
(\(n_{\mathrm{boot}} = 10000\) by default) and \(p_{\mathrm{boot}}\),
the fraction of paired bootstrap replicates with \(R_b^* \le 1\). We
call conditions with \(p_{\mathrm{boot}}<0.05\) bootstrap-supported;
this is an empirical percentile-bootstrap support diagnostic, not an
exact null-calibrated parametric \(p\)-value. Notation used throughout
this paper is indexed in Supplementary Table \ref{tbl:notation-summary}.

\section{Results}\label{sec-results}

Because \ac{ROSA} is an operator on distances, all experiments compare a
base metric \(d\) against its amplified variant \(\mathcal{A}_f(d)\) on
the same graph pair and noise-realization index. Comparing \ac{ROSA}
against other distances directly (e.g., a graph kernel) would confound
distance sensitivity with the amplification mechanism; the relevant
question is always whether the amplified metric improves over its own
unamplified baseline. The guiding experimental question is therefore not
whether \ac{ROSA} works universally, but where it works, where it fails,
and which graph, noise, and filtration features determine that boundary.
We therefore separate synthetic from empirical substrates, and within
each block examine the main experimental axes rather than sweeping their
full Cartesian product at once. Graph size, filtration depth, edit
magnitude, noise model, and heuristic choice are treated in dedicated
blocks; the complete configuration summary is provided in Table
\ref{tbl:experiment-config}.

\subsection{Synthetic Graphs}\label{synthetic-graphs}

\subsubsection{Amplification Across Edit Sites and Noise
Models}\label{sec-results-heatmaps}

This first synthetic experiment tests the core claim of the paper under
controlled localized edits: when edit magnitude and noise level are
varied systematically, does \ac{ROSA} improve the stability of the base
spectral distance, and does that gain depend on the edit-site selection
rule? We mimic sparse biophysical graphs such as organelle networks by
adding the five shortest non-tree edges to a 2D minimum spanning tree,
creating an augmented \ac{MST} (\ac{MSTA}). Using this baseline graph
generator and the primary \ac{CTPL} noise model, we sweep edit magnitude
\(\Delta x \in \{0.5, 1, 2, 5, 10, 25\}\) and noise scale
\(c \in \{0.5, 1, 1.5, 2, 3\}\) across three edit sites
(bridge-adjacent, leaf/endpoint, and small-Fiedler-magnitude hub); full
configuration details are given in Table \ref{tbl:experiment-config}.
Figure \ref{fig:heatmaps} (Row 1) shows that amplification is broad but
not uniform across edit sites. Leaf/endpoint edits yield the strongest
behavior: all \(30/30\) tested \((\Delta x,c)\) conditions give
bootstrap-supported \(R>1\) (\(p_{\mathrm{boot}}<0.05\)), with
\(R \in [1.92,2.39]\). Edits at the bridge-adjacent and
small-Fiedler-magnitude hub sites each achieve \(28/30\)
bootstrap-supported conditions, with peak amplification \(R=2.31\) and
\(R=2.78\), respectively. The unsupported Row 1 cells occur only at the
largest edit magnitude \(\Delta x=25\); their point estimates remain
above one, but the raw spectral distance is already comparatively
strong. Across the moderate-edit regime \(\Delta x \le 10\), \(R\)
remains at least \(1.64\) at all three sites, giving the main synthetic
evidence that the conditional \ac{IS2} gain predicted by Proposition
\ref{thm:snr-guarantee} is realized in practice for this type of graphs
and noise. Because vertex-displacement noise induces dependencies among
incident edge-weight perturbations, we next test if these results depend
on the noise model by repeating the bridge-adjacent sweep from
Section~\ref{sec-results-heatmaps} under three noise types: \ac{CTPL}
vertex displacement (the primary model), Gaussian i.i.d.~vertex
displacement, and Gaussian i.i.d.~edge-weight perturbation. For Gaussian
models, the noise multiplier \(m_{\mathrm{noise}}\) scales
\(\sigma_{\text{gauss}}\); for \ac{CTPL}, it scales the cap
\(r_{\max}=0.45\,w_{\min}\) at fixed \(c=1\) with \(\lambda\) re-tuned
for each multiplier. Edit magnitude \(\Delta x\) is expressed in units
of \(\sigma_{\text{gauss}}\) for cross-model comparability. Row 2 of
Figure \ref{fig:heatmaps} shows that all three noise models yield a
bootstrap-supported \(R > 1\) regime: across the \(30 = 6 \times 5\)
tested \((\Delta x, m_{\mathrm{noise}})\) conditions, \(23/30\)
\ac{CTPL} cells, \(25/30\) Gaussian-vertex cells, and \(23/30\)
Gaussian-edge cells meet the \(p_{\mathrm{boot}}<0.05\) criterion. The
rectified Gaussian edge panel, where noise lacks any geometric
structure, is notable: \ac{ROSA} still amplifies, though the \(R > 1\)
region is shifted toward higher noise multipliers. This shows that the
observed advantage is not confined to a single geometric vertex-noise
construction.

\begin{figure}[htbp]
\centering
\includegraphics[width=\textwidth]{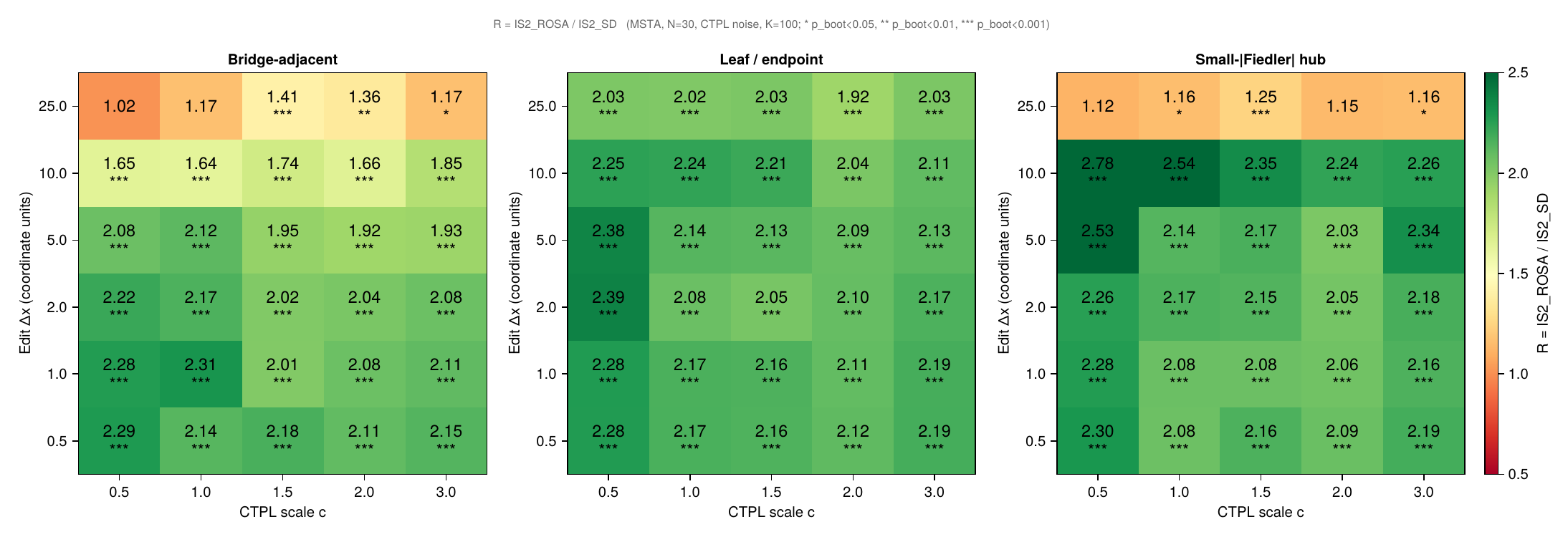}

\medskip

\includegraphics[width=\textwidth]{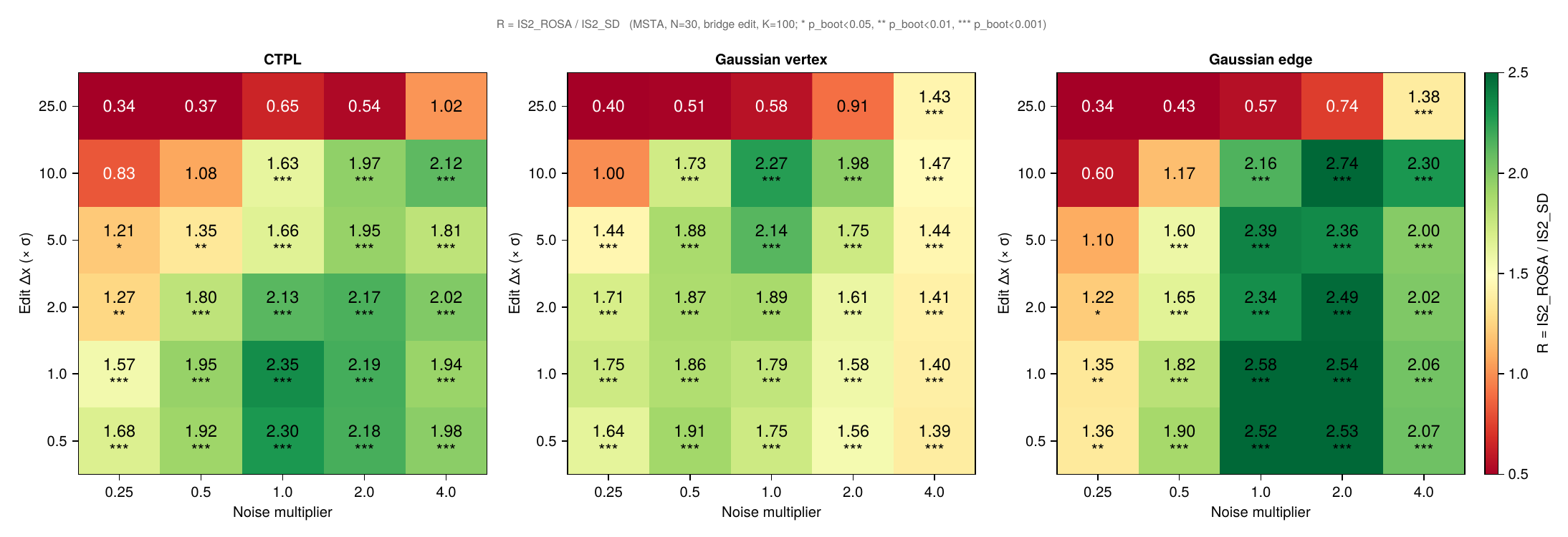}
\caption{Broad but non-universal \ac{ROSA} amplification across scenario and noise assumptions (MSTA, $N=30$, $K=100$).
\textbf{Row~1}: amplification ratio $R = \mathrm{IS2}_{\mathrm{ROSA}} / \mathrm{IS2}_{\mathrm{SD}}$ over edit magnitude $\Delta x$ and \ac{CTPL} scale $c$ for three algorithmically selected edit sites (bridge-adjacent, leaf, and small-Fiedler-magnitude hub). Bootstrap-supported $R>1$ occurs in $28/30$, $30/30$, and $28/30$ cells, respectively.
\textbf{Row~2}: amplification ratio $R$ over edit magnitude ($\times\sigma$) and noise multiplier $m_{\mathrm{noise}}$ for three noise models under the bridge-adjacent edit: \ac{CTPL} vertex displacement, Gaussian i.i.d.\ vertex displacement, and rectified Gaussian i.i.d.\ edge-weight perturbation.
Cells are annotated with $R$ and bootstrap-support stars (${}^{*}p_{\mathrm{boot}}<0.05$, ${}^{**}p_{\mathrm{boot}}<0.01$, ${}^{***}p_{\mathrm{boot}}<0.001$).
Together, the two rows show a broad \ac{IS2} advantage across both spatial edit context and noise model, with boundary cells at large edits or low noise.}
\label{fig:heatmaps}
\end{figure}

\subsubsection{Sparsity Dependence}\label{sec-results-topology}

To test how graph sparsity affects amplification, we construct a
sparsification sequence on the same 30-point set by starting from the
pure \ac{MST} (29 edges) and progressively adding
Delaunay\(\setminus\)\acs{MST} edges (shortest first) up to the full
Delaunay triangulation (78 edges); the \ac{MSTA} baseline (34 edges)
lies near the sparse end of this sweep. Edit vertices are selected once
on the pure \ac{MST} and held fixed across all densification levels;
noise is fixed at \(c=1.5\) (median of Figure \ref{fig:heatmaps}).
Figure \ref{fig:graph-diversity} shows \(R(\Delta x, \text{density})\)
for the three edit sites. Amplification is broad under full filtration:
\(29/30\) (bridge), \(29/30\) (leaf), and \(28/30\) (hub) conditions
yield bootstrap-supported \(R > 1\) (Row 1, Figure
\ref{fig:graph-diversity}). The largest full-filtration values occur at
high density, with \(R\) up to \(3.22\) (bridge), \(3.22\) (leaf), and
\(3.42\) (small-Fiedler-magnitude hub), but the trend is not strictly
monotone across all edit sites and edit magnitudes. The unsupported
full-filtration cells occur at \(\Delta x = 25\), consistent with Figure
\ref{fig:heatmaps}. Since \ac{ROSA} uses \(M = |E|-k\) filtration steps
and denser graphs have more edges, part of the \(R\) increase could
reflect the \(\sqrt{M+1}\) step-count scaling predicted by our \ac{IS2}
theorem under i.i.d.~noise. To disentangle this, Row 2 fixes the
filtration length at 27 steps (the pure-\acs{MST} value) for all
densities, using \(k = |E| - 27\). Under this control,
bootstrap-supported \(R>1\) remains in \(29/30\), \(27/30\), and
\(28/30\) cells for bridge, leaf, and hub edits, respectively, so
fixed-step evaluation reduces but does not eliminate the dense-graph
advantage. Densification can therefore benefit \ac{ROSA} in two ways:
more filtration steps (a mechanical advantage) and richer local spectra
that the \ac{JCP} heuristic can exploit (a structural advantage).
However, as shown in Section \ref{sec-results-fmri-boundary}, not all
dense graphs enable \(R>1\): counterexamples are complete graphs
(e.g.~correlation matrices) where the local versus global balance can no
longer be extracted via \ac{JCP}.

\begin{figure}[htbp]
\centering
\includegraphics[width=\textwidth]{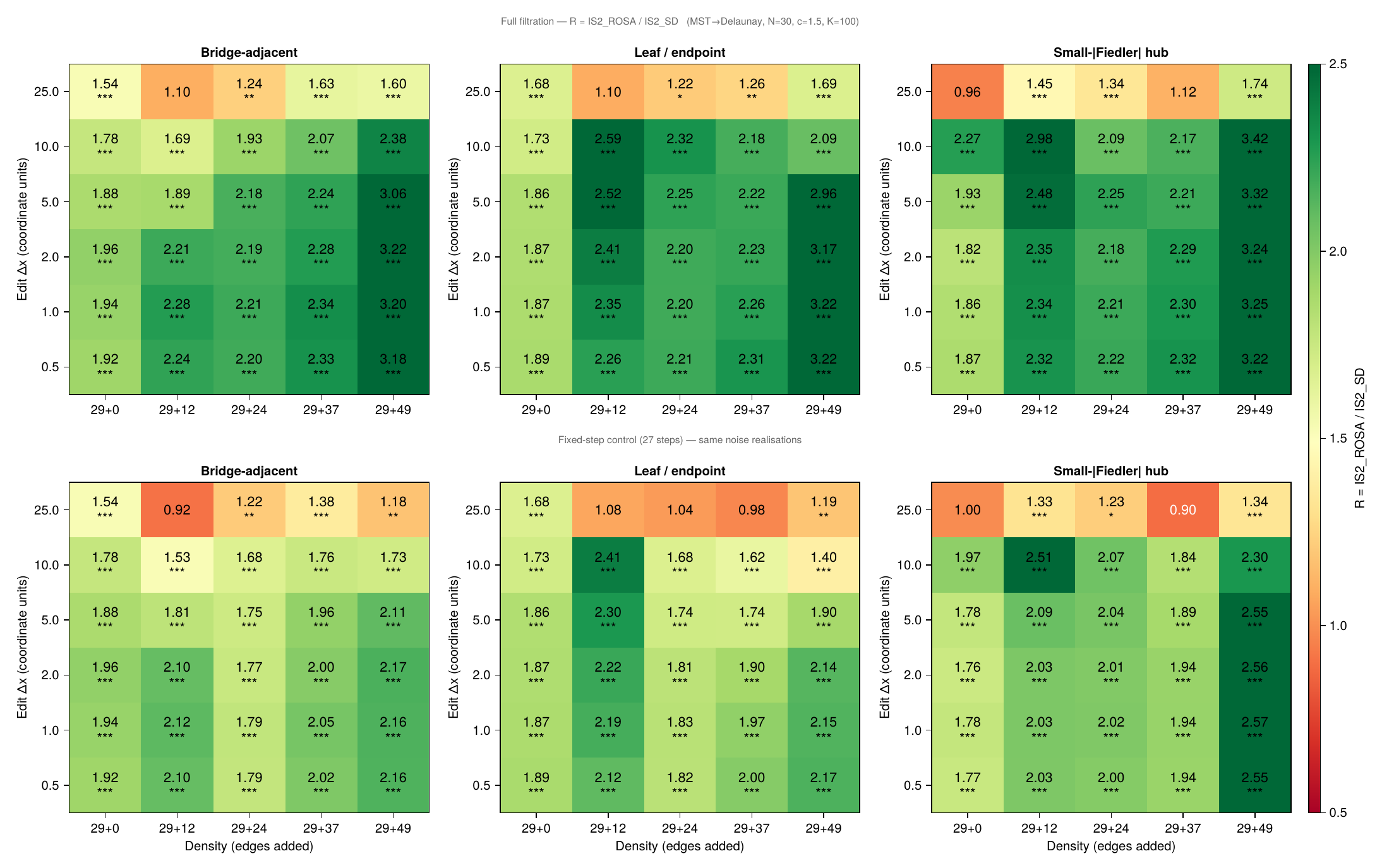}
\caption{Amplification ratio $R$ over edit magnitude $\Delta x$ and graph density (MST$\to$Delaunay) for three spatial vertex categories ($c=1.5$, $K=100$).
\textbf{Row~1}: full filtration (all $|E|-2$ steps).
\textbf{Row~2}: fixed-step control (27 steps $=$ MST filtration length), which reduces but does not remove the density-associated gain.
Density is measured by the number of Delaunay edges added to the MST (29+0 = pure MST, 29+49 = full triangulation).
Cells annotated with $R$ and bootstrap-support stars (${}^{*}p_{\mathrm{boot}}<0.05$, ${}^{**}p_{\mathrm{boot}}<0.01$, ${}^{***}p_{\mathrm{boot}}<0.001$).
The full-filtration $R$ increase from $\approx 1.9$ to $\approx 3.1$ decomposes into a step-count contribution ($\approx 0.6$) and a genuine spectral amplification ($\approx 0.6$, persistent under fixed steps).}
\label{fig:graph-diversity}
\end{figure}

\subsubsection{Operator Generality}\label{sec-results-operators}

The \ac{IS2} amplification proposition (Proposition
\ref{thm:snr-guarantee}) is operator-agnostic: it holds for any distance
function \(d\) satisfying basic regularity. In this section we test this
empirically by replacing the default Wasserstein\(_2\) spectral distance
with NetLSD (heat-trace RMSE), effective resistance distance, and a
\acf{PH} distance based on the \(H_0\) Wasserstein distance between
persistence diagrams computed from shortest-path filtrations. Figure
\ref{fig:operator-comparison} shows the results: the top row displays
amplification ratio \(R\) heatmaps, the bottom row shows absolute
\ac{IS2} values (\ac{ROSA} above, raw operator below in each cell). For
the three non-\ac{PH} operators, \ac{ROSA} consistently amplifies:
Wasserstein\(_2\) yields 25/30 bootstrap-supported conditions with
\(\mathrm{IS2}_{\mathrm{ROSA}} \approx 7.7\)--\(11.1\); NetLSD achieves
30/30 bootstrap-supported conditions at
\(\mathrm{IS2}_{\mathrm{ROSA}} \approx 3\)--\(4\); resistance distance
reaches 30/30 bootstrap-supported conditions with peak \(R = 3.80\). The
\ac{PH} operator reveals a structurally different regime: \(R < 1\) in
all 30 conditions, with raw \ac{PH} achieving
\(\mathrm{IS2} \approx 3.8\)--\(9.8\) while
\(\mathrm{ROSA}(\mathrm{PH})\) degrades to
\(\mathrm{IS2} \approx 2.9\)--\(3.8\). This attenuation is not a
deficiency of \ac{PH}, which is the second-strongest raw discriminator
after Wasserstein\(_2\), but a consequence of the interaction between
the \ac{JCP} edge ordering and the \ac{PH} distance. \ac{JCP} removes
high-coherence, spectrally sensitive edges early in the filtration; for
\ac{PH}, these removals can quickly produce graph-fragmentation events
that overwhelm the original edit signal with discrete topological
transitions (component births/deaths). Spectral metrics degrade
gracefully under progressive sparsification; \ac{PH} undergoes
discontinuous jumps, violating the smoothness conditions required for
ROSA's amplification guarantee. ROSA's amplification is thus genuinely
operator-dependent: not a universal booster, but a mechanism that
rewards smooth, filtration-compatible base metrics. \ac{ROSA} and
\ac{PH} are complementary: \ac{PH} enjoys strong stability guarantees
\cite{CohenSteiner2007}, while in our benchmark \ac{ROSA} amplifies
weaker spectral metrics to reach \ac{IS2} levels comparable with or
above raw \ac{PH} (median \(\mathrm{ROSA}(\mathrm{W}_2) \approx 8.6\)
vs.~median raw \ac{PH} \(\approx 4.5\)). These absolute \ac{IS2} levels
are operator-specific, so the comparison is indicative rather than a
claim of a common calibrated scale across operators.

\begin{figure}[htbp]
\centering
\includegraphics[width=\textwidth]{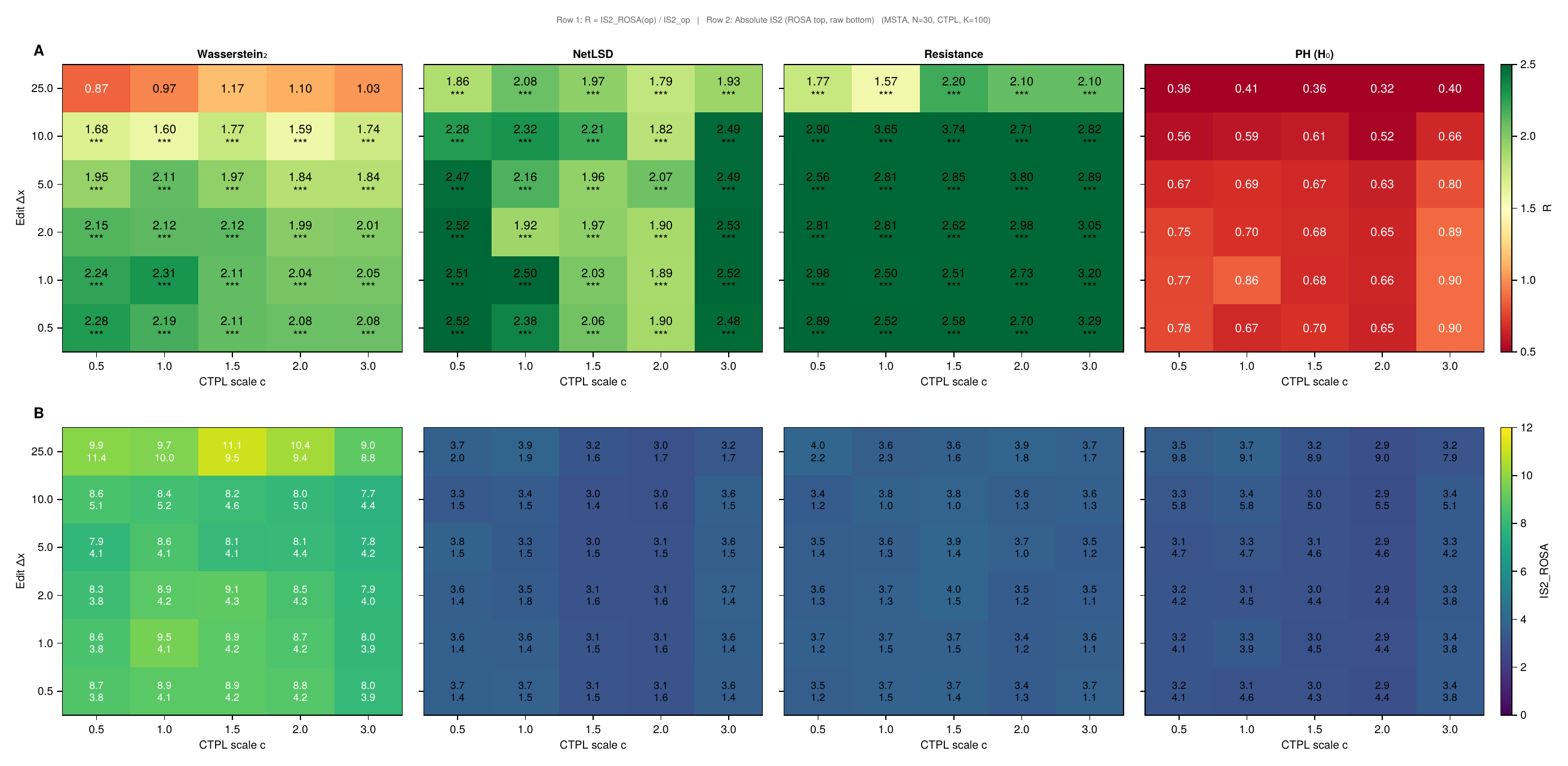}
\caption{Operator generality (MSTA, $N=30$, bridge edit, $K=100$).
\textbf{(A)} Top row: amplification ratio $R = \mathrm{IS2}_{\mathrm{ROSA}} / \mathrm{IS2}_{\mathrm{op}}$.
\textbf{(B)} Bottom row: absolute IS2 (upper number: ROSA; lower number: raw operator).
Wasserstein$_2$ (25/30 bootstrap-supported), NetLSD (30/30), resistance (30/30, peak $R = 3.80$), \ac{PH} $H_0$ (0/30, $R < 1$ everywhere).
\ac{ROSA}(W$_2$) achieves the highest median absolute \ac{IS2} ($\approx 8.6$), exceeding median raw \ac{PH} ($\approx 4.5$).}
\label{fig:operator-comparison}
\end{figure}

\subsubsection{Scaling and Parallelism}\label{sec-results-scaling}

ROSA's filtration decomposes into independent spectral distance
computations per step, enabling straightforward parallelization
(Section~\ref{sec-complexity}). Figure \ref{fig:scaling} measures
speedup across graph vertex counts \(N \in \{128, 256, 512, 1024\}\) and
Julia thread counts \(\{1, 2, 4, 8\}\), with \acs{BLAS} fixed at 4
threads and a constant 64-step filtration on graphs of density
\(|E| \approx 2N\). The measured 8-thread speedups range from
\(1.97\times\) to \(2.42\times\), with most of the gain obtained by 2--4
Julia threads. The bottleneck is likely a combination of \acs{BLAS}
oversubscription and memory-system effects. Each filtration step calls
LAPACK's eigendecomposition, which is itself internally parallel; with
\acs{BLAS} pinned to 4 threads, 8 Julia threads spawn
\(8 \times 4 = 32\) logical \acs{BLAS} threads competing for physical
cores, and the repeated dense linear-algebra workload likely incurs
additional cache pressure and eviction. The constraint
\(K_{\mathrm{Julia}} \times K_{\mathrm{BLAS}} \leq K_{\mathrm{cores}}\)
is a useful rule of thumb for effective scaling of any
\acs{BLAS}-accelerated operator. For operators without internal
parallelism and with modest memory pressure, outer-loop speedup can
approach linear scaling with thread count. The recommended scaling path
for \acs{BLAS}-accelerated operators is process-level parallelism: \(K\)
independent processes, each with its own controlled \acs{BLAS} thread
pool, reducing shared-thread contention (Section~\ref{sec-complexity}).
Orthogonally, incremental rank-1 eigenvalue updates can reduce the
per-step cost from \(O(n^3)\) to \(O(n^2)\) for \(W_2\), yielding a
factor-of-\(m\) algorithmic improvement independent of parallelism.
Panel (B) of Figure \ref{fig:scaling} shows incremental checkpoint
strategies at \(N=1024\) with \(L=64\) filtration steps: the best
secular-equation checkpoint strategy achieves a measured \(1.34\times\)
speedup over sequential eigendecomposition. Panels (C--D) confirm that
strided filtration preserves \(R > 1\) for the tested sizes
\(N\in\{128,256,512\}\), with all tested settings remaining
bootstrap-supported at \(p_{\mathrm{boot}}<0.05\), while substantially
reducing wall-clock time as stride increases. Panels (E--I) extend this
computational story from algorithmic profiling to spectral
approximation. Panels (E--G) show that the cumulative \(W_2^2\) signal
is strongly localized in a small fraction of the spectrum, and that
magnitude-ranked truncation tracks the contribution-ranked optimum
across signal, noise, and edit+noise settings. Panels (H--I) show that
this localization is computationally actionable in the tested regimes:
sparse iterative top-fraction eigensolvers achieve runtime and memory
savings at larger \(N\) for small retained fractions, but the advantage
shrinks and can reverse as the retained spectral fraction grows.

\begin{figure}[htbp]
\centering
\begin{minipage}{0.31\textwidth}
\centering
\includegraphics[width=\textwidth]{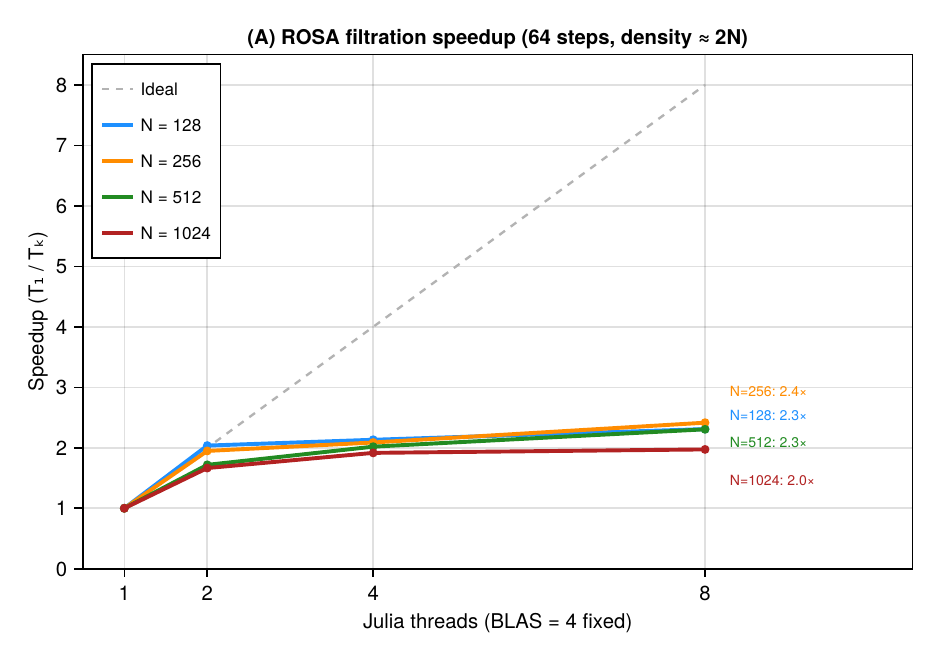}
\end{minipage}\hfill
\begin{minipage}{0.31\textwidth}
\centering
\includegraphics[width=\textwidth]{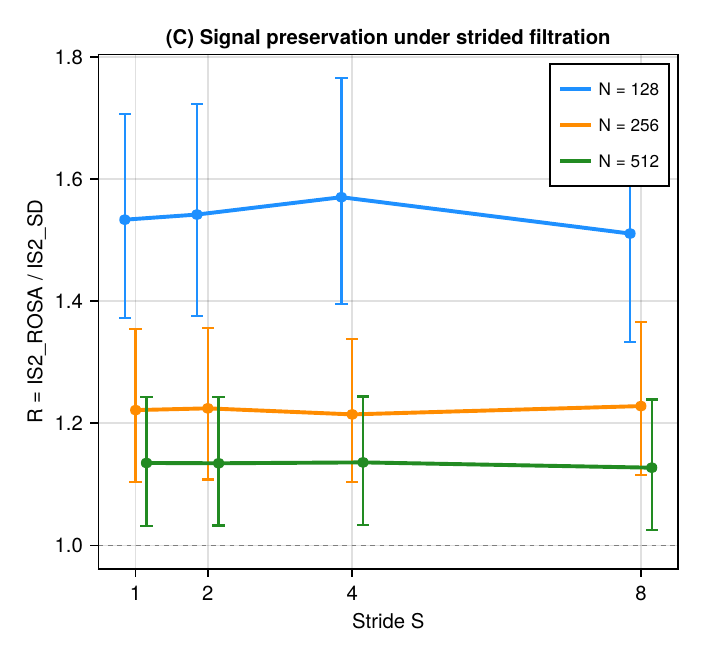}
\end{minipage}
\hfill
\begin{minipage}{0.31\textwidth}
\centering
\includegraphics[width=\textwidth]{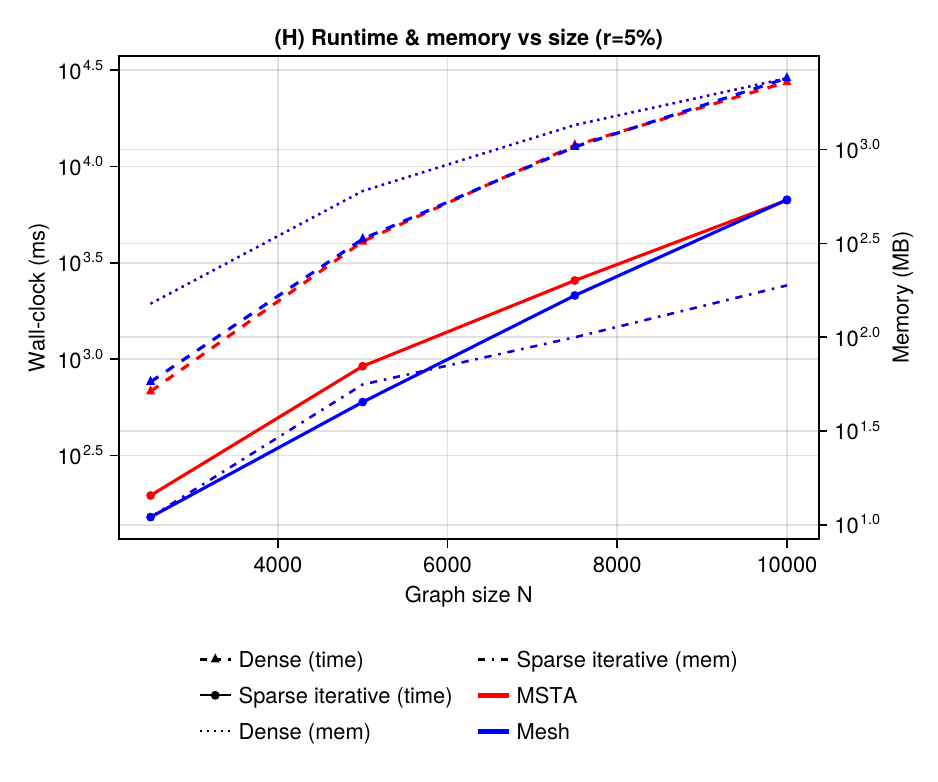}
\end{minipage}

\vspace{0.1em}

\begin{minipage}{0.31\textwidth}
\centering
\includegraphics[width=\textwidth]{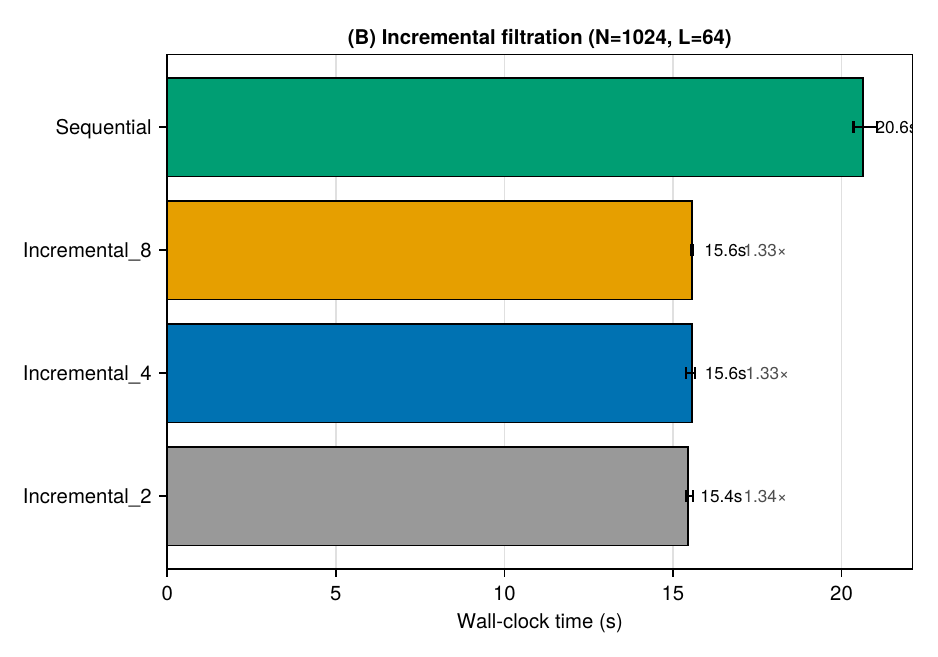}
\end{minipage}\hfill
\begin{minipage}{0.31\textwidth}
\centering
\includegraphics[width=\textwidth]{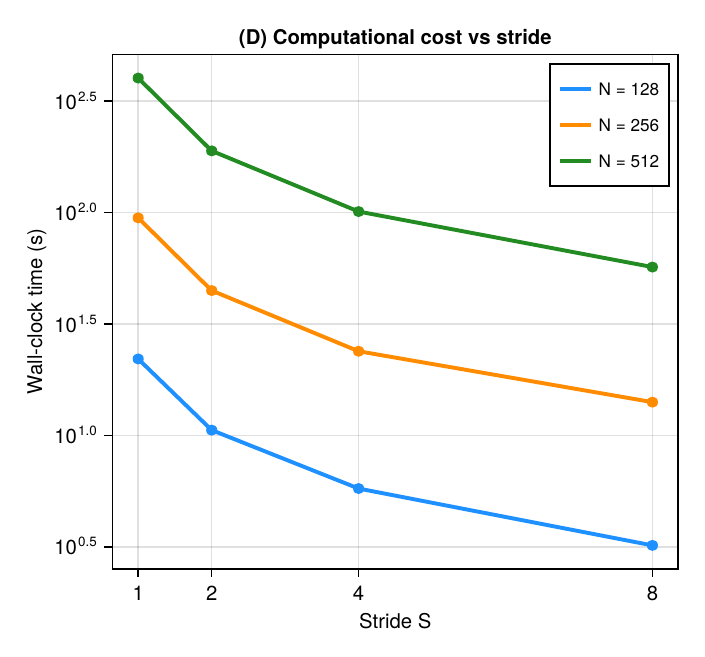}
\end{minipage}
\hfill
\begin{minipage}{0.31\textwidth}
\centering
\includegraphics[width=\textwidth]{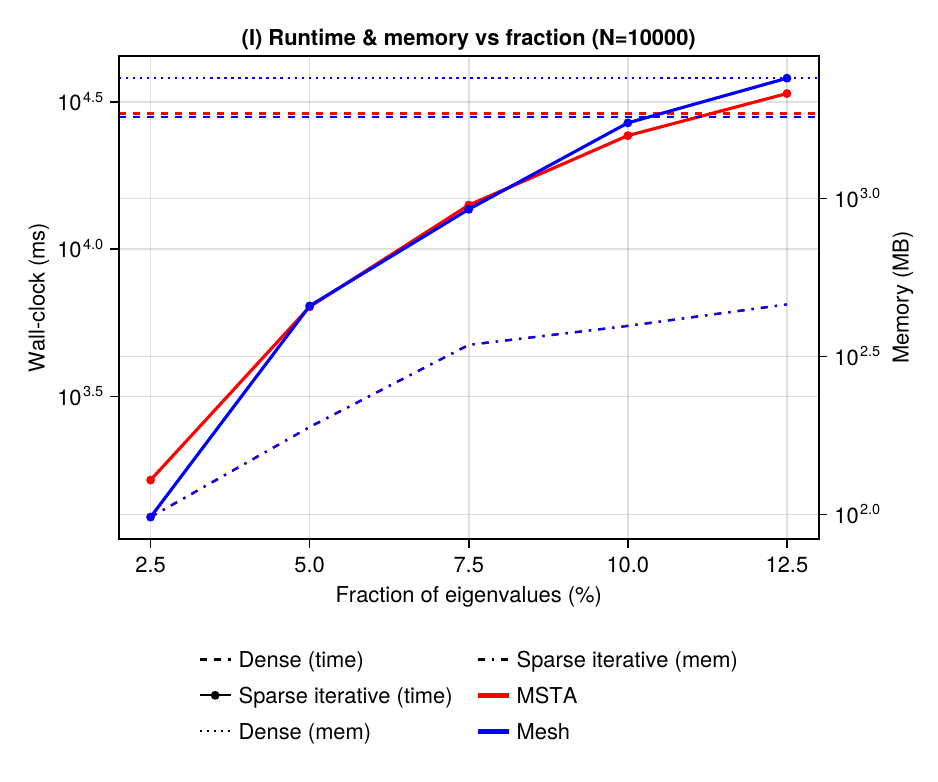}
\end{minipage}

\vspace{0.1em}

\begin{minipage}{0.78\textwidth}
\centering
\includegraphics[width=\textwidth]{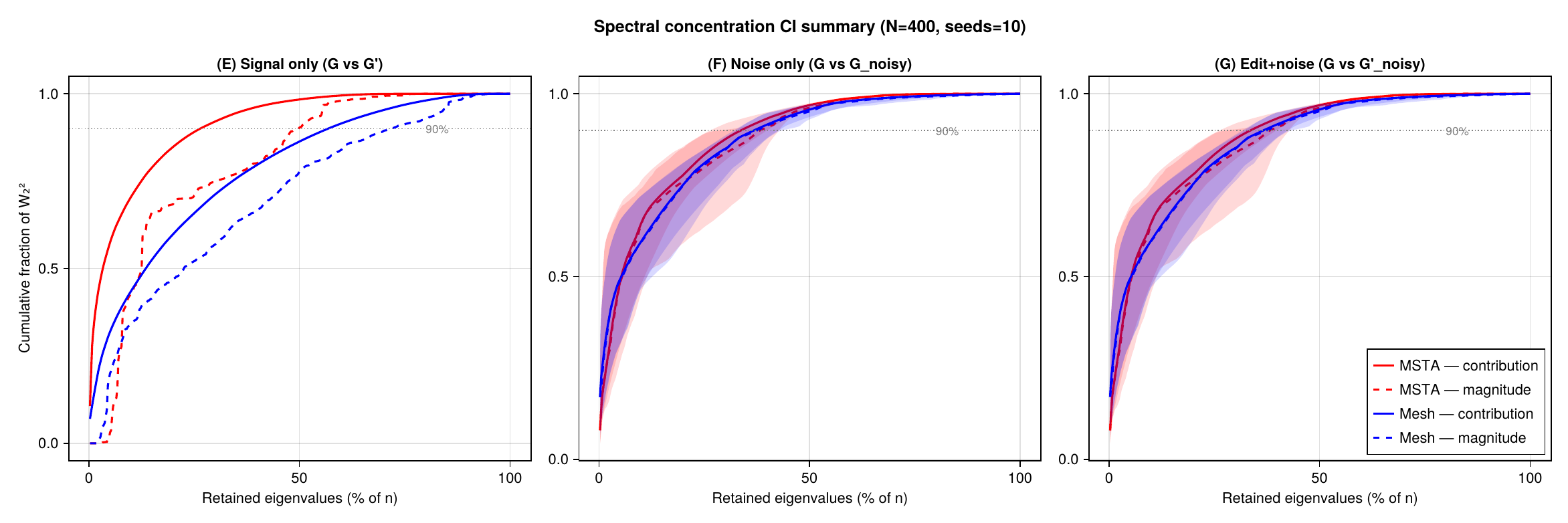}
\end{minipage}
\caption{Computational performance of ROSA.
(A)~Thread-level parallel speedup reaches $1.97$--$2.42\times$ at 8 Julia threads, consistent with BLAS oversubscription and likely additional cache-pressure effects.
(B)~Incremental checkpointing at $N=1024$, $L=64$ yields a best measured $1.34\times$ speedup over sequential eigendecomposition.
(C\textendash D)~For tested $N\in\{128,256,512\}$, strided filtration preserves $R > 1$ while substantially reducing wall-clock time as stride increases.
(E\textendash G)~Most cumulative $W_2^2$ signal is captured by a small spectral fraction across signal, noise, and edit+noise regimes.
(H\textendash I)~Sparse iterative top-fraction eigensolvers turn this localization into runtime and memory savings mainly at small retained spectral fractions.}
\label{fig:scaling}
\end{figure}

\subsubsection{\texorpdfstring{Empirical Validation of Proposition
\ref{thm:snr-guarantee}}{Empirical Validation of Proposition }}\label{sec-results-sensitivity}

We directly estimate the operative quantities of Proposition
\ref{thm:snr-guarantee} for the \(W_2\)+\acs{JCP}+\ac{CTPL}
instantiation across the same \(\Delta x \times c\) grid as Figure
\ref{fig:heatmaps}. As noted after Proposition \ref{thm:snr-guarantee},
we validate the operative variance and mean-shift conditions (A2) and
(A3) directly rather than introducing stronger ordering-based sufficient
conditions. The signed mean-shift condition (A3) holds in 25 of 30
cells, while \(R > 1\) in 28 of 30 cells. All five A3 failures occur at
\(\Delta x = 25\): two also have \(R \le 1\), while three still have
\(R>1\). This three-cell discrepancy reflects the fact that (A3) is a
sufficient operative condition, not a necessary one: some cells still
improve the \ac{IS2} score even when the signed mean-shift bound is not
met. At \(M=32\), the empirical medians decompose the diagnostic into
\(\sqrt{33}/\beta \approx 1.49\) from aggregation after correlated
variance growth and \(1+\gamma \approx 1.42\) from signed mean shift;
their product \(R_{\mathrm{pred}} \approx 2.12\) matches the observed
median \(R \approx 2.0\). This agreement should be read as an empirical
consistency check on the operative quantities in Proposition
\ref{thm:snr-guarantee}, not as a sharp out-of-sample prediction;
Appendix Figure \ref{fig:is2-holdout} reports a repeated measure/holdout
split sweep that quantifies the stability and variance of this estimate.
Full per-cell results are reported in Table \ref{tab:sensitivity}.

\subsubsection{Stochastic ROSA: Ordering Uncertainty and the Swap
Probability}\label{sec-results-stochastic}

The operative sufficient conditions (A2)--(A3) are stronger than
necessary (Section \ref{sec-results-sensitivity}), motivating stochastic
\ac{ROSA} (Section \ref{sec-stochastic-rosa}) as a practical hedge
against ordering uncertainty. We evaluate stochastic \ac{ROSA} across
three representative noise regimes (\(c \in \{0.5, 1.5, 3.0\}\)) at
fixed signal strength (\(\Delta x = 2\)), sweeping the swap probability
\(p \in \{0, 0.05, 0.10, 0.20, 0.30, 0.50, 0.70\}\) with \(K=100\) noise
realizations and \(n_{\mathrm{trials}}=30\) ordering perturbations per
realization. The deterministic case \(p=0\) is equal by construction,
and all 18 positive-swap settings improve over the deterministic
baseline. This dominance is reported descriptively; we do not perform a
separate hypothesis test of stochastic-versus-deterministic improvement.
The gain is largest at high noise (\(c=3.0\)): \(R\) increases from 2.11
to 2.35 at the best tested swap probability, an 11.6\% improvement. At
low noise (\(c=0.5\)) the best gain is 5.5\%; at medium noise
(\(c=1.5\)) the effect remains modest, peaking at 2.4\%. Table
\ref{tab:stochastic-summary} reports the deterministic baseline, the
recommended compromise setting \(p=0.3\), and the best value observed in
the full sweep; the complete sweep remains in Table
\ref{tab:stochastic}. The swap probability \(p = 0.3\) is a compromise
setting that captures part of the gain across regimes without pushing
the ordering perturbation to its most aggressive setting; it is not the
best observed setting in the medium-noise sweep. The stochastic
ordering-width summary is the mean 95\% interval width of \ac{ROSA} sums
under ordering perturbations, giving a direct ordering-sensitivity
diagnostic.

\subsection{Real-World Graph
Geometries}\label{real-world-graph-geometries}

\subsubsection{Controlled Dynamic Changes on Real Biological Graph
Geometries}\label{sec-results-real-geometry}

To test whether \ac{ROSA} remains useful in time-resolved monitoring on
empirically derived graph geometry, we construct two controlled dynamic
experiments on real biological substrates and evaluate consecutive
comparisons \(d(G_t, G_{t+1})\) under the same \ac{CTPL} noise model and
parameters. The goal here is deliberately modest: not universal recovery
of underlying change magnitude, which is graph- and spectrum-dependent,
but consistent amplification of weak time-varying change above direct
spectral comparison. The first substrate is a mitochondrial skeleton
extracted from yeast confocal microscopy (MitoTNT test data
\cite{wang2023mitotnt}, frame 040, BSD-3 license), reduced to its
largest connected component (\(N = 38\), \(E = 46\)). We move a junction
hub (\(v_{22}\)) toward a nearby branch target (\(v_{31}\)) with three
bounded displacement schedules, namely seesaw, quadratically
accelerating, and oscillating, all capped at \(D_{\max}=5~\mathrm{px}\).
The second substrate is a junction graph extracted from a TissueMiner
epithelial \ac{ROI} \cite{etournay2016tissueminer} (\(N = 34\),
\(E = 43\)). There we move one junction (\(v_{15}\)) toward a
neighboring target (\(v_{14}\)) with the same three schedules, capped at
\(D_{\max}=4~\mathrm{px}\). Appendix Figure \ref{fig:real-bio-setup}
shows the two substrates and the selected edited/target vertices. In
both cases, the edit direction and amplitude are chosen so that the
geometry remains well-posed throughout the trajectory: the global
minimum edge length stays unchanged under the imposed motion, and the
noise limit remains fixed at \(w_{\min}/3\). All runs use \(K = 50\)
noise realizations at \(c = 1.5\), reduced from the synthetic default
\(K=100\) to keep the larger dynamic benchmark suite computationally
manageable while retaining the default 10,000-resample bootstrap
intervals. Figure \ref{fig:real-bio-dynamics} reports the pairwise
amplification ratio \(R\) with 95\% bootstrap confidence intervals.
Across all three schedules and both substrates, \ac{ROSA} maintains
\(R > 1\) throughout the trajectory. For the mitochondrial graph, the
mean pairwise amplification ranges from \(1.94\) to \(2.01\) across the
three schedules, with per-step values between \(1.38\) and \(2.60\). For
the cell junction graph, the mean pairwise amplification ranges from
\(2.01\) to \(2.05\), with per-step values between \(1.57\) and
\(2.58\). All pairwise steps are bootstrap-supported at
\(p_{\mathrm{boot}}<0.05\). The two substrates differ substantially in
geometry and biological interpretation, yet both show the same practical
pattern: \ac{ROSA} makes local dynamic changes more visible over time
than direct spectral comparison alone.

\subsubsection{Identifying Protein Structure Variation Under Overlapping
Noise}\label{identifying-protein-structure-variation-under-overlapping-noise}

The preceding biological trajectories impose a controlled synthetic edit
on real graph geometry. The protein conformer row serves a different
purpose: it asks whether \ac{ROSA} still has a usable regime when the
signal is a naturally occurring geometric deformation rather than one we
inject ourselves. We use the ubiquitin NMR ensemble 1D3Z
\cite{cornilescu1998validation}, which contains ten experimentally
derived conformers of the same 76-residue protein. These are converted
to fixed-topology C\(\alpha\) contact graphs on the shared 242-edge
contact core, so the experimental variable is genuine conformational
variation under added Gaussian coordinate noise, not a synthetic
topology edit. This is also a harder test of locality than the earlier
trajectory experiments. The conformer differences are nearly global but
heavy-tailed, with a mean pairwise RMSD of about \(1.96\),\AA, while the
graph-calibrated full-noise baseline is already about \(1.66\),\AA~in
RMSD units. In other words, the added noise overlaps substantially with
the natural signal scale. In the single-layer setting, \texttt{1\ vs\ x}
comparisons still retain median \(R > 1\) throughout the tested Gaussian
sweep up to \(1.5\sigma\), but the number of bootstrap-supported pairs
decreases from \(9/9\) to \(5/9\) as noise rises. These \ac{PDB}
boundary sweeps use \(K=50\) noise realizations and 2,000-resample
bootstrap intervals, reduced from the 10,000-resample intervals used in
the main trajectory experiments to keep the expanded boundary scan
tractable. The message of this panel is therefore not just proof of
principle. It shows the expected boundary behavior: \ac{ROSA} remains
useful while the conformational signal exceeds the effective noise
floor, but confidence degrades once the two scales strongly overlap. The
layered comparison design matters as well. At fixed noise
\(\sigma = 0.5\), all four tested layered patterns remain above \(R=1\),
but the gain differs materially across constructions. Across the full
layered sweep, the tested contrasts stay above \(R=1\) in 22 of 24
settings, while bootstrap support decreases at higher noise. The
strongest \texttt{xx\ vs\ yy} construction also crosses below unity by
the highest noise level. The conformer result is therefore deliberately
framed as a boundary map rather than a universal success claim: on
fixed-topology conformer graphs, \ac{ROSA} can amplify noisy spectral
discrimination, but the regime depends on both noise scale and layered
comparison design. The final two rows extend the mitochondrial
experiment to layered graphs formed by stacking temporally offset copies
of the base graph and coupling corresponding vertices across layers
\cite{DeDomenico2013multilayer}; the layer span parameter
\(k \in \{2,3,4\}\) denotes the number of coupled copies. Using ramp and
pulse patterns, we sweep both displacement magnitude and noise shape
while keeping the noise limit fixed at \(w_{\min}/3\) and re-tuning the
\ac{CTPL} tempering parameter for each \(c\). In all 36 layered
conditions, \ac{ROSA} retains \(R > 1\), showing that temporal lifting
preserves the amplification gain across both moderate displacement
variation and substantial changes in noise shape. As an additional
real-world construction check, Supplementary Figure
\ref{fig:retina-multiplex-jcp} represents ten manually segmented TREND
retinal vessel images as two coupled edge-attribute layers and reports
raw \ac{ROSA} response over \(\mathrm{JCP}_k\). This noiseless cohort
diagnostic tests the multiplex encoding and the effect of neighbourhood
radius; it is not a claim of superiority over direct spectral comparison
or of pathology detection.

\subsubsection{Signal-Recovery Limits on Dense
Connectomes}\label{sec-results-fmri-boundary}

To probe a real-world-data boundary case outside the geometric setting,
we consider a representative resting-state fMRI connectome pair from the
Autism Brain Imaging Data Exchange (ABIDE) \cite{dimartino2014abide},
using the ABIDE Preprocessed derivative at the YALE site
(\texttt{cpac/filt\_noglobal/rois\_ho}).\footnote{\href{https://preprocessed-connectomes-project.org/abide/}{ABIDE Preprocessed portal}.}
These connectomes are represented first as complete weighted graphs
whose edge weights are affine-transformed \ac{ROI} correlations. On this
raw complete-connectome representation, all four tested heuristics are
neutral or adverse: \ac{JCP}, effective resistance, weighted-degree
asymmetry, and raw edge-weight ordering all give \(R \lesssim 1\) within
uncertainty. This is a structural failure mode rather than a tuning
issue. In complete or near-complete weighted graphs, edge neighborhoods
cease to be local, so the filtration becomes close to a global rescaling
of the base distance rather than a locally informative edge-removal
process. We then restore locality by sparsifying one representative
subject to a symmetric \(k\)NN\(\cup\)\acs{MST} graph with \(k=8\)
(defined in the supplementary material), and apply nested random
edge-support prefixes with mild Gaussian edge-weight perturbations. In
this regime amplification reappears, but not monotonically in the
edited-edge count
\(n_{\mathrm{edit}} \in \{1, 2, 4, 8, 16, 32, 64, 128, |E|/2, |E|\}\);
full configuration is given in Table \ref{tbl:experiment-config}. Small
disconnected edited supports weaken as more isolated witnesses are
added, intermediate partially merged supports approach \(R \approx 1\),
larger coherent edited regions restore amplification, and fully global
edited support relaxes back toward \(R \approx 1\). The topology of the
edited support therefore helps explain the recovery; support size alone
is not a sufficient description of the regime. Figure
\ref{fig:fmri-boundary} summarizes this hard-negative plus boundary
result, which is consistent with the scope discussion after Proposition
\ref{thm:snr-guarantee}: \ac{JCP}-based amplification depends on the
persistence of local filtration structure.

\begin{figure}[htbp]
\centering
\includegraphics[width=\textwidth]{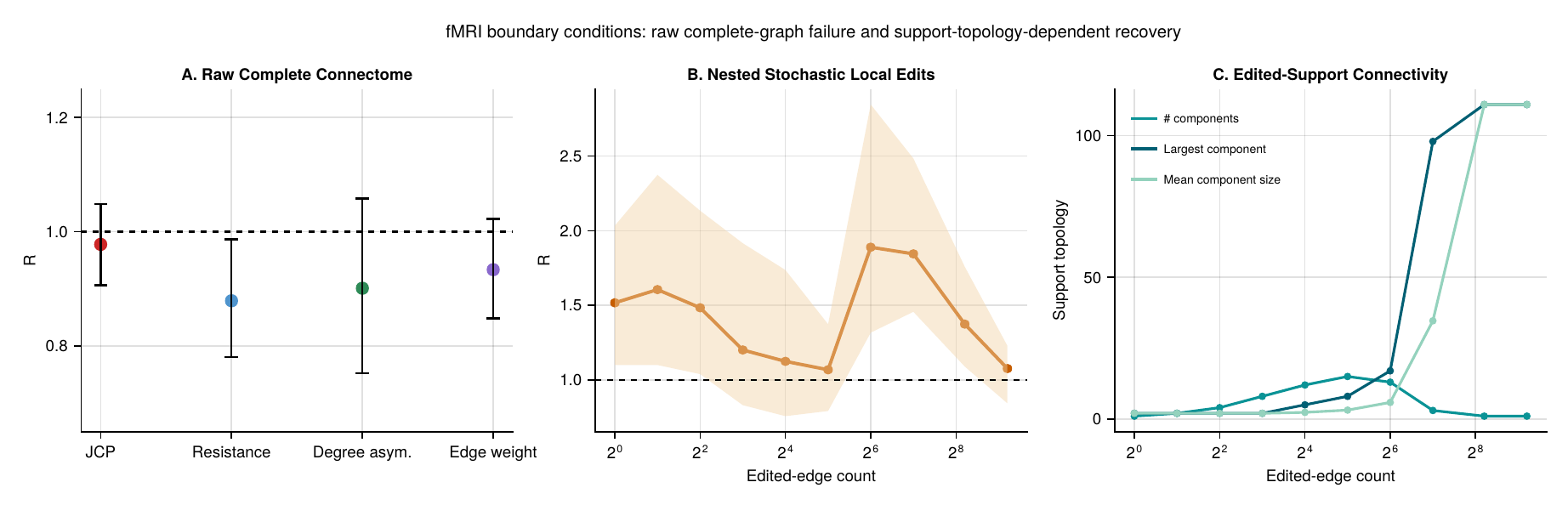}
\caption{Signal-observability limits for \ac{ROSA} on dense fMRI connectomes.
(A) On a raw complete YALE ABIDE representative pair, none of four tested heuristics (\ac{JCP}, effective resistance, weighted-degree asymmetry, raw edge weight) yields convincing amplification above the direct spectral distance.
(B) On a sparsified $k$NN$\cup$\acs{MST} representation ($k=8$), nested random edge-support prefixes restore $R>1$ with a shaded 95\% bootstrap interval, but in a non-monotone way.
(C) The non-monotonicity aligns with the topology of the edited support: many tiny disconnected witnesses drive $R$ toward $1$, a few larger coherent components restore amplification, and fully global edited support again suppresses locality.
These results mark a real-world-data boundary case rather than a contradiction of the theory: complete or near-complete correlation graphs erase the local filtration structure on which \ac{JCP}-based amplification depends.}
\label{fig:fmri-boundary}
\end{figure}

\begin{figure}[htbp]
\centering
\includegraphics[width=\textwidth]{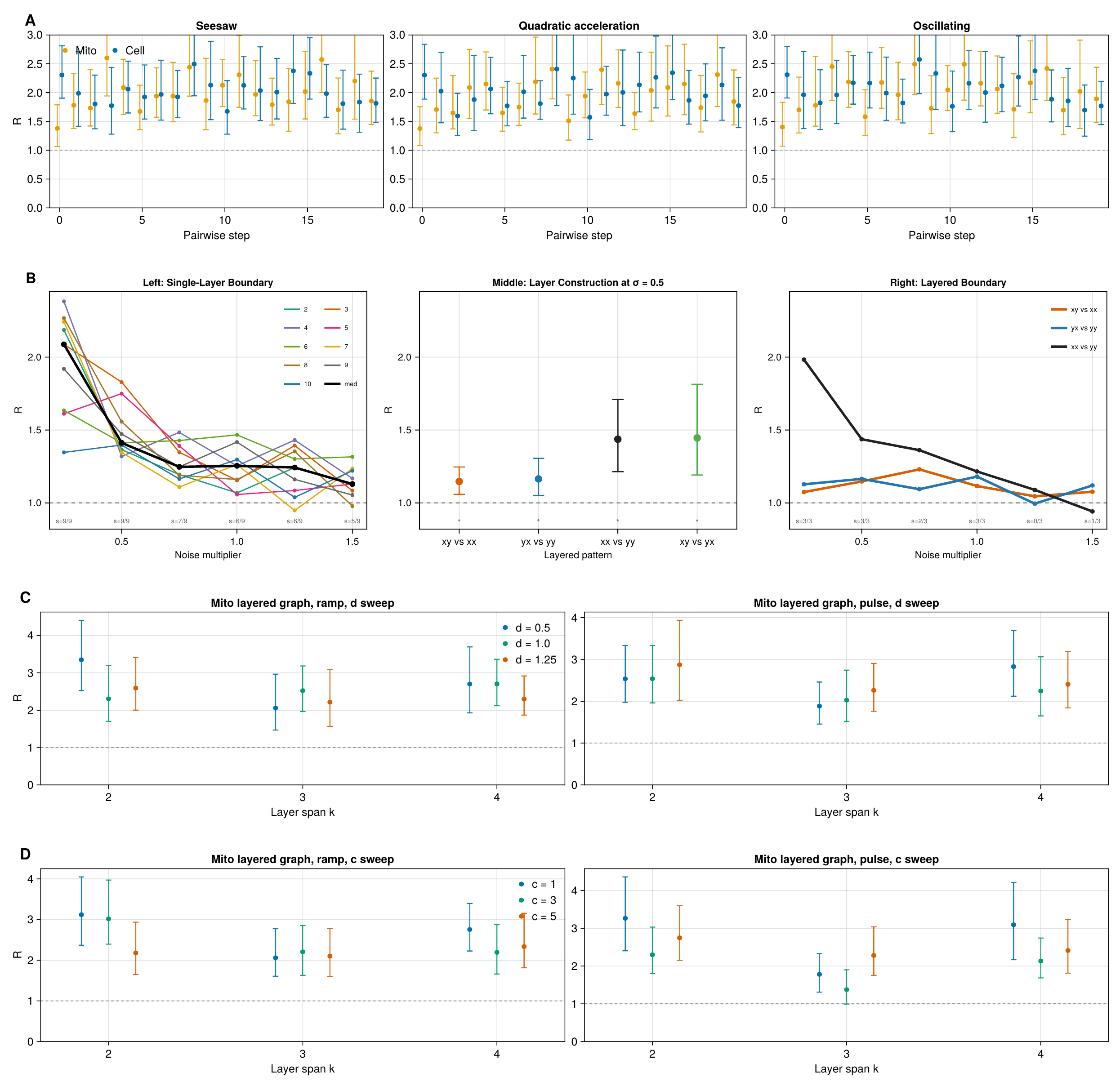}
\caption{ROSA amplification over controlled dynamic changes on real biological graph geometries.
Row A: pairwise amplification ratio $R$ with 95\% bootstrap confidence intervals for mitochondrial (orange) and cell-junction (blue) substrates under three bounded displacement schedules.
Row B: natural conformer-variation boundary experiment on ubiquitin NMR contact graphs (\acs{PDB} 1D3Z). Left: single-layer `1 vs x` amplification boundary under Gaussian coordinate noise, annotated with bootstrap-support counts $s=a/b$. Middle: effect of layered comparison design at fixed noise $\sigma=0.5$, with 95\% bootstrap confidence intervals and explicit `*`/`NS` support labels. Right: layered boundary curves for three symmetric contrasts, again annotated with support counts $s=a/b$.
Rows C--D: mitochondrial layered-graph extension for layer spans $k \in \{2,3,4\}$, shown for ramp and pulse patterns under displacement and noise-shape sweeps.
Across direct pairwise monitoring, the fixed-topology conformer boundary test, and the layered mitochondrial extension, \ac{ROSA} retains a substantial real-data regime with $R > 1$, while the conformer row makes clear that this gain is noise- and construction-dependent rather than unconditional.}
\label{fig:real-bio-dynamics}
\end{figure}

\begingroup
\small
\begin{table}[htbp]
\centering
\small
\caption{Stochastic ROSA summary across the three representative noise regimes at fixed $\Delta x = 2$. $R_{\mathrm{det}}$ is deterministic ROSA, $R_{0.3}$ uses the recommended swap probability $p=0.3$, and $R_{\max}$ is the best stochastic value observed over the tested sweep $p \in \{0,0.05,\ldots,0.7\}$.}
\label{tab:stochastic-summary}
\begin{tabular}{r r r r r r}
\hline
$c$ & $R_{\mathrm{det}}$ & $R_{0.3}$ & $\Delta R_{0.3}$ (\%) & $p_{\max}$ & $R_{\max}$ \\
\hline
0.5 & 2.12 & 2.21 & +4.2 & 0.70 & 2.24 \\
1.5 & 1.95 & 1.98 & +1.6 & 0.70 & 2.00 \\
3.0 & 2.11 & 2.24 & +6.2 & 0.70 & 2.35 \\
\hline
\end{tabular}
\end{table}

\endgroup

\section{Discussion and Conclusion}\label{sec-discussion}

\ac{ROSA} is best understood as an \ac{IS2}-amplifying representation
for graph change analysis. Its role is not to replace the base metric
with a universally stronger one, but to convert weak, localized graph
edits into a more stable and more observable graph-comparison signal by
integrating the base distance along a filtration. In this paper we aimed
to map explicitly where and why the present instantiation works and
where it fails; the operative sufficient conditions are linked closely
to geometric signature, ranking heuristic, and noise propagation (in the
vertex noise models). For the \(W_2\)+\acs{JCP}+\ac{CTPL} instantiation,
\ac{ROSA} can achieve amplification ratios \(\approx 2\) across
synthetic and real biological graphs, while the sensitivity analysis
shows that the operative assumptions of Proposition
\ref{thm:snr-guarantee} match this amplification with the correct
direction and approximate scale. This framing also clarifies the role of
the new trajectory results. Trajectory monitoring is not a separate
mechanism from amplification; it is a downstream consequence of having a
cleaner signal. When only noisy pairwise graph comparisons are
available, direct spectral distance can make diminishing or alternating
edit-magnitude patterns too unstable to monitor, whereas \ac{ROSA} can
preserve enough structure for those patterns to remain trackable. The
trajectory experiments therefore answer the practical question
motivating the paper: if the goal is to track change in noisy
graph-structured data, does amplification produce a signal that is
easier to monitor over time? For the informative nonlinear and
oscillatory regimes considered here, the answer is yes. Likewise, the
failure of \ac{ROSA} with persistent homology in our experiments is a
scope boundary rather than a contradiction: the filtration and heuristic
used here are designed for spectral sensitivity, not for topological
event tracking. On the theoretical side, the analysis still relies on
empirically validated assumptions rather than a full first-principles
derivation for the present instantiation, and closed-form links between
ordering uncertainty, variance, and amplification remain open.
Interesting future directions can explore trainable local operators that
induce more stable filtrations. From a practical perspective, \ac{ROSA}
is most valuable in the low-signal regime where direct graph distances
are least informative. Its operator-on-distances design makes it
compatible with workflows built around graph dissimilarity measures, and
the computational burden can be reduced substantially through stride,
truncation, incremental eigensolvers for spectral distances, and
process-level parallelism for the more generic case. Taken together, the
theory and experiments support \ac{ROSA} as a useful amplification
operator for monitoring weak structural change in noisy graphs, while
making clear that its gains are conditional, diagnosable, and
application-dependent rather than universal.

\section*{Acknowledgments}

The mitochondrial skeleton geometry used in Figure \ref{fig:real-bio-dynamics} was derived from the MitoTNT test dataset \cite{wang2023mitotnt}, available at \url{https://github.com/pylattice/MitoTNT} under the BSD-3-Clause license.
The epithelial junction geometry was derived from the TissueMiner demo database \cite{etournay2016tissueminer}, available at \url{https://github.com/mpicbg-scicomp/tissue_miner} under the BSD-3-Clause license.
The protein conformer structure (\acs{PDB} 1D3Z) was obtained from the RCSB \acl{PDB} at \url{https://www.rcsb.org/structure/1D3Z} under the CC0~1.0 public domain dedication; the underlying NMR structure was originally reported by Cornilescu et al.~\cite{cornilescu1998validation}.
The fMRI connectome data in Figure \ref{fig:fmri-boundary} were obtained from the ABIDE~I Preprocessed initiative \cite{dimartino2014abide}, available at \url{https://fcon_1000.projects.nitrc.org/indi/abide/} under the CC~BY-NC-SA~3.0 license.
The retinal fundus images and manual vessel annotations used in Supplementary Figure \ref{fig:retina-multiplex-jcp} were obtained from the TREND database \cite{popovic2021trenddata}, available from Zenodo under the CC BY 4.0 license at \url{https://doi.org/10.5281/zenodo.4521044}.

F.S.\ and B.C.\ acknowledge funding from a UKRI Future Leaders Fellowship MR/T043571/1.

\textbf{AI usage declaration.} During preparation of this manuscript,
the authors used generative AI tools to assist with software
implementation and review, including boilerplate code and plotting
scripts; exploration, discussion, and critical checking of mathematical
arguments and proofs; experiment-planning and manuscript-review tasks;
and manuscript production, including Quarto configuration, formatting,
and language editing. All AI-assisted material incorporated into the work
was critically evaluated and independently verified by the authors and,
where applicable, computationally tested and revised. The research
questions, selection of methods and experiments, interpretation of
results, and final editorial decisions remained with the human authors.
The authors assume responsibility for all content.

% Force bibliography here, before supplementary material
\bibliography{references.bib}
% Prevent Quarto from emitting a duplicate bibliography at end of document
\renewcommand{\bibliography}[1]{}

\clearpage
\mbox{}\thispagestyle{empty}
\clearpage

\setcounter{section}{0}
\counterwithout{table}{section}
\counterwithout{figure}{section}
\counterwithout{equation}{section}
\setcounter{table}{0}
\setcounter{figure}{0}
\setcounter{equation}{0}
\renewcommand{\thesection}{S\arabic{section}}
\renewcommand{\thetable}{S\arabic{table}}
\renewcommand{\thefigure}{S\arabic{figure}}
\renewcommand{\theequation}{S\arabic{equation}}

\section*{Supplementary Material}
\addcontentsline{toc}{section}{Supplementary Material}

\section{Notation}

{\footnotesize
\begin{longtable}{lll}
\caption{Notation summary for the main text and supplementary algorithms.\label{tbl:notation-summary}}\\
\hline
\textbf{Symbol} & \textbf{Description} & \textbf{Defined in} \\
\hline
$G=(V,E)$ & Graph with vertex set $V$ and edge set $E$ & \S\ref{sec-filtration} \\
$L = D - A$ & Combinatorial Laplacian ($A$ = weighted adjacency) & \S\ref{sec-filtration} \\
$\langle X,Y\rangle_F,\ \|X\|_F$ & Frobenius inner product and norm & \S\ref{sec-jacobian} \\
$d(G,H)$ & Base graph distance & \S\ref{sec-filtration} \\
$d_{\mathrm{SD}}(G,G')$ & Wasserstein-2 spectral distance & \S\ref{sec-filtration} \\
$f$ & Edge-ranking heuristic & \S\ref{sec-filtration} \\
$\rho_e$ & Jacobian-coherence score for edge $e$ & \S\ref{sec-jacobian} \\
$\mathcal{A}_f$ & ROSA amplification operator, $d \mapsto d^{\mathrm{ROSA}}_f$ & \S\ref{sec-filtration} \\
$d^{\mathrm{ROSA}}_f(G,H)$ & ROSA-C amplified distance, $\sum_k d(G_k, H_k)$ & \S\ref{sec-filtration} \\
$m$ & Filtration depth (number of edge removals) & \S\ref{sec-filtration} \\
$s$ & Stride parameter for strided filtration ($s=1$: full) & \S\ref{sec-complexity} \\
$\pi^X$ & Edge permutation induced by $f$ on graph $X$ & \S\ref{sec-snr-amplification} \\
$d_K(\pi,\rho)$ & Kendall--tau distance between permutations & \S\ref{sec-snr-amplification} \\
$\Delta_f^{\mathrm{signal}}$ & Signal-induced filtration-path divergence & \S\ref{sec-snr-amplification} \\
$\Delta_f^{\mathrm{noise}}$ & Noise-induced filtration-path divergence & \S\ref{sec-snr-amplification} \\
$\mathrm{IS2}_d$ & Inverse coefficient of variation stability score of metric $d$ & \S\ref{sec-snr-amplification} \\
$R$ & IS2 amplification ratio, $\mathrm{IS2}_{\mathrm{ROSA}} / \mathrm{IS2}_{SD}$ & \S\ref{sec-results} \\
$c$ & \ac{CTPL} scale parameter (power-law regime) & \S\ref{sec-setup} \\
$\lambda$ & \ac{CTPL} exponential tempering parameter & \S\ref{sec-setup} \\
$r_{\max}$ & Maximum admissible vertex displacement & \S\ref{sec-setup} \\
$K$ & Number of noise realizations & \S\ref{sec-setup} \\
\hline
\end{longtable}
}

\section{Experiment Configuration}

{\footnotesize
\setlength{\tabcolsep}{3pt}
\begin{longtable}{@{}>{\raggedright\arraybackslash}p{0.16\textwidth} >{\raggedright\arraybackslash}p{0.21\textwidth} >{\raggedright\arraybackslash}p{0.22\textwidth} >{\raggedright\arraybackslash}p{0.17\textwidth} >{\raggedright\arraybackslash}p{0.16\textwidth}@{}}
\caption{Compact summary of the main experimental configurations used in the manuscript. Here $\Delta x$ denotes the imposed edit magnitude, $K$ the number of noise realizations, $L$ the retained filtration length when truncation is used, and $s$ the stride. CTPL tempering is auto-calibrated per $c$ to satisfy $\Pr[R > r_{\max}] \le 0.05$. Reproduction scripts are available at \url{https://github.com/systems-mechanobiology/ROSA/tree/main/reproduction}.}\label{tbl:experiment-config}\\
\hline
\textbf{Block} & \textbf{Graph / substrate} & \textbf{Signal sweep} & \textbf{Noise / calibration} & \textbf{Evaluation settings} \\
\hline
\endfirsthead
\hline
\textbf{Block} & \textbf{Graph / substrate} & \textbf{Signal sweep} & \textbf{Noise / calibration} & \textbf{Evaluation settings} \\
\hline
\endhead
Heatmaps (Figure~\ref{fig:heatmaps}, Row 1) &
MSTA baseline, $N=30$, $E=34$, seed $42$, $d_{\min}=8$ &
$\Delta x \in \{0.5,1,2,5,10,25\}$ across bridge, leaf, and small-Fiedler-magnitude hub edit sites &
CTPL vertex noise with $c \in \{0.5,1,1.5,2,3\}$; $r_{\max}=0.45\,w_{\min}$ &
$K=100$; full ROSA-C; JCP ranking; empirical spectral $W_2$ on combinatorial-Laplacian spectra \\

Noise robustness (Figure~\ref{fig:heatmaps}, Row 2) &
Same MSTA baseline and bridge-adjacent edit &
$\Delta x/\sigma_{\mathrm{gauss}} \in \{0.5,1,2,5,10,25\}$ &
CTPL, Gaussian vertex, and Gaussian edge noise with multiplier $m_{\mathrm{noise}} \in \{0.25,0.5,1,2,4\}$ &
$K=100$; full ROSA-C; JCP ranking; same base distance and bootstrap-support criterion \\

Graph density sweep (Figure~\ref{fig:graph-diversity}) &
Same 30-point cloud; MST to full Delaunay densification, $E \in \{29,\dots,78\}$ &
$\Delta x \in \{0.5,1,2,5,10,25\}$; three spatial edit classes fixed from the MSTA baseline &
CTPL with fixed $c=1.5$ &
Row 1: full filtration. Row 2: fixed-step control with $L=27$ retained steps. $K=100$ \\

Operator comparison (Figure~\ref{fig:operator-comparison}) &
MSTA baseline, bridge-adjacent edit &
$\Delta x \in \{0.5,1,2,5,10,25\}$ &
CTPL with $c \in \{0.5,1,1.5,2,3\}$ &
$K=100$; JCP ordering; operators: spectral $W_2$, NetLSD, resistance, PH($H_0$) \\

Scaling / incremental panels (Figure~\ref{fig:scaling}) &
Random geometric graphs with density $\approx 2N$, $N \in \{128,256,512,1024\}$ &
Panel B: fixed $L=64$; Panels C--D: stride $s \in \{1,2,4,8\}$ &
CTPL benchmark setting reused from the main synthetic setup &
Panel A: Julia thread sweep $\{1,2,4,8\}$; Panel B incremental eigensolver at $N=1024$; Panels C--D compare $R$ and runtime versus stride \\

fMRI signal-recovery limits (Figure~\ref{fig:fmri-boundary}) &
ABIDE Preprocessed (YALE; C-PAC, no-global, HO atlas); complete weighted connectomes from affine-transformed ROI correlations, plus a symmetric $k$NN+MST representative graph with $k=8$ &
Raw representative-pair probe across four heuristics; sparse-graph support sweep with edited support sizes 1, 2, 4, 8, 16, 32, 64, 128, $E/2$, and $E$ &
Representative-pair probe: iid Gaussian edge noise with scale tied to the remaining headroom to $1$; support sweep: local Gaussian multiplicative edge edits with mean $0.01$ plus mild external Gaussian edge noise &
Representative-pair heuristic probe: $K=30$, $L=40$; support sweep: JCP only, $K=30$, $L=40$; reports $R$, confidence intervals, and edited-support connectivity diagnostics \\

Real biological dynamics (Figure~\ref{fig:real-bio-dynamics}, Row A) &
Mito graph ($N=38$, $E=46$) and cell junction graph ($N=34$, $E=43$) &
Mito: $22 \rightarrow 31$, $D_{\max}=5$; Cell: $15 \rightarrow 14$, $D_{\max}=4$; schedules = seesaw, accelerating, oscillating &
CTPL with $c=1.5$; noise limit fixed at $w_{\min}/3$ &
Pairwise comparisons $d(G_t,G_{t+1})$; $K=50$; geometry audited to preserve minimum edge length \\

Fixed-topology conformer boundary (Figure~\ref{fig:real-bio-dynamics}, Row B) &
Ubiquitin NMR conformers (PDB 1D3Z), C$\alpha$ contact graphs; single-layer core graph with $N=76$, $E=242$; layered two-copy construction for selected pairwise probes &
Single-layer: reference conformer 1 versus conformers $2\ldots 10$. Layered: strongest conformer pair with patterns \texttt{xy vs xx}, \texttt{yx vs yy}, \texttt{xx vs yy}, \texttt{xy vs yx} &
Gaussian vertex-coordinate noise with graph-calibrated base scale $\sigma_{\mathrm{gauss}}$, swept by multipliers $0.25$, $0.5$, $0.75$, $1.0$, $1.25$, and $1.5$ &
Single-layer: $K=50$, $L=50$, reports per-pair $R$, \ac{IS2}, mean/std, and bootstrap-support counts. Layered: $K=50$, $L=50$; fixed-$\sigma=0.5$ pattern panel plus layered boundary sweep on the strongest pair \\

Layered mito extension (Figure~\ref{fig:real-bio-dynamics}, Rows C--D) &
Mito substrate lifted to $k \in \{2,3,4\}$ layers &
Patterns = ramp and pulse; displacement sweep $d \in \{0.5,1.0,1.25\}$ &
Fixed-$d$ noise-shape sweep with $c \in \{1,3,5\}$; fixed noise limit $w_{\min}/3$ &
$K=50$; layer spans $k \in \{2,3,4\}$; full-layer graph with capped retained filtration horizon tied to the base-edge count \\

Sensitivity table (Table~\ref{tab:sensitivity}) &
MSTA baseline &
Same $\Delta x \times c$ grid as Figure~\ref{fig:heatmaps}, Row 1 &
CTPL with auto-calibrated tempering &
$K=100$; reports $M$, $\gamma$, $\beta$, A3 status, $R_{\mathrm{pred}}$, $R_{\mathrm{obs}}$ \\

Stochastic ROSA (Tables~\ref{tab:stochastic-summary} and \ref{tab:stochastic}) &
MSTA baseline, fixed bridge-adjacent edit &
Fixed $\Delta x = 2$; adjacent-swap probability $p \in \{0,0.05,\ldots,0.7\}$ &
CTPL with representative $c \in \{0.5,1.5,3.0\}$ &
$K=100$ noise draws; $n_{\mathrm{trials}}=30$ stochastic orderings per draw; reports deterministic and stochastic $R$ \\
\hline
\end{longtable}

\noindent\emph{Hardware note.} All runtime and scaling measurements reported in the manuscript were obtained on a local Apple MacBook Pro with an M4 Pro processor.

}

\section{A Geometric Obstruction to ROSA-C}

\begin{proposition}[Weighted-isomorphism obstruction for ROSA-C]
Let $G$ and $H$ be weighted graphs on the same number of vertices, let $d_{\mathrm{SD}}$ be any graph distance that depends only on the Laplacian spectrum, and let $f$ be an automorphism-equivariant edge-ordering heuristic such as \ac{JCP}. If there exists a vertex permutation matrix $P$ such that
\[
L_H = P^\top L_G P,
\]
then
\[
d_{\mathrm{SD}}(G,H)=0
\qquad\text{and}\qquad
d_f^{\mathrm{ROSA}}(G,H)=0.
\]
In particular, the operative amplification conditions of Proposition \ref{thm:snr-guarantee} do not hold for this pair.
\end{proposition}

\begin{proof}
If $L_H = P^\top L_G P$, then $L_G$ and $L_H$ are permutation-similar and therefore isospectral. Since $d_{\mathrm{SD}}$ depends only on the Laplacian spectrum, it follows that $d_{\mathrm{SD}}(G,H)=0$.

Because $f$ is automorphism-equivariant, the edge scores on $G$ and $H$ are carried into each other by the same permutation, so the filtered graphs $G_k$ and $H_k$ remain permutation-similar at every filtration step $k$. Hence $d_{\mathrm{SD}}(G_k,H_k)=0$ for all $k$, and summing over the trajectory gives $d_f^{\mathrm{ROSA}}(G,H)=\sum_{k=0}^{m} d_{\mathrm{SD}}(G_k,H_k)=0$.
Therefore no positive cumulative mean shift can occur for this pair, and the amplification assumptions fail in this obstruction regime.
\end{proof}

\noindent\textbf{Worked example: Mirror-shifted hub on a square.} Place four leaves at the corners of a square and connect them to a hub by weighted spokes, with weights given by Euclidean edge lengths. Let $G$ be the weighted star obtained by shifting the hub horizontally by $+\delta$, and let $H$ be obtained by shifting it by $-\delta$. The two left spokes in $G$ have the same length, the two right spokes in $G$ have the same length, and the corresponding lengths in $H$ are exchanged by swapping the left and right leaf pairs. Thus $G$ and $H$ are isomorphic as weighted graphs, so the weighted-isomorphism obstruction above gives $d_{\mathrm{SD}}(G,H)=d_f^{\mathrm{ROSA}}(G,H)=0$ despite the nontrivial geometric edit.

\section{Sharp Amplification Boundary on a Weighted Star}

The obstruction above shows that amplification can vanish.
We now show that the \emph{same} graph family also exhibits a sharp boundary between observable trajectory divergence and complete obstruction, controlled entirely by the symmetry structure of the edge weights.
The stronger amplification condition (\ref{a:filtration-amplifies}) is then assessed numerically on this motif via the normalized margin $\gamma := d_f^{\mathrm{ROSA}}/((M{+}1)d_{\mathrm{SD}})-1$, using the worked example and the hub-position scan in Figure \ref{fig:sharp-boundary-heatmap}.

\paragraph{Setup}
Let $S_4$ be the weighted star on five vertices: a hub $v_0=(h_x,h_y)$ connected by spokes to four leaves at the corners of the unit square, $v_1=(0,0)$, $v_2=(1,0)$, $v_3=(1,1)$, $v_4=(0,1)$.
Spoke weights are $w_i = 1/\|v_0 - v_i\|$, so the weighted combinatorial Laplacian is
\[
L(w_1,w_2,w_3,w_4) = \begin{pmatrix}
\textstyle\sum_i w_i & -w_1 & -w_2 & -w_3 & -w_4 \\
-w_1 & w_1 & 0 & 0 & 0 \\
-w_2 & 0 & w_2 & 0 & 0 \\
-w_3 & 0 & 0 & w_3 & 0 \\
-w_4 & 0 & 0 & 0 & w_4
\end{pmatrix}.
\]
The eigenvalue $\lambda=0$ always appears (connected graph); the four nonzero eigenvalues are the roots of the secular equation
\[
1 - \sum_{i=1}^{4} \frac{w_i}{\lambda - w_i} = 0,
\]
equivalently the quartic $q(\lambda) = \sum_{k=0}^{4} a_k \lambda^k$ with coefficients
\[
a_0 = -5\,e_4,\quad a_1 = 4\,e_3,\quad a_2 = -3\,e_2,\quad a_3 = 2\,e_1,\quad a_4 = -1,
\]
where $e_j = e_j(w_1,w_2,w_3,w_4)$ is the $j$-th elementary symmetric polynomial.
These coefficients are verified by symbolic computation (\texttt{Symbolics.jl}; script \texttt{reproduction/theorem\_sharp\_boundary/symbolic\_star.jl}).

\begin{theorem}[Sharp boundary on a weighted star]\label{thm:sharp-boundary}
Let $S_4$ be the five-vertex weighted star above with spectral distance $d_{\mathrm{SD}} = W_2$ and automorphism-equivariant heuristic $f$ (e.g.\ JCP).
\begin{enumerate}[label=\textup{(\roman*)},ref=\roman*,leftmargin=*,itemsep=6pt]

\item\label{case:observable} \textbf{Observable (generic hub position).}
If the hub position $v_0=(h_x,h_y)$ lies on no perpendicular bisector of two square-corner leaves, equivalently
\[
h_x \neq \tfrac12,\qquad
h_y \neq \tfrac12,\qquad
h_x \neq h_y,\qquad
h_x+h_y \neq 1,
\]
then the four spoke weights are pairwise distinct, $f$ induces a strict ordering on the edges of both $G$ and $H$, and for any sufficiently small generic displacement of $v_0$, the filtration trajectories of $G$ and $H$ diverge.
The empirical hub-position scan in Figure \ref{fig:sharp-boundary-heatmap} shows both positive and negative regions under the implemented descending-\ac{JCP} removal order; positive $\gamma$ is therefore a sampled numerical regime, not a generic consequence of asymmetry alone.
The three cases are illustrated in Figure \ref{fig:sharp-boundary-setup}, and the corresponding empirical $\gamma$ landscape over hub positions appears in Figure \ref{fig:sharp-boundary-heatmap}.

\item\label{case:stealth} \textbf{Stealth (axis-aligned mirror edit).}
If $v_0^{(G)} = (\tfrac{1}{2}+\delta,\, \tfrac{1}{2})$ and $v_0^{(H)} = (\tfrac{1}{2}-\delta,\, \tfrac{1}{2})$ for any $\delta > 0$, then $L_H = P^\top L_G P$ where $P$ is the permutation matrix for $\pi=(1\;2)(3\;4)$, and $d_f^{\mathrm{ROSA}}(G,H)=0$ identically.
Condition \textup{(\ref{a:filtration-amplifies})} fails: $\gamma = 0$.

\item\label{case:single-edge} \textbf{First-removed single-edge edit on a tree.}
If $G$ and $H$ differ in exactly one spoke weight, both are trees (no cycles), and the edited spoke is the first edge removed by the shared strict ordering, then the filtration signal is confined to the initial step. More generally, once the edited bridge has been removed from both graphs, the remaining filtered graphs are identical. In the first-removed case this gives $\gamma = d_f^{\mathrm{ROSA}}/((M{+}1)d_{\mathrm{SD}}) - 1 = -M/(M{+}1)$.
Condition \textup{(\ref{a:filtration-amplifies})} fails: $\gamma < 0$.

\end{enumerate}
\end{theorem}

\begin{proof}
\textbf{Case~(\ref{case:observable}).}
The equality of two spoke lengths occurs exactly when $v_0$ lies on the perpendicular bisector of the corresponding two leaves. For the four unit-square leaves, these bisectors are
\[
h_x=\tfrac12,\qquad h_y=\tfrac12,\qquad h_x=h_y,\qquad h_x+h_y=1.
\]
Therefore, if $v_0$ avoids these four lines, the squared distances
\[
d_i^2 = (h_x - x_i)^2 + (h_y - y_i)^2
\]
are pairwise distinct, and hence so are the weights $w_i = 1/\sqrt{d_i^2}$. By the explicit weighted-star formula for the \ac{JCP} scores given above, the strict spoke-weight ordering induces a strict \ac{JCP} ordering as well.
A sufficiently small generic perturbation $H$ of the hub position changes all four weights while preserving that strict ordering, so the filtration trajectories $(G_k)$ and $(H_k)$ are generated by the same edge-removal sequence.
Since the secular equation roots move continuously with the weights, each nonzero eigenvalue of $G_k$ differs from that of $H_k$, giving $W_2(G_k,H_k)>0$ for steps $k < M$ where at least one edited spoke remains.
Because the hub edit changes all four weights, multiple filtration steps contribute positive $W_2$, which proves qualitative trajectory divergence.
The stronger amplification inequality $\sum_{k=1}^{M} W_2(G_k,H_k) > M\,W_2(G,H)$ is evaluated numerically rather than derived symbolically here: the displayed worked instance (hub at $(0.3,0.4)$ vs $(0.35,0.45)$) gives $\gamma \approx -0.27$ under the implemented descending-\ac{JCP} removal order, while the full hub-position scan in Figure \ref{fig:sharp-boundary-heatmap} shows sampled positive regions outside the spoke-length equality loci.
The corresponding trajectory is recorded in \texttt{reproduction/theorem\_sharp\_boundary/results.md}.

\textbf{Case~(\ref{case:stealth}).}
For the mirror placement, the squared distances satisfy $d_1^2(G) = d_2^2(H)$, $d_2^2(G) = d_1^2(H)$, $d_3^2(G) = d_4^2(H)$, $d_4^2(G) = d_3^2(H)$ identically in $\delta$ (verified symbolically: all four differences expand to exact zero).
Therefore $\{w_1(G),\ldots,w_4(G)\} = \{w_{\pi(1)}(H),\ldots,w_{\pi(4)}(H)\}$ with $\pi = (1\;2)(3\;4)$, and the weighted-isomorphism obstruction (Proposition above) applies at every filtration step.

\textbf{Case~(\ref{case:single-edge}).}
On a star (or any tree), each spoke is a bridge: its removal disconnects its leaf.
After removing spoke $i$ from both $G_{k^*}$ and $H_{k^*}$, vertex $v_i$ becomes isolated in both graphs with the same zero Laplacian block.
The remaining connected component is a star on $\{v_0\} \cup \{v_j : j \ne i\}$ with identical weights in $G$ and $H$ (since only spoke $i$ was edited).
Hence $L_{G_k} = L_{H_k}$ for all $k > k^*$, giving $W_2(G_k,H_k) = 0$.
Only steps $k \le k^*$ contribute nonzero $W_2$, and when $i$ is the first-removed, highest-\ac{JCP}-score spoke ($k^*=0$), only step~$0$ is nonzero: $d_f^{\mathrm{ROSA}} = W_2(G_0,H_0) = d_{\mathrm{SD}}(G,H)$, so $\gamma = d_f^{\mathrm{ROSA}}/((M{+}1)d_{\mathrm{SD}})-1 = -M/(M{+}1)$.
\end{proof}

\begin{remark}[Signal distribution and condition (\ref{a:filtration-amplifies})]
The sharp boundary identifies the structural mechanism underlying empirical satisfaction of (\ref{a:filtration-amplifies}): \emph{the edit must distribute spectral signal across multiple filtration steps}.
This holds naturally for vertex displacements in geometric graphs (which change all incident edge weights) and for global noise processes.
It can fail for first-removed single-edge edits on trees, where the signal is confined to one step and diluted by $1/(M{+}1)$ averaging.
The boundary between cases~(\ref{case:observable}) and~(\ref{case:stealth}) is controlled by the automorphism group $\mathrm{Aut}(G,w)$: positions off the spoke-length equality loci break the weighted-isomorphism obstruction used in the strict-ordering argument, while the displayed axis-aligned mirror placements preserve a reflection symmetry that kills it.
On this weighted-star motif, the empirical $\gamma$ landscape indicates that some sampled generic positions support amplification, while others dilute the base signal under the same removal convention.
\end{remark}

\begin{figure}[ht]
\centering
\includegraphics[width=\textwidth]{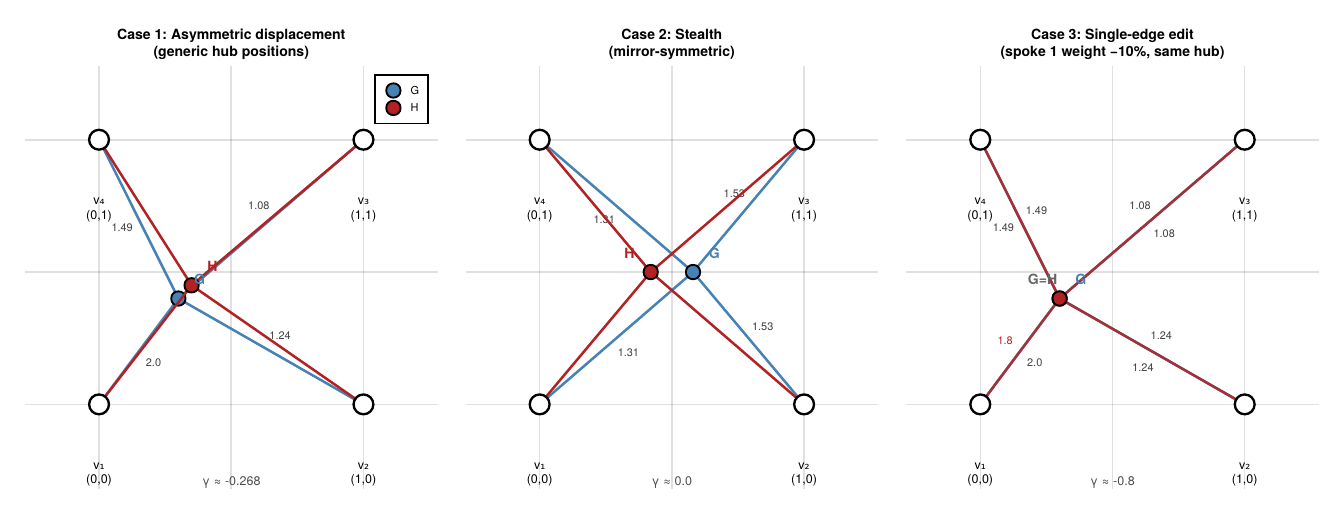}
\caption{%
  Three cases of the Sharp Boundary Theorem on a five-vertex weighted star.
  Leaf nodes $v_1$--$v_4$ are fixed at unit-square corners; hub $v_0$ lies at the marked position.
  Blue: graph $G$; red: graph $H$ (perturbed).
  Spoke weights (rounded) are annotated at edge midpoints; in Case~3 the changed weight is highlighted in red.
  \emph{Case~1 (Asymmetric displacement):} generic hub positions, all four weights shift and the filtration trajectories diverge; the displayed worked example has $\gamma \approx -0.27$ under the implemented descending-\ac{JCP} order.
  \emph{Case~2 (Stealth):} mirror-symmetric placement, weighted isomorphism at every filtration step, $\gamma = 0$.
  \emph{Case~3 (Single-edge edit):} same hub, one first-removed spoke weight reduced by $10\%$, signal confined to one filtration step, $\gamma \approx -0.8 = -M/(M{+}1)$.%
}
\label{fig:sharp-boundary-setup}
\end{figure}

\begin{figure}[ht]
\centering
\includegraphics[width=0.62\textwidth]{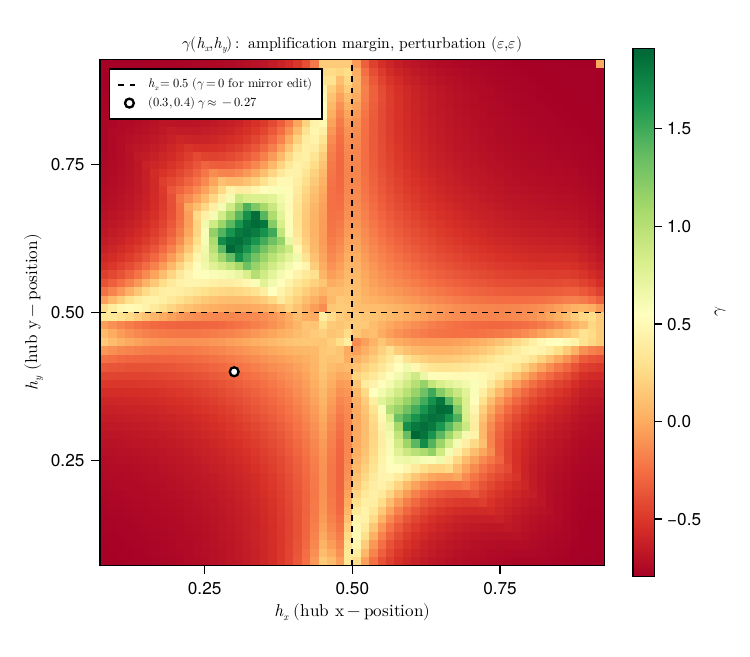}
\caption{%
  Amplification parameter $\gamma(h_x, h_y)$ over hub position for the five-vertex weighted star
  (perturbation $H$: hub shifted by $(\varepsilon,\varepsilon)$ with $\varepsilon=0.05$).
  Green: $\gamma > 0$ (ROSA amplifies); pale yellow: $\gamma \approx 0$; red: $\gamma < 0$ (dilution).
  Dashed lines mark the displayed mirror-obstruction axes $h_x = 0.5$ and $h_y = 0.5$.
  Additional spoke-length equality loci $h_x = h_y$ and $h_x + h_y = 1$ also mark non-generic positions where the strict-ordering argument does not apply.
  The white circle marks the worked example $(0.3,\,0.4)$ discussed in the proof.%
}
\label{fig:sharp-boundary-heatmap}
\end{figure}

For a weighted complete graph with similarity weights $w_{ij}$, we first convert similarities to distances by $d_{ij}=1-w_{ij}$.
The symmetric $k$NN graph retains an undirected edge $(i,j)$ whenever $j$ lies among the $k$ nearest neighbors of $i$ or $i$ lies among the $k$ nearest neighbors of $j$ under $d_{ij}$.
We then compute a minimum spanning tree on the same distance graph and take the union of its edges with the symmetric $k$NN edges.
The resulting $k$NN$\cup$MST graph is therefore connected by construction while restoring local edge neighborhoods absent from the raw complete connectome representation.

\section{Dense-Connectome Comparison Heuristics}

For the dense-connectome boundary probe in Section \ref{sec-results-fmri-boundary}, we compare \ac{JCP} against three simple edge-ordering heuristics on a weighted graph $G=(V,E,w)$.
Let $d_i := \sum_{j:(i,j)\in E} w_{ij}$ denote the weighted degree of vertex $i$.
The comparison orderings are:
\[
\pi^{R}: \quad s_R(i,j) = R_{\mathrm{eff}}(i,j),
\]
\[
\pi^{D}: \quad s_D(i,j) = w_{ij}\sqrt{\frac{\max(d_i,d_j)}{\min(d_i,d_j)}},
\]
\[
\pi^{W}: \quad s_W(i,j) = w_{ij},
\]
where $R_{\mathrm{eff}}(i,j)$ is the effective resistance between the edge endpoints under the combinatorial Laplacian of $G$.
In all three cases, edges are ordered by decreasing score.

\begin{algorithm2e*}
\caption{Edge ordering by a comparison heuristic}
\KwIn{Weighted graph $G=(V,E,w)$, edge score function $s:E\to\mathbb{R}$}
\KwOut{Edge ordering $\pi = (e^{(1)},\dots,e^{(|E|)})$}
\ForEach{$e=(i,j)\in E$}{
    compute score $s(e)$\;
}
$\pi \leftarrow \text{sort}(E \text{ by decreasing } s(e))$\;
\Return{$\pi$}\;
\end{algorithm2e*}

For the dense-connectome comparison, $s$ is taken to be $s_R$, $s_D$, or $s_W$ above, while \ac{JCP} uses the score defined in Section \ref{sec-jacobian}.

\section{Reproducibility}

These source repositories are collected on the \href{https://github.com/systems-mechanobiology}{systems-mechanobiology GitHub organization page}. The software dependencies are resolved directly by the frozen reproduction environment rather than through Julia's General registry.

\begin{itemize}
\item \href{https://github.com/systems-mechanobiology/ROSA}{ROSA}: manuscript source, signed-off figures, appendix assets, and the `reproduction/` tree that drives the paper-facing experiments. Each figure or table is created by a single documented script, in a frozen Julia environment to maximize exact reproducibility.
\item \href{https://github.com/systems-mechanobiology/ROSA.jl}{ROSA.jl}: core ROSA algorithm, including ROSA-C, stochastic ROSA, the JCP heuristic, and the graph-distance routines used by the reproduction scripts.
\item \href{https://github.com/systems-mechanobiology/DynamicGeometricGraphs.jl}{DynamicGeometricGraphs.jl}: graph construction and graph-geometry utilities for the synthetic and biological substrates.
\item \href{https://github.com/systems-mechanobiology/CalibratedTemperedPowerLaw.jl}{CalibratedTemperedPowerLaw.jl}: calibrated tempered power-law noise model used in the perturbation experiments.
\item \href{https://github.com/systems-mechanobiology/Spectral.jl}{Spectral.jl}: sparse partial-spectrum backend used for the top-fraction benchmark.
\item Additional support packages declared in `reproduction/Project.toml` include \texttt{GraphSpectra.jl} and \texttt{SkeletonExtractor.jl}; these support specific benchmark and data-preparation paths but are not the main manuscript/reproduction repositories listed above.
\end{itemize}

This modularization avoids tying \ac{ROSA} to a single graph implementation, or buried spectral implementation.
All repositories have test coverage and documentation.

\section{Algorithms}

The ROSA algorithms are presented here for reference.
See Section \ref{sec-filtration} for ROSA-C theoretical foundations, Section \ref{sec-stochastic-rosa} for the stochastic variant, and Section \ref{sec-complexity} for computational scaling.

\begin{algorithm2e*}
\caption{ROSA-C trajectory computation (with optional stride)}
\label{alg:rosa-c}
\KwIn{Graphs $G = (V_G, E_G)$, $H = (V_H, E_H)$, edge ranking function $f : (e, X) \rightarrow \mathbb{R}_{\geq 0}$, base distance $d(\cdot,\cdot)$, integer retained-edge delimiter $0 \le q \le \min(|E_G|,|E_H|)$, stride $s \geq 1$}
\KwOut{Trajectory $\mathcal{T}$ sampled every $s$ steps and at the terminal state, where $m=\min(|E_G|,|E_H|)-q$}
$\mathcal{T} \leftarrow []$\;
$(e_H^{(1)},\dots,e_H^{(|E_H|)}) \leftarrow \text{sort}(E_H \text{ by decreasing } e \mapsto f(e, H))$\;
$(e_G^{(1)},\dots,e_G^{(|E_G|)}) \leftarrow \text{sort}(E_G \text{ by decreasing } e \mapsto f(e, G))$\;
\For{$k \leftarrow 0$ \KwTo $m$}{
    $G_k \leftarrow (V_G, E_G \setminus \{e_G^{(1)}, \dots, e_G^{(k)}\})$\;
    $H_k \leftarrow (V_H, E_H \setminus \{e_H^{(1)}, \dots, e_H^{(k)}\})$\;
    \If{$k \bmod s = 0 \;\lor\; k=m$}{
        append $d(G_k, H_k)$ to $\mathcal{T}$\;
    }
}
\Return{$\mathcal{T}$}\;
\end{algorithm2e*}

Algorithm \ref{alg:rosa-c} returns the filtered distance trajectory; the scalar ROSA-C value is the sum of this trajectory (or of all filtration steps in the unstrided case).

\begin{algorithm2e*}
\caption{Stochastic ROSA-C: randomized scalar aggregation}
\label{alg:rosa-stochastic}
\KwIn{Graphs $G$, $H$, edge ranking function $f$, similarity $s$, swap probability $p_{\mathrm{swap}} \in [0,1]$, number of ordering randomizations $B$}
\KwOut{Monte Carlo summaries of scalar ROSA-C sums and violation rates}
$\text{ordering}^{\mathrm{base}}_G \leftarrow \text{sort}(f(e, G) \text{ for } e \in E_G)$\;
$\text{ordering}^{\mathrm{base}}_H \leftarrow \text{sort}(f(e, H) \text{ for } e \in E_H)$\;
$\text{sums},\text{vrs},\text{traces} \leftarrow \text{empty lists of length } B$\;
\For{$b \leftarrow 1$ \KwTo $B$}{
    $\text{ordering}^{(b)}_G \leftarrow \text{PairwiseAdjacentSwap}(\text{ordering}^{\mathrm{base}}_G, p_{\mathrm{swap}})$\;
    $\text{ordering}^{(b)}_H \leftarrow \text{PairwiseAdjacentSwap}(\text{ordering}^{\mathrm{base}}_H, p_{\mathrm{swap}})$\;
    $(\text{sums}[b],\text{vrs}[b],\text{traces}[b]) \leftarrow \text{TrajectoryROSA}(G, H, \text{ordering}^{(b)}_G, \text{ordering}^{(b)}_H, s)$\;
}
\Return{$\operatorname{mean}(\text{sums})$, $\operatorname{std}(\text{sums})$, $\operatorname{quantile}_{0.025,0.975}(\text{sums})$, and analogous summaries for $\text{vrs}$}\;
\end{algorithm2e*}

\begin{algorithm2e*}
\caption{Parallel ROSA-C: Pairwise Edge-Mask Batching}
\label{alg:parallel-rosa}
\KwIn{Graphs $G,H$, orderings $\pi^G,\pi^H$, base distance $d$, number of batches $K$}
\KwOut{ROSA-C trajectory $(d(G_k,H_k))_{k=0}^{m}$}
\tcp{Step 1: Precompute pairwise edge masks}
$M_G \leftarrow \text{BuildEdgeMask}(G, \pi^G)$\;
$M_H \leftarrow \text{BuildEdgeMask}(H, \pi^H)$\;
$m \leftarrow \min(|\pi^G|, |\pi^H|)$\;
$\text{trajectory}[0] \leftarrow d(M_G.A_{\mathrm{init}}, M_H.A_{\mathrm{init}})$\;
$\text{batches} \leftarrow \text{Partition}(\{1,\dots,m\}, K)$\;
\BlankLine
\tcp{Step 2: Parallel execution across contiguous step blocks}
\textbf{parallel for} $\mathcal{B} \in \text{batches}$ \textbf{do} \{
    $a \leftarrow \min(\mathcal{B})$\;
    $A_G \leftarrow \text{ReconstructAdjacency}(M_G, a-1)$\;
    $A_H \leftarrow \text{ReconstructAdjacency}(M_H, a-1)$\;
    \For{$k \in \mathcal{B}$}{
        $\text{RemoveEdge}(A_G, M_G.\text{removals}[k])$\;
        $\text{RemoveEdge}(A_H, M_H.\text{removals}[k])$\;
        $\text{partial}[k] \leftarrow d(A_G, A_H)$\;
    }
\}
\Return{$(\text{trajectory}[0], \text{partial}[1], \dots, \text{partial}[m])$}\;
\end{algorithm2e*}

\section{Detailed Results Tables}\label{sec-detailed-tables}

{\footnotesize
\begin{longtable}{rr r rrr rr}
\caption{Empirical validation of the IS2 amplification framework assumptions (Proposition~\ref{thm:snr-guarantee}) across the $\Delta x \times c$ parameter grid ($K=100$, MSTA $N=30$). $M$: removal depth. $\gamma$: amplification factor. $\beta$: variance constant normalized by $\sqrt{M+1}$. A3: signed mean-shift condition holds ($\mu_{\mathrm{ROSA}} > (M+1)\mu_{\mathrm{SD}}$). $R_{\mathrm{pred}} = (1+\gamma)\sqrt{M+1}/\beta$. $R_{\mathrm{obs}}$: observed amplification ratio.}
\label{tab:sensitivity} \\
\hline
$\Delta x$ & $c$ & $M$ & $\gamma$ & $\beta$ & A3 & $R_{\mathrm{pred}}$ & $R_{\mathrm{obs}}$ \\
\hline
\endfirsthead
\hline
$\Delta x$ & $c$ & $M$ & $\gamma$ & $\beta$ & A3 & $R_{\mathrm{pred}}$ & $R_{\mathrm{obs}}$ \\
\hline
\endhead
0.5 & 0.5 & 32 & 0.509 & 3.80 & \checkmark & 2.28 & 2.28 \\
0.5 & 1.0 & 32 & 0.496 & 3.93 & \checkmark & 2.19 & 2.19 \\
0.5 & 1.5 & 32 & 0.511 & 4.12 & \checkmark & 2.11 & 2.11 \\
0.5 & 2.0 & 32 & 0.454 & 4.01 & \checkmark & 2.08 & 2.08 \\
0.5 & 3.0 & 32 & 0.331 & 3.67 & \checkmark & 2.08 & 2.08 \\
1.0 & 0.5 & 32 & 0.502 & 3.86 & \checkmark & 2.24 & 2.24 \\
1.0 & 1.0 & 32 & 0.492 & 3.84 & \checkmark & 2.23 & 2.23 \\
1.0 & 1.5 & 32 & 0.507 & 4.11 & \checkmark & 2.11 & 2.11 \\
1.0 & 2.0 & 32 & 0.453 & 4.09 & \checkmark & 2.04 & 2.04 \\
1.0 & 3.0 & 32 & 0.330 & 3.73 & \checkmark & 2.05 & 2.05 \\
2.0 & 0.5 & 32 & 0.492 & 3.99 & \checkmark & 2.15 & 2.15 \\
2.0 & 1.0 & 32 & 0.480 & 3.85 & \checkmark & 2.21 & 2.21 \\
2.0 & 1.5 & 32 & 0.493 & 4.04 & \checkmark & 2.12 & 2.12 \\
2.0 & 2.0 & 32 & 0.445 & 4.17 & \checkmark & 1.99 & 1.99 \\
2.0 & 3.0 & 32 & 0.325 & 3.79 & \checkmark & 2.01 & 2.01 \\
5.0 & 0.5 & 32 & 0.414 & 4.17 & \checkmark & 1.95 & 1.95 \\
5.0 & 1.0 & 32 & 0.418 & 3.77 & \checkmark & 2.16 & 2.16 \\
5.0 & 1.5 & 32 & 0.410 & 4.18 & \checkmark & 1.94 & 1.94 \\
5.0 & 2.0 & 32 & 0.384 & 4.32 & \checkmark & 1.84 & 1.84 \\
5.0 & 3.0 & 32 & 0.273 & 3.98 & \checkmark & 1.84 & 1.84 \\
10.0 & 0.5 & 32 & 0.085 & 3.71 & \checkmark & 1.68 & 1.68 \\
10.0 & 1.0 & 32 & 0.124 & 4.02 & \checkmark & 1.60 & 1.60 \\
10.0 & 1.5 & 32 & 0.121 & 3.64 & \checkmark & 1.77 & 1.77 \\
10.0 & 2.0 & 32 & 0.116 & 4.04 & \checkmark & 1.59 & 1.59 \\
10.0 & 3.0 & 32 & 0.065 & 3.51 & \checkmark & 1.74 & 1.74 \\
25.0 & 0.5 & 32 & -0.586 & 2.73 & --- & 0.87 & 0.87 \\
25.0 & 1.0 & 32 & -0.578 & 2.53 & --- & 0.96 & 0.96 \\
25.0 & 1.5 & 32 & -0.578 & 2.08 & --- & 1.17 & 1.17 \\
25.0 & 2.0 & 32 & -0.564 & 2.27 & --- & 1.10 & 1.10 \\
25.0 & 3.0 & 32 & -0.565 & 2.43 & --- & 1.03 & 1.03 \\
\hline
\end{longtable}

}

{\footnotesize
\begin{longtable}{rr rrr rr}
\caption{Stochastic ROSA: effect of swap probability $p$ on amplification ratio $R$ across noise regimes ($\Delta x = 2$, $K=100$, $n_{\mathrm{trials}}=30$). $R_{\mathrm{det}}$: deterministic ROSA. $R_{\mathrm{stoch}}$: stochastic ROSA. $\Delta R$: relative improvement. Ordering width: mean 95\% interval width over stochastic ordering perturbations. $\sigma_{\mathrm{red}}$: variance reduction relative to deterministic.}
\label{tab:stochastic} \\
\hline
$c$ & $p$ & $R_{\mathrm{det}}$ & $R_{\mathrm{stoch}}$ & $\Delta R$ (\%) & Ordering width & $\sigma_{\mathrm{red}}$ (\%) \\
\hline
\endfirsthead
\hline
$c$ & $p$ & $R_{\mathrm{det}}$ & $R_{\mathrm{stoch}}$ & $\Delta R$ (\%) & Ordering width & $\sigma_{\mathrm{red}}$ (\%) \\
\hline
\endhead
0.5 & 0.00 & 2.12 & 2.12 & +0.0 & 0.0000 & +0.0 \\
0.5 & 0.05 & 2.12 & 2.14 & +0.9 & 3.4957 & +0.8 \\
0.5 & 0.10 & 2.12 & 2.16 & +2.0 & 4.6033 & +1.7 \\
0.5 & 0.20 & 2.12 & 2.19 & +3.4 & 5.9047 & +2.9 \\
0.5 & 0.30 & 2.12 & 2.21 & +4.2 & 6.6341 & +3.4 \\
0.5 & 0.50 & 2.12 & 2.23 & +5.4 & 7.1080 & +4.3 \\
0.5 & 0.70 & 2.12 & 2.24 & +5.5 & 6.5504 & +4.0 \\
\hline
1.5 & 0.00 & 1.95 & 1.95 & +0.0 & 0.0000 & +0.0 \\
1.5 & 0.05 & 1.95 & 1.96 & +0.5 & 3.5011 & +0.4 \\
1.5 & 0.10 & 1.95 & 1.96 & +0.5 & 4.5256 & +0.3 \\
1.5 & 0.20 & 1.95 & 1.98 & +1.6 & 5.8693 & +1.1 \\
1.5 & 0.30 & 1.95 & 1.98 & +1.6 & 6.3840 & +0.9 \\
1.5 & 0.50 & 1.95 & 1.99 & +2.1 & 6.9278 & +0.9 \\
1.5 & 0.70 & 1.95 & 2.00 & +2.4 & 6.5624 & +0.7 \\
\hline
3.0 & 0.00 & 2.11 & 2.11 & +0.0 & 0.0000 & +0.0 \\
3.0 & 0.05 & 2.11 & 2.14 & +1.3 & 3.2463 & +1.2 \\
3.0 & 0.10 & 2.11 & 2.16 & +2.4 & 4.2214 & +2.2 \\
3.0 & 0.20 & 2.11 & 2.19 & +4.1 & 5.5390 & +3.6 \\
3.0 & 0.30 & 2.11 & 2.24 & +6.2 & 6.0517 & +5.3 \\
3.0 & 0.50 & 2.11 & 2.30 & +9.3 & 6.4409 & +7.8 \\
3.0 & 0.70 & 2.11 & 2.35 & +11.6 & 6.0452 & +9.4 \\
\hline
\end{longtable}

}

{\footnotesize
\begin{table}[htbp]
\centering
\caption{Worked JCP examples illustrating a sparse/localized positive case and a dense/uniform boundary case.}
\label{tab:jcp-worked-example}
\begin{tabular}{lllr}
\hline
Graph & Edge class & JCP score & Rank pattern \\
\hline
$K_5$ complete & any edge & 0.5000 & all equal \\
N=5 bow-tie & hub-leaf & 0.6030 & above triangle-closing \\
N=5 bow-tie & triangle-closing & 0.5477 & below hub edges \\
N=6 barbell & bridge & 0.5774 & below triangle edges \\
N=6 barbell & triangle-closing & 0.6255 & above bridge \\
\hline
\end{tabular}
\\[0.5em]
\begin{tabular}{rrrr}
\hline
$\sigma$ & $P(\mathrm{top\ is\ hub})$ & $P(\mathrm{all\ hub}>\mathrm{leaf})$ & mean gap \\
\hline
0.02 & 1.000000 & 1.000000 & 0.055319 \\
0.05 & 0.999750 & 0.978150 & 0.055085 \\
0.10 & 0.973850 & 0.685700 & 0.054713 \\
0.20 & 0.865700 & 0.368650 & 0.053465 \\
0.40 & 0.757150 & 0.229300 & 0.048651 \\
\hline
\end{tabular}
\end{table}

}

\begin{figure}[htbp]
\centering
\includegraphics[width=0.72\textwidth]{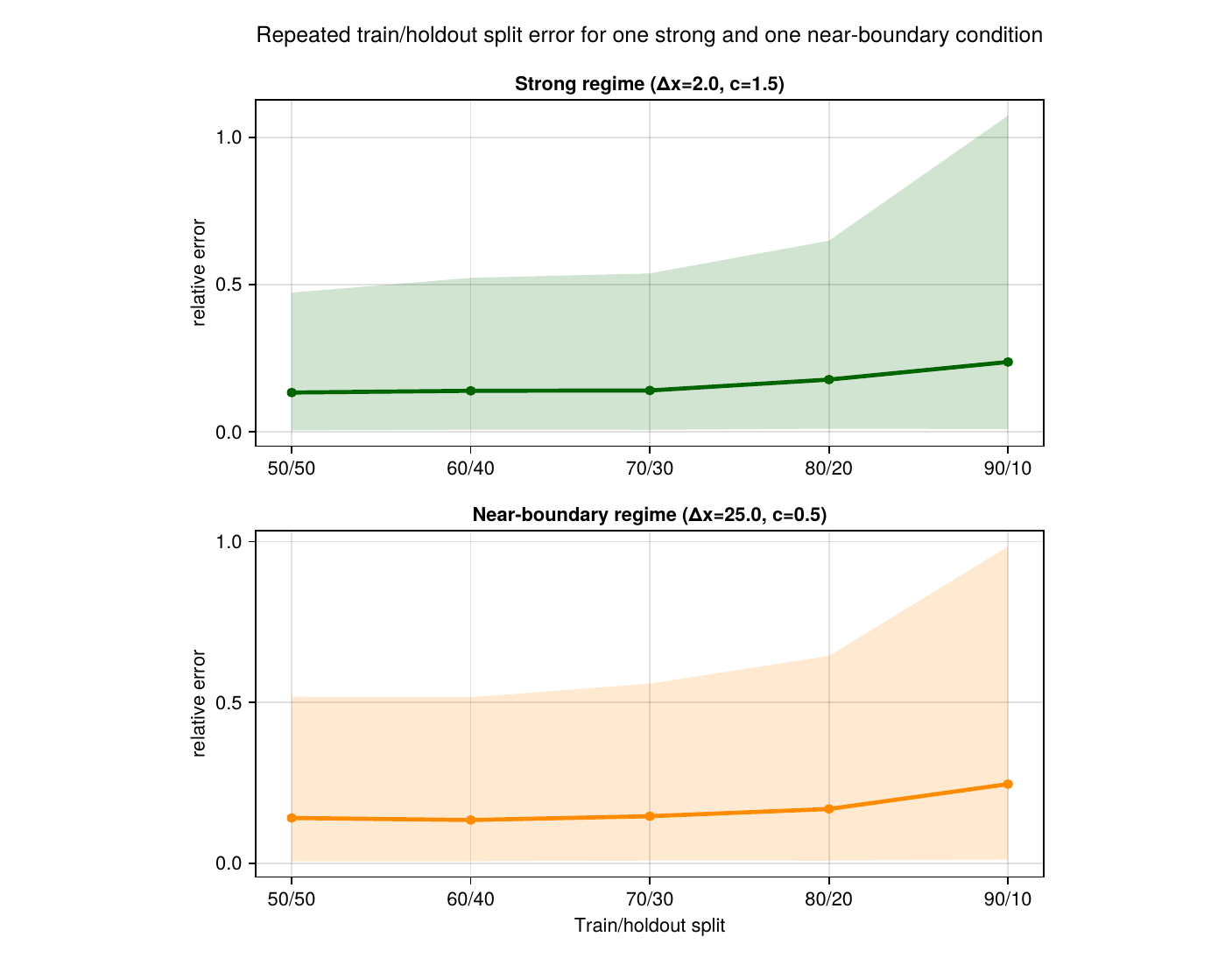}
\caption{Repeated measure/holdout split error for the IS2 consistency check in Section \ref{sec-results-sensitivity}. The strong regime ($\Delta x=2.0$, $c=1.5$) and near-boundary regime ($\Delta x=25.0$, $c=0.5$) show how the relative prediction error varies across train/holdout splits from 50/50 to 90/10. Shaded bands denote bootstrap 95\% intervals over repeated random splits.}
\label{fig:is2-holdout}
\end{figure}

\begin{figure}[htbp]
\centering
\includegraphics[width=0.65\textwidth]{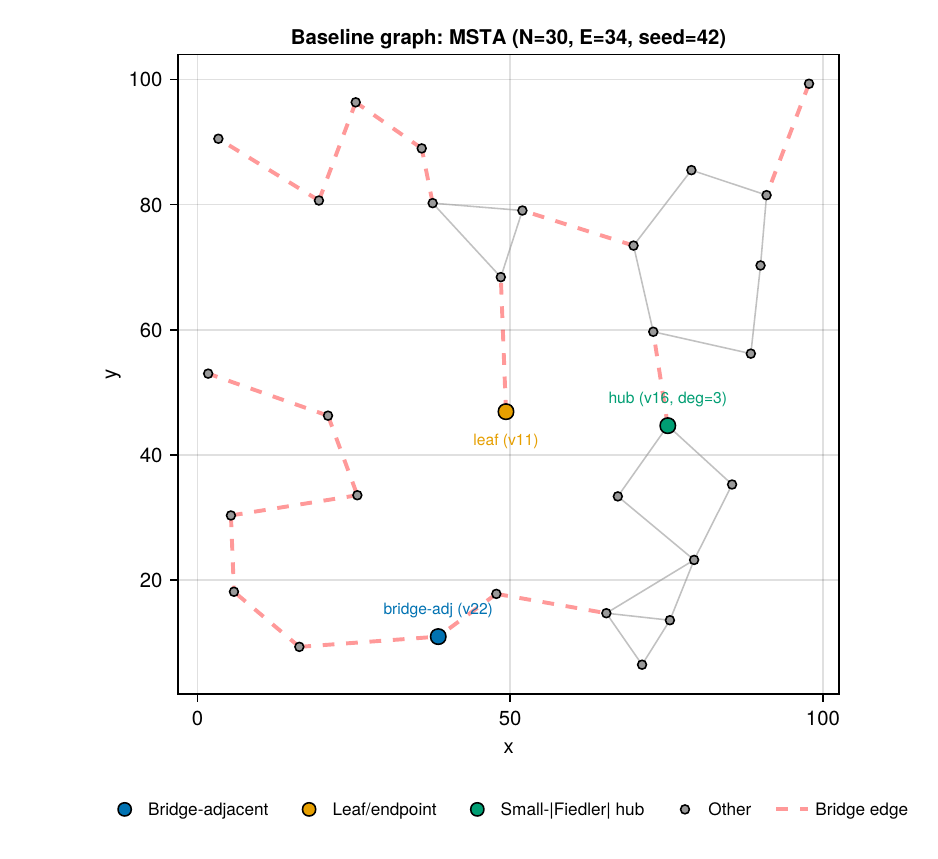}
\caption{The MSTA baseline graph (MST augmented with 5 shortest non-tree edges; $N=30$, $E=34$, $\mathrm{seed}=42$, $d_{\min}=8.0$) used in all synthetic experiments (Sections \ref{sec-results-heatmaps}--\ref{sec-results-stochastic}). Vertex edit sites are highlighted by selection rule: bridge-adjacent/Fiedler-gradient proxy (blue), leaf/endpoint (orange), and small-Fiedler-magnitude hub (green). Dashed red edges are bridges (edges whose removal disconnects the graph). These three sites are selected algorithmically; for real biological graphs (Section \ref{sec-results-real-geometry}), edit vertices are chosen manually based on biological plausibility.}
\label{fig:mst-sites}
\end{figure}

\begin{figure}[htbp]
\centering
\begin{minipage}{0.48\textwidth}
\centering
\includegraphics[width=\textwidth]{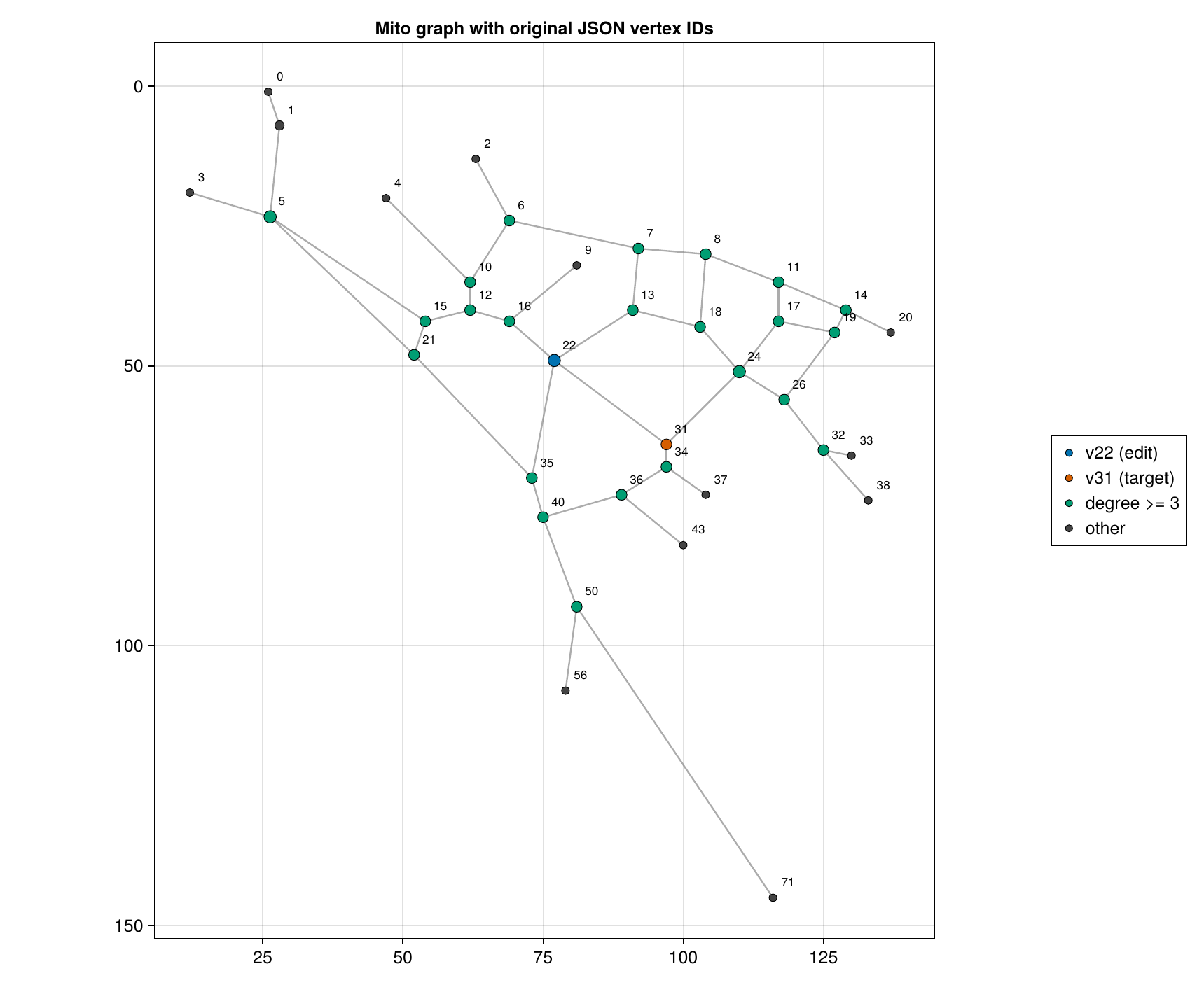}
\end{minipage}\hfill
\begin{minipage}{0.48\textwidth}
\centering
\includegraphics[width=\textwidth]{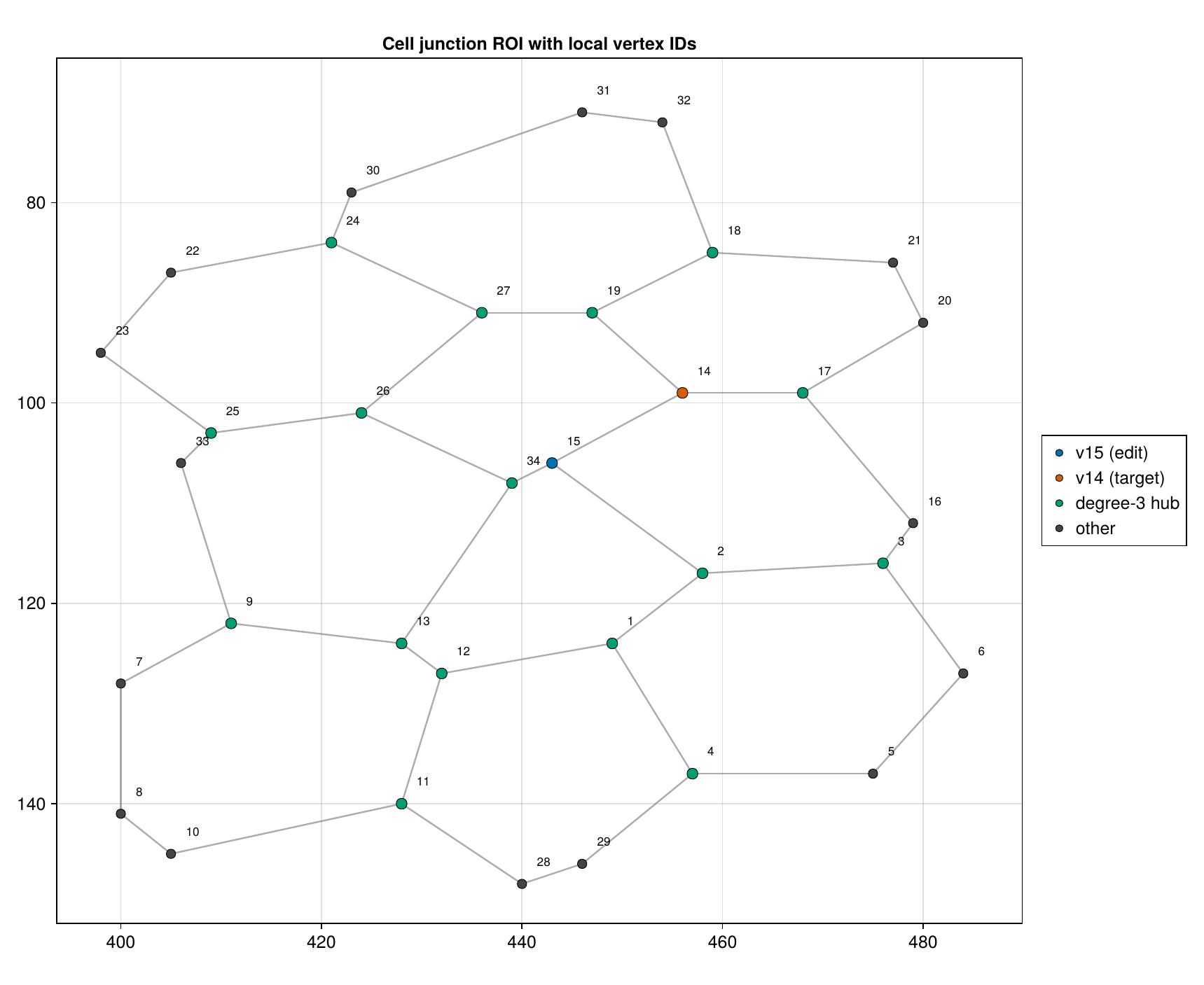}
\end{minipage}
\caption{Real-biological substrate setups used in Section \ref{sec-results-real-geometry}. Left: mitochondrial skeleton with the edited hub $v_{22}$ and target $v_{31}$ highlighted. Right: TissueMiner-derived junction ROI with edited junction $v_{15}$ and target $v_{14}$ highlighted. These are the geometry-level edit sites used for the controlled pairwise trajectories in Figure \ref{fig:real-bio-dynamics}.}
\label{fig:real-bio-setup}
\end{figure}

\subsection{Retinal Multiplex Construction and Neighbourhood Radius}
In this experiment we test how \ac{ROSA} can work on multiplex graphs where the edges per layer are not from the same metric. 
In previous multiplex scenarios the edges per layer were all coming from spatial distances or similarity, but never a mixture. 
Here we use retinal vasculature to check whether a stored-edge multiplex representation can expose edge-attribute patterns through skeleton path length and a caliber-derived flow proxy.
Second, we use this to test the effect of $k$ in \ac{JCP}.
For multiplex graphs $k$ can potentially unlock cross-layer neighbourhoods, combining patterns across layers in a way that $k=1$ cannot.

We use all ten good-quality healthy-older image/mask pairs from the TREND retinal fundus database \cite{popovic2021trenddata} as an inspected real-data construction check. The supplied manual binary vessel annotations are skeletonized in Julia with \texttt{SkeletonExtractor.jl}; TREND supplies the fundus images and annotations, not skeleton graphs. The extracted graphs contain 167--277 vertices and 182--309 edges. Each traced branch carries two raw edge measurements: its skeleton-path length $L_e$ in pixels and a caliber proxy $C_e=r_e^4$, where $r_e$ is the median vessel-mask distance-transform value over branch-interior skeleton pixels after excluding the two pixels nearest each incident vertex. The quantity $C_e$ is width-derived and flow-relevant under the Poiseuille scaling, but it is neither an observed flow measurement nor the full conductance $r_e^4/L_e$.

The multiplex graph contains two copies of the extracted topology, one weighted by $C_e$ and one by $L_e$, together with identity edges joining corresponding vertices. For each layer, raw weights are standardized using the mean and standard deviation fitted to the unedited graph. A single common translation then makes the minimum baseline intra-layer weight $10^{-2}$, and each identity edge is assigned weight $10^{-2}/3$; the same fitted transformation is applied after editing. Spectra and \ac{JCP} scores are computed from these stored semantic weights rather than coordinate-induced Euclidean weights. For a fixed edge with nonnegative weights, increasing $k$ in Equation \ref{eq:jcp-score} leaves its numerator unchanged and cannot decrease its denominator, so the raw score is non-increasing; distinct edges may decrease at different rates, so their filtration ordering and the resulting \ac{ROSA} sum need not be monotone in $k$.

Three noiseless comparisons are evaluated: a $25\%$ increase of one caliber-proxy edge, a $25\%$ increase of one branch-length edge, and both layer-specific increases together. To avoid selecting either terminal branches or ambiguous degree-four crossings, candidate edited edges are restricted to segments joining two degree-three junctions; within this class, each layer selects the largest raw attribute independently for each image. Panels A and B show representative image N10C\_L, whose selected edges are $(v_{27},v_{28})$ for the caliber proxy and $(v_{66},v_{104})$ for branch length. The filtration is bounded to 64 removals and reports raw $\mathrm{ROSA}_{\mathrm{sum}}$ for $\mathrm{JCP}_k$ at $k \in \{1,2,4,8,16,64\}$. Across the ten images, the mean caliber-proxy response is largest at $k=1$ and the mean branch-length response at $k=4$; the combined-edit response is largest at $k=1$ with a second elevated response at $k=4$. Descriptive Student-$t$ intervals report across-image variability, while leave-one-image-out inspection preserves the single-layer peak radii. Consistent with the fixed-edge property described in Section \ref{sec-jacobian}, $k$ can influence the response range for locally disjoint edits by changing how multiplex edge-attribute patterns enter the filtration; this is a construction and radius-sensitivity result, not a comparative-performance or pathology-detection claim.

Domain-specific knowledge about the expected attribute pattern is likely needed to exploit this dependence on graph construction. As expected, editing both layers causes a larger response, but the contribution is not necessarily additive.

\begin{figure}[htbp]
\centering
\includegraphics[width=\textwidth]{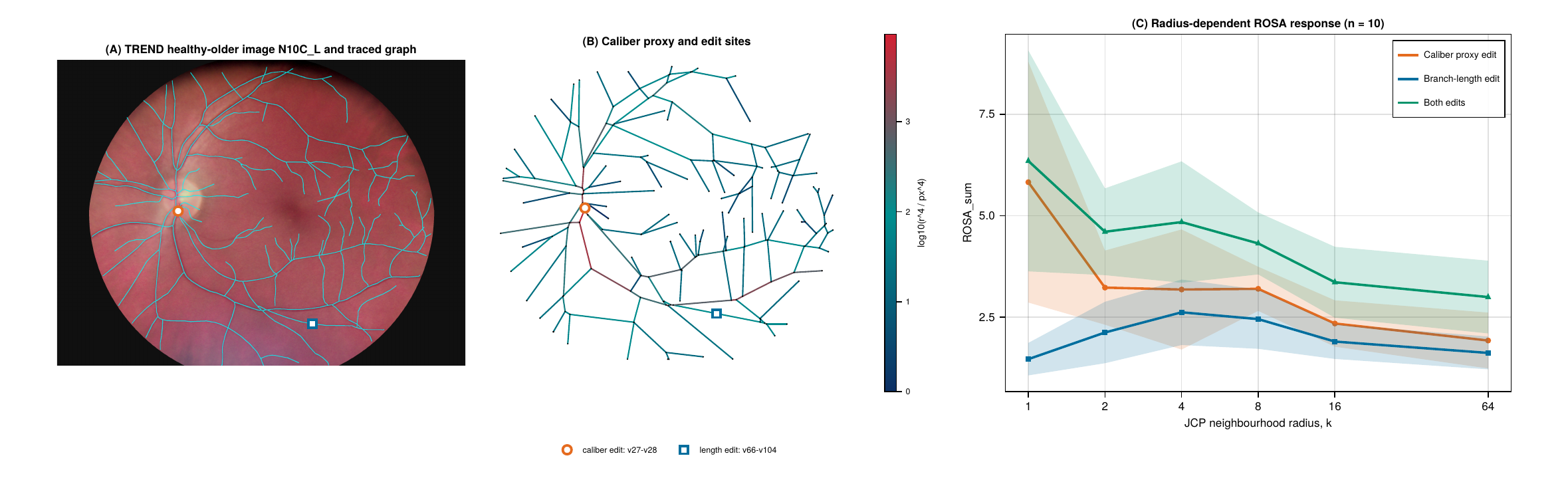}
\caption{Retinal multiplex construction and $\mathrm{JCP}_k$ sensitivity on ten TREND healthy-older fundus images. (A) Representative TREND image N10C\_L with graph traces extracted from the supplied manual vessel annotation. (B) Its extracted topology colored by $\log_{10}(r^4)$ for display of the raw caliber proxy's dynamic range; open circle and square markers identify its independently selected caliber-proxy edge $(v_{27},v_{28})$ and branch-length edge $(v_{66},v_{104})$, respectively. Computation uses the normalized raw layer attributes described in the text, not this display transform. (C) Mean raw $\mathrm{ROSA}_{\mathrm{sum}}$ with descriptive Student-$t$ 95\% intervals over ten images for noiseless $25\%$ layer-specific edits, evaluated across neighbourhood radii $k$; each image uses its own selected edge or edges. The panel evaluates the stored-edge multiplex representation and radius sensitivity; it does not compare \ac{ROSA} with a baseline distance or make a pathology-detection claim.}
\label{fig:retina-multiplex-jcp}
\end{figure}

\section{Figure Provenance Audit}\label{sec-figure-provenance}

Table \ref{tab:figure-provenance} records the paper-facing figure assets, the reproduction entry points, and the committed data artifacts used to regenerate or audit each figure.

{\scriptsize
\setlength{\tabcolsep}{3pt}
\begin{longtable}{@{}>{\raggedright\arraybackslash}p{0.12\textwidth} >{\raggedright\arraybackslash}p{0.22\textwidth} >{\raggedright\arraybackslash}p{0.25\textwidth} >{\raggedright\arraybackslash}p{0.37\textwidth}@{}}
\caption{Figure provenance and regeneration audit for the paper-facing figures. ``CSV plot'' means the plotting script reads committed CSV or TeX artifacts. ``Coupled'' means the current generator recomputes the experiment and writes the CSV during the same run.}
\label{tab:figure-provenance}\\
\hline
Figure & Paper asset(s) & Reproduction entry point & Data/regeneration artifacts \\
\hline
\endfirsthead
\hline
Figure & Paper asset(s) & Reproduction entry point & Data/regeneration artifacts \\
\hline
\endhead
\hline
\endfoot
Figure \ref{fig:heatmaps}, row 1 & \path{figures/figure_1_panel_1.pdf} & \path{reproduction/figure_1/panel_1/generate.jl}; wrapper \path{reproduction/figure_1/generate.jl} & CSV plot with \texttt{--from-csv} from \path{reproduction/figure_1/panel_1/figure_1_panel_1_raw.csv}; full coupled mode also writes that CSV. \\
Figure \ref{fig:heatmaps}, row 2 & \path{figures/figure_1_panel_2.pdf} & \path{reproduction/figure_1/panel_2/generate.jl}; wrapper \path{reproduction/figure_1/generate.jl} & CSV plot with \texttt{--from-csv} from \path{reproduction/figure_1/panel_2/figure_1_panel_2_raw.csv}; full coupled mode also writes that CSV. \\
Figure \ref{fig:graph-diversity} & \path{figures/figure_2.pdf} & \path{reproduction/figure_2/generate.jl} & CSV plot with \texttt{--from-csv} from \path{reproduction/figure_2/figure_2_raw.csv}; full coupled mode also writes that CSV. \\
Figure \ref{fig:operator-comparison} & \path{figures/figure_3.pdf} & \path{reproduction/figure_3/generate.jl} & CSV plot with \texttt{--from-csv} from \path{reproduction/figure_3/figure_3_raw.csv}; full coupled mode also writes that CSV. \\
Figure \ref{fig:scaling}, A--I & \path{figures/figure_5_panel_a.pdf}, \path{figures/figure_5_panel_b.pdf}, \path{figures/figure_5_panel_c.pdf}, \path{figures/figure_5_panel_d.pdf}, \path{figures/figure_5_panel_efg.pdf}, \path{figures/figure_5_panel_h.pdf}, \path{figures/figure_5_panel_i.pdf} & \path{reproduction/figure_5/generate.jl} plus panel scripts in \path{reproduction/figure_5/} & Mixed: CSV plot for scaling, stride, incremental, spectral-localization, and sparse-eigensolver panels using \path{fig6_scaling_raw.csv}, \path{fig6_stride_raw.csv}, \path{fig6_incremental_raw.csv}, \path{fig6_row3_data_N400_ci/*.csv}, and \path{fig6_row4_sparse_eigensolver_data/*.csv}. \\
Figure \ref{fig:fmri-boundary}, A--C & \path{figures/figure_4.pdf} & \path{reproduction/figure_4/generate.jl}; source scripts in \path{reproduction/fMRI/} & CSV plot from \path{reproduction/figure_4/fmri_heuristic_probe.csv} and \path{reproduction/figure_4/fmri_support_sweep_random_edit.csv}; source-generation scripts are \path{fmri_heuristic_probe.jl} and \path{fmri_support_sweep.jl}. \\
Figure \ref{fig:real-bio-dynamics}, rows A--D & \path{figures/figure_6.pdf} & \path{reproduction/figure_6/generate.jl}; row scripts in \path{panel_abc/} and \path{panel_de/} & Mixed: trajectory CSVs \path{trajectory_pairwise.csv}, \path{cell_pairwise.csv}; \acs{PDB} CSVs \path{pdb_single_sigma_aggregate.csv}, \path{pdb_layered_pair_summary.csv}; layered mito CSV \path{mito_k234_sweeps_summary.csv}. \\
Figure \ref{fig:sharp-boundary-setup} & \path{figures/FigS2_sharp_boundary_setup.pdf} & \path{reproduction/theorem_sharp_boundary/generate.jl} & Coupled symbolic/numeric generator; setup figure written as \path{theorem_sharp_boundary_setup.pdf}. \\
Figure \ref{fig:sharp-boundary-heatmap} & \path{figures/FigS3_sharp_boundary_heatmap.pdf} & \path{reproduction/theorem_sharp_boundary/generate.jl} & Coupled generator; writes \path{reproduction/theorem_sharp_boundary/gamma_grid.csv}. \\
Figure \ref{fig:is2-holdout} & \path{figures/FigA1_is2_holdout.pdf} & \path{reproduction/appendix_figure_1/generate.jl} & CSV plot from \path{holdout_summary.csv}, \path{holdout_bootstrap.csv}, and \path{holdout_realizations.csv}. \\
Figure \ref{fig:mst-sites} & \path{figures/FigS1_mst_sites.pdf} & \path{reproduction/figure_S1/generate.jl} & Coupled deterministic MSTA construction from \path{reproduction/common.jl}; no external CSV required. \\
Figure \ref{fig:real-bio-setup} & \path{figures/FigS4_real_bio_setup_mito.pdf}, \path{figures/FigS4_real_bio_setup_cell.pdf} & \path{reproduction/trajectory_experiment/plot_mito_graph_ids.jl} and \path{reproduction/trajectory_experiment/plot_cell_graph_ids.jl} & Real-data geometry setup plots from \path{trajectory_geometry_audit.csv}, \path{cell_geometry_audit.csv}, and graph JSON inputs in \path{reproduction/trajectory_experiment/}. \\
Figure \ref{fig:retina-multiplex-jcp} & \path{figures/FigS5_retina_multiplex_trend.pdf} & \path{reproduction/retina_multiplex/generate_figure.jl} & CSV plot from \path{trend_multiplex_rosa_summary.csv} and \path{trend_multiplex_rosa_raw.csv}; image/graph inspection artifacts in \path{reproduction/retina_multiplex/inspection/}. \\
\end{longtable}
}

\bibliography{references.bib}

\end{document}